\documentclass{amsart}

\usepackage{amsmath,amstext,amssymb,mathrsfs,amscd,amsthm,relsize,stmaryrd}
\usepackage{mathtools}
\usepackage{amsfonts}
\usepackage[all,cmtip]{xy}
\usepackage{booktabs}
\usepackage{verbatim}
\usepackage[shortlabels]{enumitem}
\usepackage[utf8]{inputenc}
\usepackage[bookmarks=false]{hyperref}
\usepackage{color}
\usepackage{graphicx}
\usepackage{subfig}
\usepackage{xspace}

\newtheorem{lemma}{Lemma}[section]
\newtheorem{proposition}[lemma]{Proposition}
\newtheorem{theorem}[lemma]{Theorem}

\newtheorem{atheorem}{Theorem}
\newtheorem{acorollary}[atheorem]{Corollary}

\newtheorem{corollary}[lemma]{Corollary}

\newtheorem{predf}[lemma]{Definition} 
\newenvironment{df}{\begin{predf}\rm}{\end{predf}}
\newtheorem{preremark}[lemma]{Remark}  
\newenvironment{remark}{\begin{preremark}\rm}{\end{preremark}}
\newtheorem{preremark0}[lemma]{Remark}  

\newtheorem{preexample}[lemma]{Example}
\newenvironment{example}{\begin{preexample}\rm}{\end{preexample}}
\newtheorem*{prenotation}{Notation}
\newenvironment{notation}{\begin{prenotation}\rm}{\end{prenotation}}
\newtheorem*{preproblem}{Problem}

\newtheorem{prequestion}[lemma]{Question}

\newtheorem*{prewarning}{Warning}
\newenvironment{warning}{\begin{prewarning}\rm}{\end{prewarning}}
\numberwithin{equation}{section}

\newcommand\cupi[1]{cup-$#1$\xspace}
\newcommand\undersumtwo[3]{\displaystyle\sum_{\substack{ #1 \\ #2}}{#3}}
\newcommand\undersumthree[4]{\sum_{\substack{ #1 \\ #2 \\ #3}}{#4}}

\newcommand\undercoprodthree[4]{\coprod_{\substack{ #1 \\ #2 \\ #3 }}{#4}}
\def\even{{\text{even}}}
\def\odd{{\text{odd}}}

\def\sq{\mathrm{Sq}}
\def\Sq{\mathfrak{sq}}
\def\ansob{an augmented semi-simplicial object in the Burnside category\xspace}
\def\sob{augmented semi-simplicial object in the Burnside category\xspace}
\def\sobs{augmented semi-simplicial objects in the Burnside category\xspace}
\def\Sobs{Augmented semi-simplicial objects in the Burnside category\xspace}
\def\osob{ordered augmented semi-simplicial object in the Burnside category\xspace}
\def\osobs{ordered augmented semi-simplicial objects in the Burnside category\xspace}

\newcommand{\smallwedge}{\circ}

\def\op{\mathrm{op}}
\def\source{\mathrm{source}}
\def\target{\mathrm{target}}
\def\Tot{\operatorname{Tot}}
\def\Setp{\mathbf{Set}_\bullet}
\def\Set{\mathbf{Set}}

\def\Kh{Kh}
\makeatletter
\renewcommand{\boxed}[1]{\text{\fboxsep=.2em\fbox{\m@th$\displaystyle#1$}}}
\makeatother
\newcommand{\cube}[1]{\mathbf{2^{#1}}}
\newcommand{\cubeop}[1]{\left(\cube{#1}\right)^{\op}}

\def\Deltainj{\Delta_{\mathrm{inj}}}
\def\Deltainjaug{\Delta_{\mathrm{inj*}}}
\def\DoldKan{Moore\xspace}
\def\DK{M}
\def\Power{P}
\newcommand\abelianisation[1]{$#1$-linearization}
\newcommand\caja[1]{\setlength{\fboxrule}{0pt}\boxed{#1}}
\newcommand\btoabfunctor[1]{\cA_{#1}}
\newcommand\btoab[2]{\cA_{#1}\left(#2\right)}
\def\igualdos{\equiv}
\def\igual{=}

\newcommand\presh[2]{{#2}^{#1^\op}}
\newcommand\preshcube[2]{{#2}^{(#1)^\op}}
\newcommand\Mod[1]{\text{$#1$-Mod}}
\def\Ch{\mathrm{Ch}}
\def\Sp{\mathbf{Sp}}

\def\APower{\mathcal{P}}
\def\udiamond{\vartriangle}
\def\bdiamond{\triangledown}
\def\diamond{\lozenge}

\def\oneU{\dot{U}}
\def\twoU{\ddot{U}}
\def\twoV{\ddot{V}}

\def\sign{\sigma}
\def\bwedge{\mathbin{\bar{\wedge}}}
\def\pwedge{\mathbin{||}}
\def\ind{\mathrm{ind}}

\def\sd{\operatorname{sd}}
\newcommand\edge[2]{\vect{e}(#1,#2)}

\newcommand{\maxstrong}[1]{h^s_{#1}}
\newcommand{\maxweak}[1]{h^w_{#1}}
\newcommand{\minstrong}[1]{\ell^s_{#1}}
\newcommand{\minweak}[1]{\ell^w_{#1}}
\newcommand{\almaxweak}[1]{m^w_{#1}}
\def\semipositive{semi-positive\xspace}
\def\seminegative{semi-negative\xspace}
\def\semiparallel{semi-parallel\xspace}
\def\W{\boldsymbol{W}}
\def\WW{\mathbb{W}}
\def\weak{\mathrm{weak}}
\def\strong{\mathrm{strong}}
\def\m{n_1}
\def\bnabla{\bar{\nabla}}

\def\cerodos{{012}}
\def\cerouno{{01}}
\def\bbpartial{\bar{\bar{\partial}}}
\def\bpartial{\bar{\partial}}
\def\bmu{\bar{\mu}}

\def\very{very\xspace}
\def\wellordered{well-ordered\xspace}

\newcommand\gen[1]{\mathbf{#1}}
\def\z{\gen{z}}
\def\y{\gen{y}}
\def\x{\gen{x}}

\newcommand{\bC}{\mathbb{C}}

\newcommand{\bF}{\mathbb{F}}

\newcommand{\bR}{\mathbb{R}}
\newcommand{\bS}{\mathbb{S}}

\newcommand{\bW}{\mathbb{W}}

\newcommand{\bZ}{\mathbb{Z}}

\newcommand{\cA}{\mathcal{A}}
\newcommand{\cB}{\mathcal{B}}
\newcommand{\cC}{\mathcal{C}}
\newcommand{\cD}{\mathcal{D}}

\newcommand{\cF}{\mathcal{F}}

\newcommand{\cO}{\mathcal{O}}

\newcommand{\cS}{\mathcal{S}}

\newcommand{\cX}{\mathcal{X}}

\newcommand\lra{\longrightarrow}
\newcommand\lla{\longleftarrow}
\newcommand\la{\leftarrow}

\newcommand{\vect}[1]{\boldsymbol{#1}}

\renewcommand{\geq}{\geqslant}
\renewcommand{\leq}{\leqslant}
\def\Id{\mathrm{Id}}
\def\map{\operatorname*{Map}}

\usepackage{mathrsfs}
\newcommand\mnote[1]{}

\title{Higher Steenrod squares for Khovanov homology}
\author{Federico Cantero Mor\'an}
\thanks{The author would like to thank the Isaac Newton Institute for Mathematical Sciences for support and hospitality during the programme Homotopy Harnessing Higher Structures and the Biblioteca Miguel Gonz\'alez Garc\'es at A Coru\~{n}a, where work on this paper was undertaken. The author was supported by EPSRC grant number EP/R014604/1, by project MTM2016-76453-C2 (AEI/FEDER, UE), and by the Spanish Ministry of Economy and Competitiveness through the Mar\'ia de Maeztu Programme for Units of Excellence in R\&D (MDM-2014-0445).}
\email{federico.j.cantero@gmail.com}
\address{Departament de Matem\`atiques i Inform\`atica, Universitat de Barcelona, Gran via de les corts catalanes, 585, 08007 Barcelona}
\begin{document}
\begin{abstract} We describe stable cup-$i$ products on the cochain complex with $\bF_2$~coefficients of any augmented semi-simplicial object in the Burnside category. An example of such an object is the Khovanov functor of Lawson, Lipshitz and Sarkar. Thus we obtain explicit formulas for cohomology operations on the Khovanov homology of any link. 
\end{abstract}
\maketitle
\section{Introduction}\label{section:intro}
In 2014 \cite{LS}, Lipshitz and Sarkar, using framed flow categories, defined a new invariant of knots and links valued in spectra that refined Khovanov homology: they associated to each link a cellular spectrum $\cX$ whose cellular cochain complex was the Khovanov complex, and so its cohomology was the Khovanov homology of the link. As a consequence, Khovanov homology became endowed with stable operations, such as Steenrod squares when cohomology is taken with coefficients in the field $\bF_2$ with two elements. 

Shortly after, Lipshitz and Sarkar \cite{LS-Steenrod} were able to give a combinatorial formula for the second Steenrod square on the Khovanov homology of any link $L$ 
\[\sq^2\colon \Kh^*(L;\bF_2)\lra  \Kh^{*+2}(L;\bF_2),\]  
in terms of the Khovanov complex and an extra datum called ladybug matching. They also showed (see also \cite{Seed}) that $\sq^2$ distinguishes some pairs of knots that are not distinguished by Khovanov homology. 

Three years later, together with Lawson \cite{LLS2015}, they gave two new constructions of Khovanov spectra that simplified the original construction of Lipshitz and Sarkar. 
 In their second construction, they associated to each link diagram $D$ a strictly unital lax $2$-functor $\cF_D$ from a cube poset to the Burnside $2$-category, and associated a realisation spectrum $|\cF_D|$ to each such $2$-functor. The spectrum $|\cF_D|$ is homotopy equivalent to the spectrum $\cX$ first constructed by Lipshitz and Sarkar.
This construction was revisited in \cite{LS-refinements} and in \cite{LLS-cube}, where they asked the following question: 

%
\begin{quotation}
\emph{Are there nice formulations of the action of the Steenrod algebra on $\Kh^*(L;\bF_2)$, purely in terms of the Khovanov functor to the Burnside category?}
\end{quotation}
%


\subsection*{Symmetric multiplications} In order to make the question concrete, we introduce the following nice formulation of Steenrod squares on a cochain complex: 
A \emph{symmetric multiplication} on a cochain complex $(C^*,d)$ of $\bF_2$-modules is a family of operations
\[\smile_i\colon C^p\otimes C^q\lra C^{p+q-i}, \quad i\in \bZ,\]
satisfying that 
\begin{align}
\label{eq:1} \alpha\smile_{i} \beta &= 0 &&\text{ for $i<0$,}\\
\label{eq:2} d(\alpha\smile_i \beta) &= d\alpha\smile_{i} \beta + \alpha\smile_{i} d\beta + \alpha\smile_{i-1} \beta + \beta\smile_{i-1} \alpha &&\text{ for all $i$.}
\end{align}
Such structure endows the cohomology groups of the cochain complex with \emph{Steenrod squares}, which are operations
\[\Sq^i\colon H^n(C^*)\lra H^{n+i}(C^*), \quad \Sq^i([\alpha]) = [\alpha\smile_{n-i} \alpha], \]
defined for $i\geq 0$. As a consequence of \eqref{eq:1}, $\Sq^i([\alpha]) = 0$ if $i>n$ and the $0$th operation $\smile_0$ gives a well-defined graded multiplication on the cohomology of $(C^*,d)$.

The prominent example of these structures appears in the normalised cochain complex $N^*(Y_\bullet;\bF_2)$ of a simplicial set $Y_\bullet$, which becomes endowed with a symmetric multiplication using the cup-$i$ product formulas of Steenrod \cite{Steenrod}. 

The normalisation process in the construction of $N^*(X_\bullet;\bF_2)$ is done in two steps: first, kill the image of the degeneracies of the simplicial set $Y_\bullet$ thus obtaining a semi-simplicial set $X_\bullet$ (a simplicial set without degeneracies), and then take the dual of the chain complex $C_*(X_\bullet;\bF_2)$ of alternating sums of face maps on the semi-simplicial set $X_\bullet$. The \cupi{i}products are defined out of the semi-simplicial structure and involve only face maps, so the cochain complex of any semi-simplicial set is also enhanced with \cupi{i}products.

As an example, here are formulas for $\smile_0$ and $\Sq^1$. If $\alpha\in C^p(X_\bullet;\bF_2)$ and $\beta\in C^q(X_\bullet;\bF_2)$, then $\alpha\smile_0\beta\in C^{p+q}(X_\bullet;\bF_2)$ is the Alexander--Whitney product of $\alpha$ and $\beta$, whose value on a chain $\sigma\in C_{p+q}(X_\bullet;\bF_2)$ is
\[(\alpha\smile_0 \beta)(\sigma) = \alpha\left(\partial_{p+1}\cdots\partial_{p+1}\sigma\right)\cdot \beta\left(\partial_{0}\cdots \partial_0 \sigma\right).\]
If $\alpha\in C^{n}(X_\bullet;\bF_2)$ is a cocycle, then the first Steenrod square $\Sq^1([\alpha])$ of $[\alpha]$ can be computed as $[\alpha\smile_{n-1} \alpha]$, which is defined as 
\[(\alpha\smile_{n-1} \alpha)(\sigma) = \undersumtwo{j<k}{j,k\text{ even}}{\alpha(\partial_{j}\sigma)\cdot \alpha(\partial_k\sigma)} + \undersumtwo{j>k}{j,k\text{ odd}}{\alpha(\partial_{j}\sigma)\cdot \alpha(\partial_k\sigma)}.\]

On the other hand, Steenrod squares on a topological space $X$ can be defined as natural transformations $\sq^i\colon  H^*(X;\bF_2)\to H^{*+i}(X;\bF_2)$ that satisfy certain axioms. The fact that the Steenrod squares $\Sq^i$ for a simplicial set $X_\bullet$ that arise from the cup-$i$ products coincide with the axiomatic Steenrod squares $\sq^i$ for the topological space $|X_\bullet|$ is not immediate, and uses the singular chain functor of Eilenberg \cite{eilenberg1944} to compare the cohomology operations in both settings.

Observe also that the simplicial structure is crucial to define the symmetric multiplication. In contrast, the cellular cochain complex of a CW-complex does not have in general a symmetric multiplication. 

\subsection*{Stable symmetric multiplications} Condition \eqref{eq:1} above implies that $\smile_0$ is a well-defined product on cohomology, but is not necessary for the definition of the Steenrod squares. A \emph{stable symmetric multiplication} on a cochain complex $C^*$ of $\bF_2$-modules is a family of operations
\[\smile_i\colon C^p\otimes C^q\lra C^{p+q-i}, \quad i\in \bZ,\]
satisfying \eqref{eq:2}. Such a structure gives again operations
\[\Sq^i\colon H^n(C^*)\lra H^{n+i}(C^*), \quad \sq^i([\alpha]) = [\alpha\smile_{n-i} \alpha] \]
defined for $i\geq 0$. 

In Section \ref{ssection:Khovanov-intro} we explain how to associate to the Khovanov functor $\cF_D$ of Lawson, Lipshitz and Sarkar \ansob $X_\bullet$ whose cochain complex $C^*(X;\bF_2)$ is an iterated suspension of the Khovanov complex. These objects are defined in Section \ref{section:techniques}, and are generalisations of augmented semi-simplicial sets, where face maps $\partial_i\colon X_n\to X_{n-1}$ are replaced by zig-zags $X_{n}\la Q^n_i\to X_{n-1}$ plus some higher categorical data. When each map $X_n\la Q^n_i$ is a bijection one recovers the concept of augmented semi-simplicial set, and if additionally the augmentation is trivial, one recovers the concept of semi-simplicial set. An order on a \sob is a choice of order on each $Q^n_i$ and on the higher categorical data. 


\begin{atheorem} The cochain complex of an \osob has a natural stable symmetric multiplication, i.e., there are explicit operations
\[\smile_i\colon C^p(X_\bullet;\bF_2)\otimes C^q(X_\bullet;\bF_2)\lra C^{p+q-i}(X_\bullet;\bF_2), \quad i\in \bZ\]
satisfying \eqref{eq:2}, and, for every free order-preserving map $f\colon X_\bullet\to Y_\bullet$, 
\[f^*(\alpha\smile_i \beta) = f^*(\alpha)\smile_i f^*(\beta).\]
If $X_\bullet$ is a semi-simplicial set, then these operations are the Steenrod cup-$i$ products.
\end{atheorem}
The explicit formulas for the $\smile_i$ products are given in Section \ref{section:theorem} 
 and the naturality is proven in Proposition \ref{prop:naturality-cup}. In Section \ref{section:well-defined} we prove that the Steenrod squares induced by the symmetric multiplication are invariant under suspension, satisfy a Cartan formula and the first square $\Sq^1$ is the Bockstein homomorphism. In Theorem \ref{thm:naturality-Steenrod} we improve their naturality properties obtaining the following result.
\begin{atheorem} The Steenrod operations $\Sq^i\colon H^*(X_\bullet;\bF_2)\to H^{*+i}(X_\bullet;\bF_2)$ associated to the stable symmetric multiplication are natural with respect to maps of \sobs. In particular, they do not depend on the chosen order on $X_\bullet$.
\end{atheorem} 
In Corollary \ref{cor:citable} we obtain the following consequence for Khovanov homology.
\begin{acorollary}\label{cor:Khovanov} There is an explicit stable symmetric multiplication on the Khovanov complex of any oriented link diagram $D$ with $n_-$ negative crossings
\[\smile_i\colon C^{p-n_-+1}(D;\bF_2)\otimes C^{q-n_-+1}(D;\bF_2)\lra C^{p+q-i-n_-+1}(D;\bF_2)\quad i\in \bZ.\]
Therefore Khovanov homology becomes endowed with the Steenrod squares
\[\Sq^i\colon \Kh^{n}(D;\bF_2)\lra \Kh^{n+i}(D;\bF_2), \quad \Sq^i([\alpha]) = [\alpha\smile_{n+n_--1-i } \alpha]\]
associated to this stable symmetric multiplication, which are invariant under Reidemeister moves and reordering of the crossings.
\end{acorollary}

The techniques of this paper do not allow us to prove that the Steenrod squares of Corollary \ref{cor:Khovanov} coincide with the Steenrod squares 
\[\sq^i\colon H^n(|\cF_D|;\bF_2)\lra  H^{n+i}(|\cF_D|;\bF_2),\quad i\geq 0\]
of the realisation spectrum of $\cF_D$. Such comparison will be developed in the companion paper \cite{Cantero-Gutierrez}, where we will construct a ``singular chain functor'' from the category of spectra to the category of \sobs. The results in the present paper are purely combinatorial, and have to be compared with the constructions of Steenrod in \cite{Steenrod} for simplicial complexes, whereas the results of the companion paper \cite{Cantero-Gutierrez} are mainly homotopy-theoretic and have to be compared with the constructions of Eilenberg \cite{eilenberg1944} for topological spaces. In particular, spectra will be essentially absent from this paper, and will only be barely mentioned in some examples in Section \ref{section:examples}.


\subsection*{Outline of the paper} In Section \ref{section:techniques}, we first explain how to translate the framework of \cite{LLS2015}, which is expressed in terms of cubes in the Burnside category, to our framework in terms of augmented semi-simplicial objects in the Burnside category. Then we introduce several definitions and constructions that will be used through the paper. In Section \ref{section:theorem}, we present formulas for \cupi{i}products, and we prove in Section \ref{section:proof} that they endow the cohomology of any \sob with a stable symmetric multiplication. In Section \ref{section:well-defined} we define Steenrod squares and we prove that they are stable under suspension, that they satisfy a Cartan formula and that the first square is the Bockstein homomorphism. The proofs of naturality are deferred to Section \ref{section:functoriality}. In Section \ref{section:khovanov} we apply the previous results to the Khovanov functor of Lawson, Lipshitz and Sarkar and we prove Corollary \ref{cor:Khovanov}. The paper finishes with several examples in Section \ref{section:examples}. The reader only interested on explicit formulas for operations on Khovanov homology will find them in Section \ref{section:theorem} after having got used to the terminology introduced in Section \ref{section:techniques}, and may afterward safely skip Sections \ref{section:proof}, \ref{section:well-defined} and \ref{section:functoriality} and proceed directly to Section \ref{ssection:khovanov-khovanov} and the examples in Section \ref{section:examples}. 


\subsection*{Acknowledgments}The author is especially grateful to An\'ibal Medina-Mar\-do\-nes for the inspiration received while reading his paper \cite{Anibal}. He is also grateful to Fernando Muro and David Chataur, and to Javier Guti\'errez, Carles Casacuberta, Joana Cirici and Marithania Silvero from the Topology group at Barcelona. He thanks Tyler Lawson, Clemens Berger, 
 and Oscar Randal-Williams for their feedback during his stay at the Isaac Newton Institute for Mathematical Sciences.
\section{Khovanov functors and semi-simplicial objects in the Burnside category}\label{section:techniques}
This section begins with a quick explanation of how to translate the context in \cite{LLS2015,LLS-cube,LS-refinements} (cubes in the Burnside category) to our context (augmented semi-simplicial objects in the Burnside category). Sections \ref{ssection:sequences} to \ref{ssection:realisations} are devoted to present and prove the concepts and claims used in this explanation, while Sections \ref{ssection:sob} and \ref{ssection:sobmaps} will set up the notation for augmented semi-simplicial objects in the Burnside category. In this exposition, $R$ denotes a commutative ring with unit.


\subsection{Khovanov spectra}\label{ssection:Khovanov-intro} In \cite{Khovanov,Bar-Natan}, Khovanov associated to each link diagram $D$ with $c$ ordered crossings and $n_-$ negative crossings, a contravariant functor $F_D$ from the cube category $\cube{c}$ to the category of $R$-modules. Every such functor has a totalisation $\Tot F_D\in \Ch(R)$, which is a bounded chain complex of $R$-modules. Khovanov proved that if two link diagrams $D$ and $D'$ are obtained from each other using some Reidemeister move, then the chain complexes $\Sigma^{-n_-}\Tot F_D$ and $\Sigma^{-n_-'}\Tot F_{D'}$ are quasi-isomorphic. The Khovanov homology of a link $L$ is then defined as the cohomology of $\Sigma^{-n_-}\Tot F_D$, for any diagram $D$ representing $L$, and it is a link invariant up to isomorphism.

The construction of the Khovanov spectrum of Lawson, Lipshitz and Sarkar \cite{LLS2015, LLS-cube} first associates to a link diagram $D$ with $c$ ordered crossings a lax strictly unital contravariant functor $\cF_D\colon \cube{c}\to \cB$ from the cube category $\cube{c}$ to the Burnside $2$-category; and to a Reidemeister move from $D$ to a diagram $D'$ a ``stable equivalence'' between $\cF_D$ and $\cF_{D'}$.
 They showed that postcomposing with a certain functor $\btoabfunctor{R}\colon \cB\to \Mod{R}$, one obtains back the Khovanov construction: $F_D = \btoabfunctor{R}\circ \cF_D$. Moreover, they constructed a realisation functor $|\cdot|$ that converts every contravariant strictly unital lax functor $F\colon \cube{c}\to \cB$ into a spectrum $|F|$ and every strictly unital lax natural transformation into a map of spectra, with the property that 
\[H_*(\Tot \btoabfunctor{R}\circ \cF_D)\cong H_*(|\cF_D|;R).\]
Finally, they proved that the stable equivalences
 associated to the Reidemeister moves are converted into weak equivalences of spectra, thus concluding that the stable homotopy type of $\Sigma^{-n_-}|\cF_D|$ is a link invariant.

An augmented semi-simplicial object in a $2$-category $\cC$ is a strictly unital lax functor from the category of possibly empty finite ordinals $\Deltainjaug$ to $\cC$. There is a cofinal functor $\cube{c}\to \Deltainjaug$, and, if $\cC$ has finite coproducts, taking left Kan extension along it defines a functor
\[\Lambda\colon \preshcube{\cube{c}}{\cC}\to \presh{\Deltainjaug}{\cC}\]
between categories of strictly unital lax functors. When $\cC$ is the category of $R$-modules, we let  
\[\DK\colon \presh{\Deltainjaug}{\Mod{R}}\lra \Ch(R)\] 
be the functor that takes a semi-simplicial $R$-module to its Moore chain complex\mnote{Moore o Dold-Kan} of alternating sums of face maps. Then, the upper part of the following diagram commutes: 
\begin{equation}
\label{eq:picture}\begin{aligned}\xymatrix{
\preshcube{\cube{c}}{\cB}\ar[d]^-\Lambda\ar[r]^-{\btoabfunctor{R}\circ -} & \preshcube{\cube{c}}{(\Mod{R})}\ar[d]^-\Lambda \ar[r]^-{\Tot}  &  \Ch(R)\ar[d]^-{\Sigma^{-1}}\\
\presh{\Deltainjaug}{\cB}\ar[d]^-{*}\ar[r]^-{\btoabfunctor{R}\circ-} & \presh{\Deltainjaug}{(\Mod{R})}\ar[r]^-{\DK} &\Ch(R)\ar[d]^-{H_*} \\
\mathbf{Spectra}\ar[rr]^-{H_*(-;R)}&& \text{Graded \Mod{R}}
}\end{aligned}\end{equation}
In this paper we are only interested in the upper part of the diagram. We give the following comment without proof regarding the bottom part: The functor $*$ is constructed in \cite{Barwick} and the composition of $\Lambda$ and that functor is homotopy equivalent in an $\infty$-categorical sense to the desuspension of the realisation construction $|\cdot|$ of Lawson, Lipshitz and Sarkar. Additionally, the bottom square commutes.

Let $X_\bullet\in \presh{\Deltainjaug}{\cB}$, and let $C^*(X_\bullet;\bF_2)$ be the dual of the Moore chain complex $\DK(\btoabfunctor{\bF_2}\circ X_\bullet)$. We will construct natural operations 
\[\smile_i\colon C^p(X_\bullet;\bF_2)\otimes C^q(X_\bullet;\bF_2)\to C^{p+q-i}(X_\bullet;\bF_2)\] analogous to the classical \cupi{i}products for semi-simplicial sets. If $\cF_D\in \preshcube{\cube{c}}{\cB}$ is the functor of Lipshitz, Lawson and Sarkar, then, since the above diagram commutes, we will have operations $\smile_i$ defined on 
$$\DK\circ\btoabfunctor{R}\circ \Lambda(\cF_D) = \Sigma^{-1}\Tot\circ \btoabfunctor{R}(\cF_D) = \Sigma^{-1}\Tot(F_D),$$ the $(n_--1)st$ suspension of the  Khovanov complex. Thus, we will obtain, for any link $L$, cohomology operations
\[\Sq^i\colon \Kh^*(L;\bF_2)\lra \Kh^{*+i}(L;\bF_2),\quad i\geq 0.\]


\subsection{Sequences}\label{ssection:sequences} Let $\Power_q(n)$ be the set of all increasing sequences $U=(u_1,\ldots,u_q)$ such that $0\leq u_i\leq n$ for each $i=1,\ldots,q$, and let $\Power(n) = \bigcup_q \Power_q(n)$. Suppose that $V\in \Power_p(n)$ and define 
\begin{align*}
\psi_{V}\colon \{U\in \Power_{q}(n)\mid U\cap V=\emptyset\}&\lra \Power_q(n-p)& \\ 
(u_1,\ldots,u_q) &\longmapsto (w_1,\ldots,w_q), &w_j = u_j-|\{v\in V\mid v<u_j\}|,\\[.2cm]
\gamma_V\colon \Power_q(n) &\lra \Power_{q}(n+p)& \\
(u_1,\ldots,u_q)&\lra (w_1,\ldots,w_q),&w_j = u_j+|\{v\in V\mid v\leq u_j\}|,\\[.2cm]
\eta_V\colon \{U\in \Power_q(n)\mid V\subset U\}&\lra \Power_{q-p}(n-p)& \\
U&\longmapsto \psi_{V}(U\smallsetminus V),&\\[.2cm]
\xi_V\colon \Power_q(n) &\lra \Power_{q+p}(n+p) &\\
U&\lra \gamma_V(U)\cup V.&
\end{align*}
We have that
\begin{align*}
\psi_{V}(\gamma_V(U))&=U,& \gamma_{V}(\psi_V(U))&=U, \\
\eta_{V}(\xi_V(U))&=U,& \xi_{V}(\eta_V(U))&=U.
\end{align*}

\subsection{The augmented semi-simplicial category} Let $\Deltainj$ be the category of non-empty finite ordinals and order-preserving injections between them, and let $\Deltainjaug$ be the category of finite ordinals and order-preserving maps between them. We write $[n] = \{0,\ldots,n\}$ for the ordinal with $n+1$ elements and note that $[-1]$ stands for the empty ordinal. For each sequence $U\in \Power_q(n)$, define the \emph{$U$th generalised face map} $\partial^n_U\colon [n-q]\to [n]$ as the unique order-preserving injective map that misses $U$. Every morphism in $\Deltainjaug$ is a generalised face map, and, if $U\in \Power(n)$ and $U=V\cup W$, with $V\in \Power_q(n)$ then
\begin{equation}\label{eq:simplicial}\partial^{n-q}_{\psi_V(W)}\circ \partial^{n}_{V} = \partial^n_{U}.\end{equation}
If $V\in \Power_q(n)$ and $W\in \Power_p(n-q)$, we can rewrite this condition as:
\[\partial^{n-q}_{W}\circ \partial^n_V = \partial^n_{\xi_V(W)}.\]
When $U=\{u\}$ has a single element, $\partial_U$ is called a \emph{face map} and the equation \eqref{eq:simplicial} becomes the usual simplicial identities of the face maps. In that case, we write $\partial_u$ instead of $\partial_{\{u\}}$. Every generalised face map is a composition of face maps. 

%
%

\subsection{The cube poset}\label{ssection:cubes} Let $0<1$ be the poset with two elements $0,1$ and one morphism from $0$ to $1$. Let $\cube{n}$ be the $n$th power of this poset, whose elements are tuples $v=(v_1,\ldots,v_n)$ with $v_i\in \{0,1\}$ and $v\leq v'$ if and only if $v_i\leq v'_i$ for all $i=1\ldots n$. The elements of the cube are graded by the Manhattan norm $|v| = \sum_i v_i$. Alternatively, $\cube{n}$ is the poset of subsets of $\{1,\ldots,n\}$ ordered by inclusion. We will denote a subset $A$ of $\{1,\ldots,n\}$ as the ordered sequence $(a_1,\ldots,a_m)$ on the elements of $A$. The grading of a subset $A\subset \{1,\ldots,n\}$ is its cardinality $|A|$. The relation between both perspectives is given by identifying a tuple $(v_1,\ldots,v_n)$ with the subset $A=\{i\mid v_i=1\}\subset \{1,\ldots,n\}$. In this paper we will use the second description of $\cube{n}$.
 

\subsection{2-categories} In this section we quickly remind the definitions of $2$-category, of strictly unital lax functor and of strictly unital lax natural transformation \cite[\S 7.5]{Borceaux}. As the $2$-morphisms in the $2$-categories used in this paper are invertible, it is customary to replace the adjective ``lax'' by ``pseudo-natural''. Moreover, it will be convenient to work with op-lax functors and op-lax natural transformations, i.e., our structural $2$-morphisms in a pseudo-functor or a pseudo-natural transformation are the inverses of the usual $2$-morphisms. In Sections \ref{ssection:sob} and \ref{ssection:sobmaps} we give a more careful definition of the particular functors and natural transformations that are used in the paper.

Recall that a $2$-category $\cC$ is a category enriched in the category of small categories, i.e., it consists on the data of
\begin{enumerate}
\item a collection of objects,
\item for each pair of objects $X,Y$, a category of morphisms $\hom(X,Y)$,
\item for each object $X$, an object $\Id_X$ of $\hom(X,X)$ and
\item for each triple of objects $X,Y,Z$, a composition functor $$\hom(Y,Z)\times \hom(X,Y)\to \hom(X,Z)$$
\end{enumerate}
satisfying that the composition is associative and unital. The objects of $\hom(X,Y)$ are called \emph{$1$-morphisms} and the morphisms in $\hom(X,Y)$ are called \emph{$2$-morphisms}.

Recall that a strictly unital pseudo-functor $F$ from a category $\cD$ to a $2$-category $\cC$ with invertible $2$-morphisms consists of the data:
\begin{enumerate}
\item for each object $a$ of $\cD$, an object $F(a)$ in $\cC$,
\item for each morphism $f$ of $\cD$, a morphism $F(f)$ in $\cC$,
\item for each decomposition $f=g_2\circ g_1$, a $2$-morphism $F(g_1,g_2)$ from $F(f)$ to $F(g_2)\circ F(g_1)$
\end{enumerate}
satisfying that
\begin{enumerate}[resume]
\item for each object $a$ of $\cD$, $F(\Id_a) = \Id_{F(a)}$,
\item for each morphism $f\colon a\to b$ of $\cD$, both $\mu_{\Id_a,f}$ and $\mu_{f,\Id_b}$ are the identity $2$-morphisms.
\item for each decomposition $f=g_3\circ g_2\circ g_1$, we have
\[(\Id\times F(g_2,g_3))\circ F(g_1,g_3\circ g_2) = (F(g_1,g_2)\times \Id)\circ F(g_2\circ g_1,g_3)\]
\end{enumerate}
Recall that a strictly unital pseudo-natural transformation between two $2$-functors $F,G\colon \cD\to \cC$ consists of the data:
\begin{enumerate}
\item for each object $a$ of $\cD$, a morphism $\alpha_a\colon F(a)\to G(a)$ and
\item for each morphism $f\colon a\to b$ of $\cD$, a $2$-morphism $\alpha_f$ from $\alpha_b\circ F(f)$ to $G(f)\circ \alpha_a$,
\end{enumerate}
satisfying that
\begin{enumerate}[resume]
\item for each object $a$ of $\cD$, $\alpha_{\Id_a}$ is the identity $2$-morphism on $\alpha_a\colon F(a)\to G(a)$ and
\item for each decomposition $f=g_2\circ g_1$, we have:
\[G(g_1,g_2)\circ \alpha_f = \alpha_{g_2}\circ\alpha_{g_1}\circ F(g_1,g_2)\]
\end{enumerate}
Let $\presh{\cD}{\cC}$ be the category whose objects are contravariant strictly unital pseudo-functors from $\cD$ to $\cC$ and whose morphisms are strictly unital pseudo-natural transformations between them.

\subsection{The Burnside $2$-category}\label{ssection:burnside}   Given two sets $X,Y$, a \emph{locally finite span} from $X$ to $Y$ is a pair of functions 
\[\xymatrixcolsep{36pt}\xymatrix{X & Q\ar[l]_-{\source}\ar[r]^-{\target} & Y}.\]
such that $\source^{-1}(x)$ is a finite set for every $x\in X$. A locally finite span is \emph{free} if the source map is an injection. A \emph{fibrewise bijection} between two locally finite spans $X\overset{\source}{\longleftarrow} Q\overset{\target}{\longrightarrow} Y$ and $X\overset{\source'}{\longleftarrow} Q'\overset{\target'}{\longrightarrow} Y$ is a bijection $\tau\colon Q\to Q'$ such that $\source'\circ\tau=\source$ and $\target'\circ\tau=\target$:
\[\xymatrixcolsep{32pt}\xymatrixrowsep{10pt}\xymatrix{
&Q\ar[ld]_{\source}\ar[rd]^\target\ar[dd]^{\tau}&\\
X && Y \\
&Q'\ar[lu]^{\source'}\ar[ru]_{\target'}&
}\]
The composition of two fibrewise bijections of locally finite spans is the composition of bijections:
\[\xymatrixcolsep{32pt}\xymatrixrowsep{10pt}\xymatrix{
& Q_1\ar[d]^{\tau_1}\ar[dr]\ar[dl] & \\
X&Q_2\ar[d]^{\tau_2}\ar[r]\ar[l] & Y.\\
&Q_3\ar[ur]\ar[ul]&
}\qquad
\xymatrix{
& Q_1\ar[dd]^{\tau_2\circ\tau_1}\ar[dr]\ar[dl] & \\
X& & Y,\\
&Q_3\ar[ur]\ar[ul]&
}
\]
and the identity morphism of a locally finite span is the identity bijection. This defines a category of locally finite spans from a set $X$ to a set $Y$.

\begin{df} We denote by $\cB$ the Burnside $2$-category for the trivial group, whose objects are sets, and the category of morphisms from a set $X$ to a set $Y$ is the category of locally finite spans from $X$ to $Y$. Composition of locally finite spans is given by taking their fibre product:
\[\xymatrixcolsep{16pt}\xymatrixrowsep{6pt}\xymatrix{
&& Q_1\times_{X_2} Q_2\ar@{-->}[dl] \ar@{-->}[dr] && \\
&Q_1\ar[dl]\ar[dr] && Q_2\ar[dl]\ar[dr] & \\
X_1 && X_2 && X_3,
}\]
composition of fibrewise bijections is their fibre product
\[\xymatrixcolsep{24pt}\xymatrixrowsep{8pt}\xymatrix{
&Q_1\ar[dl]\ar[dr]\ar[dd]^{\tau_1} && Q_2\ar[dl]\ar[dr]\ar[dd]^{\tau_2} & \\
X_1 && X_2 && X_3 \\
&Q_3\ar[ul]\ar[ur] && Q_4\ar[ul]\ar[ur] & \\
}\qquad
\xymatrix{
& Q_1\times_{X_2} Q_2\ar[dd]^{\tau_1\times\tau_2}\ar[dr]\ar[dl] & \\
X_1& & X_3.\\
&Q_3\times_{X_2} Q_4,\ar[ur]\ar[ul]&
}\]
and the identity on a set $X$ is the span $X\overset{\Id}{\la} X\overset{\Id}{\to} X$.
\end{df}
\begin{warning}\label{warning-burnside} In \cite{LLS2015} and \cite{LLS-cube} the objects of the Burnside category are required to be finite sets. This restriction is unnecessary for the results in this paper. 
\end{warning}

The Burnside $2$-category has a coproduct
\[\amalg\colon \cB\times \cB\lra \cB\]
that sends a pair of finite sets to their disjoint union and a pair of spans $X\la Q\to Y$ and $X'\la Q'\to Y'$ to the span
\[X\amalg X'\lla Q\amalg Q'\lra Y\amalg Y'.\]
The categorical product in $\cB$ coincides with the coproduct and there is a tensor product
\[\times\colon \cB\times\cB\lra \cB\]
that sends a pair of finite sets to their product and a pair of spans $X\la Q\to Y$ and $X'\la Q'\to Y'$ to the product span
\[X\times X'\lla Q\times Q'\lra Y\times Y'.\]
The category of pointed sets $\Setp$ includes into the Burnside $2$-category $\cB$ by sending a pointed set $(X,x_0)$ to $X\smallsetminus \{x_0\}$ and a morphism $f\colon (X,x_0)\to (Y,y_0)$ to the span 
$$X\smallsetminus \{x_0\}\hookleftarrow X\smallsetminus f^{-1}(y_0) \overset{f}{\to} Y\smallsetminus \{y_0\}.$$
This inclusion induces an equivalence of categories between the category $\Setp$ and the $2$-subcategory $\cB_{\mathrm{free}}$ of $\cB$ whose objects are sets and whose $1$-morphisms are free spans. Observe that there is at most one $2$-morphism between any two free spans.

The category of sets $\Set$ includes into the category $\Setp$ by sending a set to its disjoint union with some fixed basepoint. The composition of these two inclusions induces an equivalence of categories between $\Set$ and the $2$-subcategory of $\cB$ whose objects are sets and whose $1$-morphisms are spans for which the source map is a bijection. 

\subsection{\boldmath From the Burnside category to the category of $R$-modules}
Define the \emph{\abelianisation{R}} functor $$\btoabfunctor{R}\colon \cB\lra \Mod{R}$$ that takes a set to the free $R$-module on it, and a span $\xymatrixcolsep{30pt}\xymatrix{X & Q\ar[l]_-{\source}\ar[r]^-{\target} & Y}$ to the homomorphism $R\langle X\rangle\to R\langle Y\rangle$ whose value on $x\in X$ is 
\[\sum_{y\in Y} \left|\source^{-1}(x)\cap \target^{-1}(y)\right| \cdot y.\]
Any two spans connected by a $2$-morphism are sent to the same homomorphism and the sum is well-defined because $\source^{-1}(x)$ is finite, so $\btoabfunctor{R}$ is well-defined. Moreover, this functor is symmetric monoidal with respect to the tensor products $\amalg,\times$ in $\cB$ and the tensor products $\oplus,\otimes$ in $\Mod{R}$, so:
\[\btoab{R}{X\amalg Y}\cong \btoab{R}{X}\oplus \btoab{R}{Y},\quad \btoab{R}{X\times Y}\cong \btoab{R}{X}\otimes \btoab{R}{Y}.\]
Let $f,g\colon X\to Y$ be the spans $X\la Q\to Y$ and $X\la Q'\to Y$. We denote by $f+g$ the span $X\la Q\amalg Q'\to Y$. This operation commutes with \abelianisation{R} too: 
\[\btoab{R}{f+g} = \btoab{R}{f}+\btoab{R}{g}.\]  
\begin{df} Two spans $X\la Q\to Y$ and $X\la Q'\to Y$ are \emph{equivalent} if there is a $2$-morphism between them. Alternatively, they are equivalent if they have the same \abelianisation{\bZ}. They are \emph{$\bF_2$-equivalent} if they have the same \abelianisation{\bF_2}.
\end{df}
\begin{notation}
We use the symbol ``$\igual$'' to denote the relation of being equivalent, and the symbol ``$\igualdos$'' for the relation of being $\bF_2$-equivalent. 
\end{notation}

\subsection{\boldmath The functor $\Lambda$}\label{ssection:lambda} Let $\cC$ be a $2$-category with invertible $2$-morphisms. An \emph{$n$-cube in $\cC$} is a strictly unital pseudo-functor $F\colon \cube{n}\lra \cC$, a \emph{semi-simplicial object in $\cC$} is a strictly unital pseudo-functor $X\colon \Deltainj^\op\to \cC$ and an \emph{augmented semi-simplicial object in $\cC$} is a strictly unital pseudo-functor $X\colon \Deltainjaug^\op\to \cC$. As customary, if $X$ is an augmented semi-simplicial object in a category $\cC$, we will write $X_\bullet$ for $X$, $X_n$ for $X(n)$ and $\partial^n_U$ for $X(\partial^n_U)$.

A \emph{map} between two $n$-cubes in $\cC$ is a strictly unital pseudo-natural transformation between them. A \emph{map} between two \sobs in $\cC$ is a strictly unital pseudo-natural transformation between them.

If $A=(a_1,\ldots,a_m)$ is an element of $\cube{n}$, let $\varphi_A\colon A\to \{0,\ldots,m-1\}$ be the function $\varphi_A(a_i)=i-1$. There is a functor $\lambda\colon \cube{n}\to \Deltainjaug$ that sends a vertex $A\subset \{1,\ldots,n\}$ of $\cube{n}$ to the ordinal $|A|-1$ and a morphism $B\subset A$ to the morphism $\partial^{|A|-1}_{\varphi_A(A\smallsetminus B)}$. If $\cC$ has finite coproducts, left Kan extension along this functor, defines a functor
\[\Lambda\colon \preshcube{\cube{n}}{\cC}\lra \presh{\Deltainjaug}{\cC}\]
given explicitly on an object $F\colon \cubeop{n}\to \cC$ as
\begin{align*}
\Lambda(F)(k) &= \coprod_{|A|=k+1} F(A) &k\geq -1,\\
\Lambda(F)(\partial^k_{U}) &= \coprod_{|A|=k+1} F\left(A\smallsetminus \varphi_A^{-1}(U)\subset A\right) & U\in \Power_q(n),\\
\Lambda(F)(\partial^n_{V},\partial^{n-q}_{W}) &= \coprod_{|A|=k+1} F(B\subset A,C\subset B) & V\in \Power_q(n), W\in \Power(n-q),
\end{align*}
where $B=A\smallsetminus \varphi_A(V)$ and $C=B\smallsetminus \varphi_B(W)$.
A similar formula holds for maps between $n$-cubes in $\bC$. Setting $\cC$ to be either the Burnside $2$-category or the category of $R$-modules, seen as a $2$-category with only identity $2$-morphisms, we obtain the upper left square in \eqref{eq:picture}:
\[\xymatrix{
\preshcube{\cube{c}}{\cB}\ar[d]^\Lambda\ar[r]^-{\btoabfunctor{R}\circ -} & \preshcube{\cube{c}}{(\Mod{R})}\ar[d]^\Lambda  \\
\presh{\Deltainjaug}{\cB}\ar[r]^-{\btoabfunctor{R}\circ-} & \presh{\Deltainjaug}{(\Mod{R})},
}\]
which commutes because $\btoabfunctor{R}\colon \cB\to \Mod{R}$ preserves finite coproducts.


%
%

\subsection{Realisations and totalisations}\label{ssection:realisations} 

Define the \emph{\DoldKan functor} $$\DK\colon \presh{\Deltainjaug}{\Mod{R}}\lra\Ch(R)$$ as the functor that sends an augmented semi-simplicial $R$-module $X_\bullet$ to the following chain complex $C_*(X_\bullet)$:
\[C_n(X_\bullet) = X_n,\quad d_n = \sum_{i=0}^n (-1)^{i}{\partial_i^n},\]
and sends a map $f\colon X_\bullet \to Y_\bullet$ to the homomorphism of chain complexes which in degree $n$ is
\[f_n\colon C_n(X_\bullet)\lra C_n(Y_\bullet).\] 
Define the \emph{totalisation} functor $\Tot\colon \preshcube{\cube{n}}{\Mod{R}}\to \Ch(R)$ as the functor that sends a cube of $R$-modules $F$ to the following chain complex $C_*(F)$:
\[C_m(F) = \bigoplus_{|A|=m} F(A),\quad d_m = \undersumtwo{|A|=m}{A=(a_1,\ldots,a_m)}{\sum_{i=1}^{m} (-1)^{i}F(A\smallsetminus \{a_{i}\}\subset A)}\]
and sends a map $f\colon F\to G$ to the homomorphism of chain complexes which in degree $m$ is
\[ \bigoplus_{|A|=m} f_A\colon \bigoplus_{|A|=m} F(A)\lra  \bigoplus_{|A|=m} G(A).\]
When $X_\bullet = \Lambda(F)$, the Moore chain complex of $X_\bullet$ and the totalisation of $F$ agree up to sus\-pen\-sion: $C_*(\Lambda(F)) = \Sigma^{-1}C_*(F)$, so the upper right square in \eqref{eq:picture} commutes. 

\begin{df} If $R$ is a commutative ring, the \emph{$R$-realisation functor} is the composition
\[|\cdot|_R\colon \presh{\Deltainjaug}{\cB}\overset{\btoabfunctor{R}\circ -}{\lra} \presh{\Deltainjaug}{\Mod{R}}\overset{\DK}{\lra} \Ch(R).\]
Explicitly, the $R$-realisation $|X_\bullet|_R$ of an augmented semi-simplicial object in the Burnside category $X_\bullet$ is the chain complex
\[C_n(X_\bullet;R) = R\langle X_n\rangle,\quad d_n = \sum_{i=0}^n (-1)^{i}\btoab{R}{\partial_i^n}.\]
The \emph{$R$-totalisation functor} is the composition
\[\Tot_R\colon \preshcube{\cube{n}}{\cB}\overset{\btoabfunctor{R}\circ -}{\lra} \preshcube{\cube{n}}{\Mod{R}}\overset{\Tot}{\lra} \Ch(R).\]
Explicitly, the $R$-totalisation $\Tot_R(F)$ of a cube in the Burnside category $F\colon \cube{n}\to \cB$ is the chain complex
\[C_m(F;R) = \bigoplus_{|A|=m} R\langle F(A)\rangle,\quad d_m = \undersumtwo{|A|=m}{A=(a_1,\ldots,a_m)}{\sum_{i=1}^{m} (-1)^{i}\btoab{\bF_2}{F(A\smallsetminus \{a_{i}\}\subset A)}}.\]
\end{df}

\begin{df}
A map $f\colon X_\bullet\to Y_\bullet$ of \sobs is an \emph{equivalence} if its $\bZ$-realisation $C_*(f;\bZ)\colon C_*(X_\bullet;\bZ)\to C_*(Y_\bullet;\bZ)$ is a chain homotopy equivalence. Two \sobs $X_\bullet,Y_\bullet$ are \emph{equivalent} if there is a zig-zag of equivalences between them, in which case we write $X_\bullet\simeq Y_\bullet$.
\end{df}
\begin{df} A map $f\colon F\to G$ of cubes in the Burnside category, is an \emph{equivalence} if the induced map $C_*(f)\colon C_*(F;\bZ)\to C_*(G;\bZ)$ is a chain homotopy equivalence. Two cubes $F,G\colon \cube{n}\to\cB$ are \emph{equivalent} if there is a zig-zag of equivalences between them, in which case we write $F\simeq G$.
\end{df}

The inclusions $\Set\subset\Setp\subset \cB$ of Section \ref{ssection:burnside} induce inclusions of categories
\begin{equation}\label{eq:inclusions}\presh{\Deltainjaug}{\Set}\lra \presh{\Deltainjaug}{\Setp}\lra \presh{\Deltainjaug}{\cB}\end{equation}
which commute with the realisation functors.

\subsection{\Sobs}\label{ssection:sob} An \sob $X\colon \Deltainjaug^\op\to \cB$ will be denoted $X_\bullet$. Additionally, the set $X(n)$ is denoted $X_n$, the locally finite span $X(\partial^n_U)$ is denoted $\partial^n_U$ and, if $\partial^n_U= \partial^{n-q}_{W}\circ \partial^n_V$, the fibrewise bijection $X(\partial^n_{V_1},\partial^{n-q}_{W_2})$ is denoted $\mu^n_{V,\gamma_{V}(W)}$. With this notation, \ansob consists of the following data: 
\begin{enumerate}
\item For each $n\geq -1$ a set $X_n$,
\item For each $n\geq -1$ and each $U \in \Power_q(n)$, a locally finite span $X_n\la Q^n_{U}\to X_{n-q}$ that we denote by $\partial^n_U$.
\item\label{sob:3} For each $n\geq -1$ and each $U \in \Power_q(n)$, and each partition $U=V_1\cup V_2$ with $V_1\in \Power_{q_1}(n)$, a fibrewise bijection
\[\mu^n_{V_1,V_2}\colon \partial^{n}_U\lra \partial^{n-q_1}_{\psi_{V_1}(V_2)}\circ \partial^n_{V_1},\]
i.e., a fibrewise bijection over $X_n$ and $X_{n-q_1-q_2}$
\[\mu^n_{V_1,V_2}\colon Q^{n}_U\lra Q^n_{V_1}\underset{X_{n-q_1}}{\times} Q^{n-q_1}_{\psi_{V_1}(V_2)}.\]
\end{enumerate}
such that 
\begin{enumerate}[resume]
\item\label{sob:4} for each $n\geq -1$, $\partial^n_\emptyset$ is the identity span on $X_n$, 
\item\label{sob:5} for each $n\geq -1$, and each $V\in \Power_q(n)$, $\mu_{\emptyset,V}$ and $\mu_{V,\emptyset}$ are the identity bijections
\item\label{sob:6} for each $n\geq -1$, and each $U\in \Power_q(n)$ and each partition $U=V_1\cup V_2\cup V_3$ with $V_i\in \Power_{q_i}(n)$, the following square commutes:
\[\xymatrix{
\partial^n_U\ar[rr]^-{\mu^n_{V_1,V_{23}}} \ar[d]^-{\mu^n_{V_{12},V_3}} &&  \partial^{n_1}_{W_{23}}\circ \partial^n_{V_1} \ar[d]^-{\Id\times \mu^{n_1}_{W_2,W_3}} \\
\partial^{n_2}_{W_3}\circ \partial^n_{V_{12}}\ar[rr]^-{\mu^{n}_{V_1,V_2}\times \Id} && \partial^{n_2}_{W_3}\circ\partial^{n_1}_{W_2}\circ \partial^n_{V_1} 
}
\]
where
\begin{align*}
n_j &= n - \sum_{i=1}^{j} q_j,& W_2 &= \psi_{V_1}(V_2),& W_3 &= \psi_{V_1\cup V_2}(V_3) \\
V_{12} &= V_1\cup V_2,& V_{23} &= V_2\cup V_3,& 
W_{23} &= \psi_{V_1}(V_2\cup V_3)
\end{align*}
\end{enumerate}
%
%

\begin{notation} We will systematically refer to the set $Q^n_U$ with the name $\partial^n_U$ of the whole span. In particular, ``$x\in \partial^n_U$'' will be used instead of ``$x\in Q^n_U$''. We will denote the diagonal bijection in \eqref{sob:6} as
\[\mu^n_{V_1,V_2,V_3}\colon \partial^{n}_U\lra \partial^{n_2}_{W_3}\circ \partial^{n_1}_{W_2}\circ \partial^n_{V_1} .\]
If $V_1\subset U$, we write $V_2=U\smallsetminus V_1$ and define
\[\xymatrix{
\lambda^n_{V_1}\colon \partial^{n}_U\ar[rr]^-{\mu^n_{V_1,V_2}}&& \partial^{n_1}_{W_2}\circ \partial^n_{V_1}  \ar[rr]^-{\mathrm{proj}} && \partial^{n}_{V_1}.
}  \]
If $V_1,V_2\subset U$ are disjoint, we write $V_3=U\smallsetminus (V_1\cup V_2)$ and define
\[\xymatrix{
\lambda^n_{V_1,V_2}\colon \partial^{n}_U\ar[rr]^-{\mu^n_{V_1,V_2,V_3}}&& \partial^{n_1}_{W_3}\circ\partial^{n_1}_{W_2}\circ \partial^n_{V_1}  \ar[rr]^-{\mathrm{proj}} && \partial^{n_1}_{W_2} = \partial^{n_1}_{\psi_{V_1}(V_2)}.
}  \]
We will omit the superscript $n$ in $\partial_U^n$ or $\mu^n_{V_1,V_2}$ or $\lambda^n_{V_1,V_2}$ whenever it agrees with the the variable $n$ in the context.
\end{notation}

\begin{df}\label{df:osob} An \emph{\osob} is \ansob together with an ordering of each span $\partial_{U}^n$ (i.e., of the set $Q_U^n$).
\end{df}

\begin{df}\label{df:suspension} The suspension of an (ordered) \sob $X_\bullet$ is the (ordered) \sob 
\[(\Sigma X_\bullet)_{-1} = \emptyset,\quad (\Sigma X_\bullet)_n = X_{n-1}\quad \text{if $n>-1$},\]
 with face maps $\bar{\partial}^n_U$ given by
\[\bar{\partial}^n_U = 
\begin{cases} 
\partial^{n-1}_{\psi_0(U)} & \text{if $0\notin U$} \\
\emptyset & \text{if $0\in U$}.
\end{cases}\]
Here $\emptyset$ is the empty span $(\Sigma X_\bullet)_n\leftarrow \emptyset \to (\Sigma X_\bullet)_{n-|U|}$. If $x\in X_{n-1}$, we write $\Sigma x$ for the corresponding element in $(\Sigma X_\bullet)_{n}$. Note that $C_*(\Sigma X_\bullet;R)=\Sigma C_*(X_\bullet;R)$.
\end{df}

\subsection{Maps between \sobs}\label{ssection:sobmaps} Let $X_\bullet$ and $Y_\bullet$ be \sobs, and write $\partial_U^n$ and $\mu^n_{V_1,V_2}$ for the structure maps in $X_\bullet$ and $\bar{\partial}_U^n$ and $\bar{\mu}^n_{V_1,V_2}$ for the structure maps in $Y_\bullet$.

A map $f$ from $X_\bullet$ to $Y_\bullet$ 
 consists on the following data:
\begin{enumerate}
\item For every $n\geq -1$, a span $f_n$: $X_n\la F_n\to Y_n$,
\item For every $n\geq -1$, every $q\geq 0$ and every $U\in \Power_q(n)$, a fibrewise bijection $f^n_{U}\colon  f_{n-q}\circ \partial^n_U \to \bpartial^n_U\circ f_n$. In other words, a $2$-morphism
\[\xymatrix{
X_n\ar[r]^{f_n}\ar[d]_{\partial^n_U} & Y_n \ar[d]^{\bar{\partial}^n_U}\ar@{=}[ld]|-{f^{n}_{U}} \\
X_{n-q}\ar[r]^{f_{n-q}} & Y_{n-q}
}\]
\end{enumerate}
such that 
\begin{enumerate}[resume]
\item for each $n\geq -1$, $f^n_\emptyset$ is the identity bijection,
\item for each $n\geq -1$, each $U\in \Power_q(n)$ and each partition $U=V_1\cup V_2$, we have, using the notation of \ref{ssection:sob}, that the following diagram of $2$-morphisms commutes:
\[\xymatrixcolsep{4pc}\xymatrixrowsep{4pc}\xymatrix{&
\ar@/_5pc/[dd]|{\caja{\partial^n_U}}  X_n \ar[d]|{\caja{\partial^n_{V_1}}} \ar[r]|{F_n} \ar@{}[dr]|{\caja{f^n_{V_1}}} 
&  
\ar@/^5pc/[dd]|{\caja{\bar{\partial}^n_U}}Y_n      \ar[d]|{\caja{\bar{\partial}^n_{V_1}}} 		
&\\&
\ar@{}[l]|{\caja{\mu^n_{V_1,V_2}}} X_{n_1} \ar[d]|{\caja{\partial^{n_1}_{W_2}}}  \ar[r]|{\caja{F_{n_1}}} \ar@{}[dr]|{\caja{f^{n_1}_{W_2}}} 
&
Y_{n_1}  \ar[d]|{\caja{\bar{\partial}^{n_1}_{W_2}}} \ar@{}[r]|{\caja{\bar{\mu}^n_{V_1,V_2}}}
&\\&
X_{n_2} \ar[r]|{\caja{F_{n_2}}} &  Y_{n_2}  &
}\]
in other words, the following diagram of bijections commutes:
\begin{equation}\label{eq:sobmaps2}
\begin{aligned}\xymatrix{
f_{n_2}\circ \partial^n_U \ar[d]^-{\mu^n_{V_1,V_2}} \ar[rr]^{f^n_U} &&
\bar{\partial}^n_U\circ f_n \ar[d]^-{\bar{\mu}^n_{V_1,V_2}}
\\
f_{n_2}\circ \partial^{n_1}_{W_2}\circ \partial^{n}_{V_1} \ar[r]^{f^{n_1}_{W_2}} 
& \bar{\partial}^{n_1}_{W_2}\circ f_{n_1}\circ \partial^{n}_{V_1} \ar[r]^{f^{n}_{V_1}}
& \bar{\partial}^{n_1}_{W_2}\circ \bar{\partial}^{n}_{V_1}\circ f_n
}
\end{aligned}
\end{equation}
(here we have omitted the subscripts under the $\times$ symbols to lighten the diagram).
\end{enumerate}
\begin{remark} If $U=V_1\cup V_2\cup V_3$ and $n-|V_1|=n_1$ and $n-|V_1|-|V_2|=n_2$, the following diagrams commute
\begin{equation}\label{eq:sobmaps_lambda1}
\begin{aligned}
\xymatrix{
\partial^n_U\ar[dd]_{\lambda_{V_1}} & \ar[l]f_{n_2}\circ \partial^n_U\ar[r]^{f^n_U}\ar[d]& \bar{\partial}^n_U\circ f_n \ar[d]\ar[r] & \bpartial^n_{U}\ar[dd]^{\bar{\lambda}_{V_1}}\\
&\partial^{n_1}_{W_2}\circ f_{n_1}\circ \partial^n_{V_1}\ar[r]\ar[d]^{\lambda_{V_1}} & \partial^{n_1}_{W_2}\circ  \partial^n_{V_1}\circ f_{n}\ar[d]^{\bar{\lambda}_{V_1}}&\\ 
\partial^n_{V_1} & \ar[l]f_{n_1}\circ \partial^n_{V_1} \ar[r]^{f^n_{V_1}} & \bar{\partial}^n_{V_1}\circ f_n\ar[r]& \bpartial^n_{V_1}
}
\end{aligned}
\end{equation}
\begin{equation}\label{eq:sobmaps_lambda2}
\begin{aligned}
\xymatrix{
\partial^n_U\ar[dd]_{\lambda_{V_1,V_2}} & \ar[l]f_{n_2}\circ \partial^n_U\ar[r]^{f^n_U}\ar[d]& \bar{\partial}^n_U\circ f_n \ar[d]\ar[r] & \bpartial^n_{U}\ar[dd]^{\bar{\lambda}_{V_1,V_2}}\\
&f_{n_2}\circ \partial^{n_1}_{W_2}\circ \partial^n_{V_1}\ar[r]\ar[d] & \partial^{n_1}_{W_2}\circ f_{n_1}\circ \partial^n_{V_1}\ar[d]&\\ 
\partial^n_{V_1} & \ar[l]f_{n_2}\circ \partial^{n_1}_{W_2} \ar[r]^{f^{n_1}_{W_2}} & \bar{\partial}^{n_1}_{V_1}\circ f_{n_1}\ar[r]& \bpartial^n_{V_1,V_2}
}
\end{aligned}
\end{equation}
\end{remark}

\section{Stable symmetric multiplications}\label{section:theorem}
In this section we state the bulk of our main theorem: the cochain complex of any \sob has a stable symmetric multiplication. The proof is given in the next section. Our presentation is parallel to the presentation of the symmetric multiplication on the cochain complex of a simplicial set given in \cite{Anibal}.

\subsection{Sequences}
Let $\APower_q(n)$ be the set of $q$-tuples $U$ of non-decreasing non-negative integers bounded by $n$ where every number appears at most twice, i.e., an element is a sequence $(0\leq u_1\leq u_2\leq\ldots\leq u_q\leq n)$ where $u_{i-1}=u_i=u_{i+1}$ never occurs. The cardinal of a sequence $U$ will be denoted $|U|$.
 
If $U\in \APower_q(n)$, let $\oneU\subset U$ be the subset of non-repeated numbers and let $\twoU\subset U$ be the subset of repeated numbers. For example, if $U=(1,2,2,3,4,4)$, then $\oneU=(1,3)$ and $\twoU=(2,4)$. Define $\APower_q^r(n) = \{U\in \APower_q(n)\mid |\twoU|=r\}$.

If $U\in \APower^0_q(n)$, define the \emph{index of $u_i$ in $U$} as $\ind_U(u_i)=u+i$. If $U\in P_q^r(n)$, let $\bar{U}\in \APower^0_{q-r}(n)$ be the result of removing one instance of each repeated number. If $u_i\in \oneU$, define the \emph{index of $u_i$ in $U$} as $\ind_{\bar{U}}(u_i)$ and if $u_i\in \twoU$, define the \emph{index of $u_i$ in $U$} as $\ind_U(u_i)=\{u+i,u+i+1\}$. Set $U^+,U^-\subset U$ as the subsets that contain all elements of even (odd) index.


\subsection{Wedge products of face spans} Let $X_\bullet$ be an \osob, as defined in Section \ref{ssection:sob}.
\begin{df}\label{df:good} Let $U,V\in \Power(n)$ and let $(s,t)\in \partial_{U}\times \partial_V$. A pair of disjoint subsets $W=(W^\shortparallel,W^\smallwedge)$ of $U\cap V$ is \emph{$(s,t)$-good} if 
\begin{align*}
\lambda_{W^\shortparallel}(s)&=\lambda_{W^\shortparallel}(t) & \lambda_{W^\shortparallel,W^\smallwedge}(s)&\neq \lambda_{W^\shortparallel,W^\smallwedge}(t).\end{align*}
Furthermore, such pair is \emph{$(s,t)$-positive} (or \emph{positive}, for short) if
\begin{align*}
n+ |U|+|V|\text{ is even and }\lambda_{W^\shortparallel,W^\smallwedge}(s)&< \lambda_{W^\shortparallel,W^\smallwedge}(t)\text{ or}\\
n+ |U|+|V|\text{ is odd and }\lambda_{W^\shortparallel,W^\smallwedge}(s)&> \lambda_{W^\shortparallel,W^\smallwedge}(t)
\end{align*}
 and \emph{$(s,t)$-negative} (or \emph{negative}, for short) if
\begin{align*}
n+ |U|+|V|\text{ is even and }\lambda_{W^\shortparallel,W^\smallwedge}(s)&> \lambda_{W^\shortparallel,W^\smallwedge}(t)\text{ or}\\
n+ |U|+|V|\text{ is odd and }\lambda_{W^\shortparallel,W^\smallwedge}(s)&<\lambda_{W^\shortparallel,W^\smallwedge}(t).
\end{align*}
\end{df}
Clearly, if $\lambda_{W^\shortparallel}(s)=\lambda_{W^\shortparallel}(t)$ and $W_1^\shortparallel\subset W^\shortparallel$, then $\lambda_{W^\shortparallel_1}(s)=\lambda_{W^\shortparallel_1}(t)$ as well, and if $\lambda_{W^\shortparallel,W^\smallwedge}(s)\neq\lambda_{W^\shortparallel,W^\smallwedge}(t)$ and $W^\smallwedge\subset W_1^\smallwedge$, then $\lambda_{W^\shortparallel,W^\smallwedge_1}(s)\neq\lambda_{W^\shortparallel,W^\smallwedge_1}(t)$ too. In this line, we have the following straightforward lemma:
\begin{lemma}\label{lemma:parallelicity} Let $U,V\in \Power(n)$, let $(s,t)\in \partial_{U}\times \partial_V$, let $(W^\shortparallel,W^\smallwedge)$ be a $(s,t)$-good pair and suppose that $0\leq w\leq n$ with $w\notin W^\smallwedge$. 
\begin{enumerate}
\item\label{para:1} The pair $(W^\shortparallel,W^\smallwedge\cup \{w\})$ is $(s,t)$-good if and only if $w\notin W^\shortparallel$, and 
\item the pair $(W^\shortparallel\smallsetminus \{w\},W^\smallwedge\cup \{w\})$ is $(s,t)$-good if and only if $w\in W^\shortparallel$.
\end{enumerate}
\end{lemma}

Two pairs $(W^\shortparallel_0,W^\smallwedge_0)$ and $(W^\shortparallel_1,W^\smallwedge_1)$ are \emph{consecutive} if $W^\smallwedge_1$ is obtained from $W^\smallwedge_0$ by adding an element and $W^\shortparallel_1$ is the result of removing that element (if it is there) from $W^\shortparallel_0$. We write $(W^\shortparallel_0,W^\smallwedge_0)\prec(W^\shortparallel_1,W^\smallwedge_1)$ if the two pairs are consecutive. Let $\cO_{s,t}(U\cap V)$ be the collection of maximal chains 
\[(W^\shortparallel_1,W^\smallwedge_1)\prec\ldots\prec (W^\shortparallel_r,W^\smallwedge_r),\quad r=|U\cap V|\]
of $(s,t)$-good pairs.  A maximal chain is \emph{positive} (resp.\ \emph{negative}) if all its entries are positive (resp.\ negative). A maximal chain is \emph{almost positive} (resp.\ \emph{almost negative}) if $(W^\shortparallel_i,W^\smallwedge_i)$ is positive (resp.\ negative) for all $i\neq r$. For every maximal chain $W$, necessarily $(W_r^\shortparallel,W_r^\smallwedge) = (\emptyset,U\cap V)$. We say that $(s,t)$ is \emph{positive} (\emph{negative}) if $(\emptyset,U\cap V)$ is positive (negative). We introduce the following subsets of $\cO_{s,t}(U\cap V)$:

\vspace{.3cm}

\begin{tabular}{l p{8cm}}
 $\cO_{s,t}(U\cap V)^\pm$ & the set of positive (negative) maximal chains,\\[.2cm]
 $\cO_{s,t}(U\cap V)^\pm_\diamond$& the set of almost positive (negative) maximal chains,\\[.2cm]
 $\cO_{s,t}(U\cap V)^\pm_\bdiamond$& the set of almost positive (negative) maximal chains such that $W^\shortparallel_{r-1}= \emptyset$,\\[.2cm]
 $\cO_{s,t}(U\cap V)^\pm_\udiamond$& the set of almost positive (negative) maximal chains such that $W^\shortparallel_{r-1}\neq \emptyset$,\\[.2cm]
 $\cO_{s,t}(U\cap V)^{\pm,x}_\bdiamond$& the set of almost positive (negative) maximal chains such that $W^\shortparallel_{r-1}= \emptyset$\phantom{\! } and~$U\cap V\smallsetminus W_{r-1}^\smallwedge = \{x\}$,\\[.2cm]
 $\cO_{s,t}(U\cap V)^{\pm,x}_\udiamond$& the set of almost positive (negative) maximal chains such that $W^\shortparallel_{r-1}\neq \emptyset$\phantom{\! } and~$U\cap V\smallsetminus W_{r-1}^\smallwedge = \{x\}$.
\end{tabular}

\begin{df} Let $U,V$ be sequences in $\Power(n)$ and let $x\in U\cap V$. If the symbol $\square$ denotes either $\lozenge,\vartriangle,\triangledown$ or the absence of a symbol, define the spans
\begin{align*}
\partial_U\wedge_\square\partial_V &\quad X_n\times X_n\longleftarrow \sum_{(s,t)\in \partial_U\times \partial_V} \cO_{s,t}(U\cap V)_\square^+ \longrightarrow X_{n-|U|}\times X_{n-|V|}\\
\partial_U\bwedge_\square\partial_V &\quad X_n\times X_n\longleftarrow \sum_{(s,t)\in \partial_U\times \partial_V} \cO_{s,t}(U\cap V)_\square^- \longrightarrow X_{n-|U|}\times X_{n-|V|} \\
\partial_U\wedge_\square^x\partial_V &\quad X_n\times X_n\longleftarrow \sum_{(s,t)\in \partial_U\times \partial_V} \cO_{s,t}(U\cap V)_\square^{+,x} \longrightarrow X_{n-|U|}\times X_{n-|V|}\\
\partial_U\bwedge_\square^x\partial_V &\quad X_n\times X_n\longleftarrow \sum_{(s,t)\in \partial_U\times \partial_V} \cO_{s,t}(U\cap V)_\square^{-,x} \longrightarrow X_{n-|U|}\times X_{n-|V|}, 
\end{align*}
where the value of the source and target maps at $\bW\in \cO_{s,t}(U\cap V)^\pm_\square$ is
\[
\source(\bW) = (\source(s),\source(t)),\quad \target(\bW) = (\target(s),\target(t)).
\]
Note that if $U\cap V=\emptyset$, then there is a single maximal chain of length $0$ that is both positive and negative, hence the first two spans are equivalent to $\partial_U\times\partial_V$, and the last two are the empty span. Note also that if $s=t$, then $\cO_{s,t}(\twoU)$ is always empty.
\end{df}


\subsection{Stable symmetric comultiplications} Let $(C_*,d)$ be a chain complex of $\bF_2$-modules, let $T\colon C_*\otimes C_*\to C_*\otimes C_*$ be the twist homomorphism $T(a\otimes b)=b\otimes a$, and let $1\colon C_*\otimes C_*\to C_*\otimes C_*$ denote the identity map. A \emph{stable symmetric comultiplication} on $(C_*,d)$ is a family of homomorphisms 
\begin{align*}\nabla_i\colon C_*&\lra C_*\otimes C_*
\end{align*}
with $i\in \bZ$, such that $\nabla_i$ has degree $i$ and 
\begin{equation}\label{eq:comult}
\nabla_i\circ d = d\circ\nabla_i + (1+T)\nabla_{i-1}.
\end{equation}
The homomorphisms $\smile_i\colon C^*\otimes C^*\lra C^*$ dual to $\nabla_i$ 
 endow the dual cochain complex with a stable symmetric multiplication..

 Let $X_\bullet$ be \ansob, and let $C_* := C_*(X_\bullet;\bF_2)$ be its $\bF_2$-realisation. Let $n\geq -1$ and $q\geq 0$. Define the span $\nabla^{\binom{n}{q}}$ from the set $X_n\times X_n$ to the set $\bigcup_{i+j=q}X_{n-i}\times X_{n-j}$ as \mnote{la notacion no es ideal: hay sumas y coproductos mezclados}
\begin{equation}\label{eq:nabla5}
\nabla^{\binom{n}{q}} = \sum_{U\in \APower_q(n)}  \partial_{U^-}\wedge \partial_{U^+}.
\end{equation}
The $\bF_2$-realisation of $\nabla^{\binom{n}{q}}$ is then a homomorphism $C_n\otimes C_n\to (C_*\otimes C_*)_{2n-q}$. Define the homomorphism $\nabla_k\colon C_*\lra C_*\otimes C_*$ of graded modules of degree~$k$ 
\[\nabla_k = \btoab{\bF_2}{\sum_{n} \nabla^{\binom{n}{n-k}}\circ\Delta_n}.\]
where $\Delta_n\colon X_n\to X_n\times X_n$ is the diagonal map. 

\begin{theorem}\label{thm:main} The chain complex with $\bF_2$ coefficients of an \osob $X_\bullet$ has a symmetric comultiplication whose $i$th operation is the homomorphism $\nabla_i$. 
\end{theorem}
If $U,V\in \Power_q(n)$ and $\partial_{U\cap V}$ is a free span, then the composition $(\partial_{U\cap V}\times \partial_{U\cap V})\circ \Delta$ is a free span as well. Hence, if $U\cap V\neq\emptyset$, and $(s,t)\in \partial^n_U\times \partial^n_V$ and $\source(s)=\source(t)$, then $\lambda_{U\cap V}(s) = \lambda_{U\cap V}(t)$, so $\cO_{s,t}(U\cap V)$ is empty and therefore the span $(\partial_{U}\wedge \partial_{V})\circ \Delta$ is empty. 
\begin{corollary}\label{cor:sset} Let $X_\bullet$ be an \osob. If $U\in \APower_q(n)$ with $\twoU\neq \emptyset$ and $\partial_{\twoU}$ is a free span, then $\partial_{U^-}\wedge\partial_{U^+}\circ\Delta$ is the empty span. As a consequence, if $X_\bullet$ is a semi-simplicial set, one may replace \eqref{eq:nabla5} by
\[\nabla^{\binom{n}{q}} = \sum_{U\in \APower^0_q(n)}  \partial_{U^-}\wedge \partial_{U^+},\]
which is the formula given in \cite{Anibal} to define the symmetric comultiplication on the chain complex of a simplicial set.
\end{corollary}
\section{Proof}\label{section:proof}
For the family $\{\nabla_k\}_{k\in \bZ}$ to be a stable symmetric comultiplication on $C_*(X_\bullet;\bF_2)$, one has to verify \eqref{eq:comult}, which is a consequence of the following equation of spans, which we prove true along this section.
\begin{align}\label{eq:proof}
\partial^{2n-q}\circ\nabla^{\binom{n}{q}}\circ\Delta_n + \nabla^{\binom{n-1}{q-1}}\circ\Delta_{n-1}\circ\partial^{n} &\igualdos (1+T)\nabla^{\binom{n}{q+1}}\circ\Delta_n.
\end{align}
Here, $\partial^{n}$ denotes the span $\sum_{x=0}^{n}\partial_{x}$ from $X_n$ to $X_{n-1}$, whose \abelianisation{\bF_2} is the differential $C_n\to C_{n-1}$. On the other hand, $\partial^{2n-q}$ is the span from $\bigcup_{i+j=2n-q} X_i\times X_j$ to $\bigcup_{i+j=2n-q-1} X_i\times X_j$ given by the union of all spans
\[X_i\times X_j\lla \sum_{y=0}^{i}\partial^i_y\times \Id_{j} + \sum_{y=0}^{j}\Id_i\times \partial^j_y\lra X_{i-1}\times X_j\cup X_i\times X_{j-1},\]
whose \abelianisation{\bF_2} is the differential $(C_*\otimes C_*)_{2n-q}\to (C_*\otimes C_*)_{2n-q-1}$. Here and during the proof we write $\Id_n$ for the identity span $\partial^n_\emptyset$.
\subsection{Rewriting the formula in terms of wedge products} Recall from Section \ref{ssection:sequences} that for each $x\in \{0,\ldots,n\}$ there are functions  
\[\gamma_x\colon \Power_q(n)\overset{\cong}{\lra} \{V\in \Power_{q}(n+1)\mid x\notin V\}\subset \Power_{q}(n+1)\]
\[\xi_x\colon \Power_q(n)\overset{\cong}{\lra} \{V\in \Power_{q+1}(n+1)\mid x\in V\}\subset \Power_{q+1}(n+1)\]
with $\xi_x(U) = \gamma_x(U)\cup \{x\}$. Define a new function 
\[\xi_{xx}\colon \APower_{q-1}^{r-1}(n-1)\hookrightarrow \APower_{q+1}^{r}(n)\]
as $\xi_{xx}(U) = \gamma_x(U)\cup \{x,x\}$, and a new span
\begin{align*}
\partial_x\pwedge \partial_x &= \left\{(a,b)\in \partial_x\times \partial_x\mid a=b\right\}.
\end{align*}
\begin{lemma}\label{lemma:aux_diagonal} If $U\in \APower_{q-1}^{r-1}(n-1)$ and $x\in \{0,\ldots,n\}$, there are equivalences of spans
\begin{align*}
\Delta_{n-1}\circ \partial_x &\igual (\partial_x\pwedge \partial_x)\circ \Delta_n \\
\partial_{U^-}^{n-1}\wedge\partial_{U^+}^{n-1} \circ \partial_x\pwedge\partial_x &\igual \partial_{\xi_{xx}(U)^-}\bwedge_{\udiamond}^x\partial_{\xi_{xx}(U)^+}.
\end{align*}
\end{lemma}
\begin{proof}
The first equation is clear. For the second, let $\ell = n-|U^-|-1$, let $m = n-|U^+|-1$ and let $V=\xi_{xx}(U)\in \APower_{q}^r(n)$, and observe that 
\begin{align*}
V^- &= \xi_{x}(U^-) & V^+ &= \xi_x(U^+) & \twoV=\twoU\cup \{x\}.
\end{align*} 
By definition, the left hand-side is
\begin{equation}\label{eq:72}
X_n\times X_n\longleftarrow \undersumthree{(a,a)\in \partial_x\pwedge \partial_x}{(s,t)\in \partial_{U^-}\times \partial_{U^+}}{\source((s,t)) = \target((a,a))}{\{(a,a)\}\times {\cO_{s,t}(\twoU)^+}} \longrightarrow X_{\ell}\times X_{m}
\end{equation}
and the right hand-side is
\[X_n\times X_n\longleftarrow \sum_{(s',t')\in \partial_{V^-}\times \partial_{V^+}}{\cO_{s',t'}(\twoV)^{-,x}_{\udiamond}} \longrightarrow X_{\ell}\times X_{m}.\]
Observe first that $\cO_{s',t'}(\twoV)^{+,x}_{\udiamond}$ is non-empty only if $\lambda_x(s')= \lambda_x(t')$. Moreover, $\mu_{x,V^{-}\smallsetminus \{x\}}\times \mu_{x,V^{+}\smallsetminus \{x\}}$ induces a bijection between the set of those $(s',t')$ such that $\lambda_x(s')= \lambda_x(t')$ and the indexing set of the summation in \eqref{eq:72}, given by sending $(s',t')$ to $((a,a),(s,t))$, where 
\begin{align*}
a&=\lambda_x(s')= \lambda_x(t'),& s&=\lambda_{x,U^{-}\smallsetminus x}(s') & t&=\lambda_{x,U^{+}\smallsetminus x}(t').
\end{align*}
To finish the proof, we construct, for each $s\neq t$, a bijection
\[\varphi_{a,s,t}\colon \{(a,a)\}\times \cO_{s,t}(\twoU)^+\longrightarrow \cO_{s',t'}(\twoV)_{\udiamond}^{-,x}\]
given by sending a maximal $(s,t)$-good chain $(W^\shortparallel_1,W^\smallwedge_1)\prec\ldots\prec (W^\shortparallel_{r-1},W^\smallwedge_{r-1})$ in $\cO_{s,t}(\twoU)^+$ to the maximal sequence
\[(\xi_x(W^\shortparallel_1),\gamma_x(W^\smallwedge_1))\prec\ldots\prec (\xi_x(W^\shortparallel_{r-1}),\gamma_x(W^\smallwedge_{r-1})) \prec (\emptyset,\twoV)\]
which is $(s',t')$-good because for each $i=1,\ldots,r-1$
\begin{align*}
\lambda_{\xi_x(W^\shortparallel_i)}(s') &= (a,\lambda_{W^\shortparallel_i}(s)), &\lambda_{\xi_x(W^\shortparallel_i),\gamma_x(W^\smallwedge_i)}(s') &= \lambda_{W^\shortparallel_i,W^\smallwedge_i}(s) \\
\lambda_{\xi_x(W^\shortparallel_i)}(t') &= (a,\lambda_{W^\shortparallel_i}(t)), &\lambda_{\xi_x(W^\shortparallel_i),\gamma_x(W^\smallwedge_i)}(t') &= \lambda_{W^\shortparallel_i,W^\smallwedge_i}(t).
\end{align*}
The left hand-side of the left column should be $\mu_{x,\gamma_x(W^\shortparallel)}(\lambda_{\xi_x(W^\shortparallel_i)}(s'))$ up and $\mu_{x,\gamma_x(W^\shortparallel)}(\lambda_{\xi_x(W^\shortparallel_i)}(t'))$ down, but we find the complete notation too burdened. The remaining pair $(\emptyset,\twoV)$ is also $(s',t')$-good because $s\neq t$.
Additionally, since all pairs in the original $(s,t)$-good chain were positive and 
\[n-1 + |U^-|+|U^+| = n-1+q-1,\quad n+|V^-| + |V^+| = n+q+1\]
have different parity, we have that all the new pairs are negative except possibly the last one, therefore the new chain is almost negative. Moreover $(\xi_x(W^\shortparallel_r),\gamma_x(W^\smallwedge_r)) = (\xi_x(\emptyset),\gamma_x(\twoU)) = (\{x\},\gamma_{x}(\twoU))$, so the new chain is in $\cO_{s',t'}(\twoV)_{\udiamond}^{-,x}$.

The maximal chains $(A^\shortparallel_1,A^\smallwedge_1)\prec\ldots\prec (A^\shortparallel_r,A^\smallwedge_r)$ in $\cO_{s',t'}(\twoV)_{\udiamond}^{-,x}$ are precisely the almost negative $(s,t)$-good maximal chains that satisfy that $x\in A^\shortparallel_i$ and $x\notin A^\smallwedge_i$ for all $i<r$. Hence, for every maximal chain $(A^\shortparallel_1,A^\smallwedge_1)\prec\ldots\prec (A^\shortparallel_{r},A^\smallwedge_{r})$ in $\cO_{s',t'}(\twoV)_{\udiamond}^{-,x}$ and every $0<i<r$, the inverse function $\gamma_x^{-1}$ is well-defined on $A^\smallwedge_i$ and the inverse function $\xi_x^{-1}$ is well-defined on $A^\shortparallel_i$. Hence
 we can define an inverse of $\varphi_{a,s,t}$ by sending a sequence $(A^\shortparallel_1,A^\smallwedge_1)\prec\ldots\prec (A^\shortparallel_{r},A^\smallwedge_{r})$ to the sequence $(\emptyset,\emptyset)\prec(\xi_x^{-1}(A^\shortparallel_1),\gamma_x^{-1}(A^\smallwedge_1))\prec\ldots\prec (\xi_x^{-1}(A^\shortparallel_{r-1}),\gamma_x^{-1}(A^\smallwedge_{r-1}))$.
\end{proof}
\begin{proposition}\label{prop:diagonalpart} The span $\nabla^{\binom{n-1}{q-1}}\circ \Delta_{n-1}\circ \partial^n$ is equivalent to the following spans
\begin{align*} 
 \sum_{x=0}^n\nabla^{\binom{n-1}{q-1}}\circ \Delta_{n-1}\circ \partial_x
&\igual \sum_{x=0}^n\left(\sum_{U\in \APower_{q-1}(n-1)} \partial_{U^-}\wedge\partial_{U^+}\right)\circ\Delta_{n-1}\circ\partial_x \\
&\igual \sum_{U\in \APower_{q-1}(n-1)}\sum_{x=0}^n \partial_{U^-}\wedge\partial_{U^+}\circ \partial_x\pwedge\partial_x\circ \Delta_n \\
&\igual \sum_{U\in \APower_{q-1}(n-1)}\sum_{x=0}^n \partial_{\xi_{xx}(U)^-}\bwedge_{\udiamond}^x\partial_{\xi_{xx}(U)^+}\circ\Delta_n \\
&\igual \sum_{U\in \APower_{q+1}(n+1)}\sum_{x\in \twoU} \partial_{U^-}\bwedge_{\udiamond}^x\partial_{U^+}\circ\Delta_n \\
&\igual \sum_{U\in \APower_{q+1}(n+1)}\partial_{U^-}\bwedge_{\udiamond}\partial_{U^+}\circ\Delta_n.
\end{align*}
\end{proposition}
\begin{proof}
The first equality is the definition, the second and third follow from Lemma \ref{lemma:aux_diagonal}, the fourth comes from sending the summand indexed by $U$ and $x$ to the summand indexed by $\xi_{xx}(U)$ and $x$, and the fifth holds by definition.
\end{proof}

\begin{lemma}\label{lemma:nondiagonalpart_aux} Let $U\in \APower_{q}^r(n)$ and let $\ell=|U^-|$ and $m=|U^+|$, and let $y\in \{0,\ldots,n-\ell\}$ or $y\in \{0,\ldots,n-m\}$ and let $x=\gamma_{U^-}(y)$ in the first case and $x=\gamma_{U^+}(y)$ in the second case. Then the following spans are equivalent
\begin{align*} 
\left(\partial^{n-\ell}_{y}\times \Id_{n-m}\right) \circ \left(\partial_{U^-}\wedge\partial_{U^+}\right)  &\igual  \begin{cases} 
\partial_{\xi_{U^-}(y)}\bwedge \partial_{U^+}  &\text{if $x\notin U^+$} \\
\partial_{\xi_{U^-}(y)}\bwedge_\bdiamond^x \partial_{U^+}  &\text{if $x\in U^+$}
\end{cases} \\
\left(\Id_{n-\ell}\times \partial^{n-m}_{y}\right) \circ \left(\partial_{U^-}\wedge\partial_{U^+}\right)  &\igual  \begin{cases}
\partial_{U^-}\bwedge \partial_{\xi_{U^+}(y)} &\text{if $x\notin U^-$} \\
  \partial_{U^-}\bwedge_\bdiamond^x \partial_{\xi_{U^+}(y)}  &\text{if $x\in U^-$}\rlap{.}
	\end{cases}
\end{align*}
\end{lemma}
\begin{proof} We give the proof of the first equation, the other being completely analogous. 
By definition, the left hand-side is
\begin{align}\label{eq:73}
X_n\times X_n\longleftarrow \undersumthree{(s,t)\in \partial_{U^-}\times \partial_{U^+}}{a\in \partial^{n-\ell}_y}{\source(a) = \target(s)}{\cO_{s,t}(U^-\cap U^+)^+\times \{a\}} \longrightarrow X_{\ell-1}\times X_{m} 
\end{align}
and the right hand-side is, depending on whether $x\notin U^+$ or $x\in U^+$,
\begin{equation}\label{eq:74}X_n\times X_n\longleftarrow \sum_{(s',t')\in \partial_{\xi_{U^-}(y)}\times \partial_{U^+}}{\cO_{s',t'}\left(\xi_{U^-}(y)\cap U^+\right)^{-}} \longrightarrow X_{\ell-1}\times X_{m}\end{equation}
\begin{equation}\label{eq:75}X_n\times X_n\longleftarrow \sum_{(s',t')\in \partial_{\xi_{U^-}(y)}\times \partial_{U^+}}{\cO_{s',t'}\left(\xi_{U^-}(y)\cap U^+\right)^{-,x}_{\bdiamond}} \longrightarrow X_{\ell-1}\times X_{m}.\end{equation}
The bijection
\[\mu_{{U^-},x}\colon \partial_{\xi_{U^-}(y)}\overset{\cong}{\lra} \partial^{n-\ell}_y\circ \partial_{U^-}\]
induces a bijection between the indexing set of the summation in \eqref{eq:73} and the indexing set of the summation in \eqref{eq:74} (or in \eqref{eq:75}) given by sending $(s',t')$ to $((s,t),a)$, where 
\begin{align*}
s&=\lambda_{U^-}(s'),& t&=t' & a&=\lambda_{U^-,x}(s').
\end{align*}
Therefore it remains to build, in each case, bijections 
\begin{align}\label{eq:81}\cO_{s,t}(\twoU)^+&\longrightarrow \cO_{s',t'}\left(\xi_{U^-}(y)\cap U^+\right)^{-},\\ \label{eq:82} \cO_{s,t}(\twoU)^+ &\longrightarrow  \cO_{s',t'}\left(\xi_{U^-}(y)\cap U^+\right)^{-,x}_{\bdiamond}.
\end{align}

In the first case, as $x\notin U^+$, we have that 
\[\xi_{U^-}(y)\cap U^+ = (\{x\}\cup U^-)\cap U^+ = U^-\cap U^+ = \twoU.\]
Define \eqref{eq:81} by sending a maximal chain to itself. This sends $(s,t)$-good chains to $(s',t')$-good chains because
for any pair $(W^\shortparallel,W^\smallwedge)$ of disjoint subsets of $U^-\cap U^+$, we have that
\begin{align*}
\lambda_{W^\shortparallel}(s) &= \lambda_{W^\shortparallel}(s'), & \lambda_{W^\shortparallel,W^{\smallwedge}}(s) &= \lambda_{W^\shortparallel,W^\smallwedge}(s') \\
\lambda_{W^\shortparallel}(t) &= \lambda_{W^\shortparallel}(t'), & \lambda_{W^\shortparallel,W^{\smallwedge}}(t) &= \lambda_{W^\shortparallel,W^\smallwedge}(t'),
\end{align*}
and it sends positive pairs to negative pairs and vice versa because
\begin{align*}n+|U^+|+|U^-| &= n+q+r, & n+|\xi_{U^-}(x)|+|U^+| &= n+q+r+1\end{align*}
have different parity.

In the second case, as $x\in U^+$, we have that,
\[\xi_{U^-}(y)\cap U^+ = (\{x\}\cup U^-)\cap U^+ = \{x\} \cup (U^-\cap U^+) = \{x\}\cup\twoU.\]
Define \eqref{eq:82} by sending a maximal chain $(W^\shortparallel_1,W^\smallwedge_1)\prec\ldots\prec (W^\shortparallel_{r},W^\smallwedge_{r})$ in $\cO_{s,t}(\twoU)^+$ to the maximal chain 
\begin{equation}\label{eq:49}(W^\shortparallel_1,W^\smallwedge_1)\prec\ldots\prec (W^\shortparallel_{r},W^\smallwedge_{r})\prec (\emptyset,\{x\}\cup\twoU).\end{equation}
in $\cO_{s',t'}(\{x\}\cup \twoU)$. For $i=1,\ldots,r$, we have that 
\begin{align}\label{eq:50}
\lambda_{W^\shortparallel_i}(s) &= \lambda_{W^\shortparallel}(s'), & \lambda_{W^\shortparallel_i,W^\smallwedge_i}(s) &= \lambda_{W^\shortparallel_i,W^\smallwedge_i}(s') \\
\lambda_{W^\shortparallel_i}(t) &= \lambda_{W^\shortparallel_i}(t'), & \lambda_{W^\shortparallel_i,W^\smallwedge_i}(t) &= \lambda_{W^\shortparallel_i,W^\smallwedge_i}(t')
\end{align}
Therefore, as $(W^\shortparallel_i,W^\smallwedge_i)$ is $(s,t)$-good for $i=1,\ldots,r$, it is also $(s',t')$-good for $i=1,\ldots,r$. If $i=r+1$, we have that $\lambda_{W^\shortparallel_i}(s')\neq \lambda_{W^\shortparallel_i}(t')$ because of Lemma \ref{lemma:parallelicity} \eqref{para:1}, taking $W^\shortparallel= W^\shortparallel_r$ and $w=x$ so that $W^\shortparallel_{r+1} = W^\shortparallel\cup \{w\}$. Therefore, the maximal chain \eqref{eq:49} is $(s',t')$-good.

Since for each $i=1\ldots r$, $(W_i^\shortparallel,W_i^\smallwedge)$ is positive for $(s,t)$, we have that $(W_i^\shortparallel,W_i^\smallwedge)$ is negative for $(s',t')$ because of the right hand-side of \eqref{eq:50} and because
\begin{align*}n+|U^+|+|U^-| &= n+q+r, & n+|\xi_{U^-}(x)|+|U^+| &= n+q+r+1\end{align*}
have different parity. Finally, by construction $\{x\}\cup \twoU\smallsetminus W_r^\smallwedge = \{x\}$. Hence \eqref{eq:82} is well-defined. Its inverse sends a maximal chain $(A_1^\shortparallel,A_1^\smallwedge)\prec\ldots\prec(A_{r+1}^\shortparallel,A_{r+1}^\smallwedge)$ to the maximal chain $(A_1^\shortparallel,A_1^\smallwedge)\prec\ldots\prec(A_{r}^\shortparallel,A_{r}^\smallwedge)$.
\end{proof}

If $U\in \APower_{q+1}(n)$ and $x\in U$, let us write $U(x)\in \APower_{q}(n)$ for the result of removing one instance of $x$ from $U$. For example, if $U=(1,2,2,3)$, then $U(2) = (1,2,3)$ and $U(3) = (1,2,2)$.
\begin{proposition}\label{prop:nondiagonalpart} The span $\partial^{2n-q}\circ\nabla^{\binom{n}{q}}$ is equivalent to the following spans
\begin{align} \nonumber
&\sum_{U\in \APower_{q}(n)}
\left(\sum_{y=0}^{n-|U^-|} \left(\partial^{n-|U^-|}_{y}\times \Id_{n-|U^+|}\right) \circ \left(\partial_{U^-}\wedge\partial_{U^+}\right) \right. 
\\ \nonumber
&+\left.\sum_{y=0}^{n-|U^+|} \Id_{n-|U^-|}\times \partial^{n-|U^+|}_{y}\circ \partial_{U^-}\wedge\partial_{U^+} \right) 
\\ \nonumber
\igual& \sum_{U\in \APower_{q}(n)}\left(
\sum_{x\notin U} \partial_{\{x\}\cup U^-}\bwedge \partial_{U^+} + \partial_{U^-}\bwedge \partial_{\{x\}\cup U^+} \right.
\\ \nonumber
&+ 
\sum_{x\in U^+\smallsetminus U^-}  \partial_{\{x\}\cup U^-}\bwedge_\bdiamond^x \partial_{U^+} 
\\ \nonumber
&+ 
\left.\sum_{x\in U^-\smallsetminus U^+} \partial_{U^-}\bwedge_\bdiamond^x \partial_{\{x\}\cup U^+}\right) \\ \label{eq:non_repeated}
\igual& \sum_{U\in \APower_{q+1}(n)}\left(
\sum_{x\in \oneU} {\partial_{\{x\}\cup U(x)^-}\bwedge \partial_{U(x)^+} + \partial_{U(x)^-}\bwedge \partial_{\{x\}\cup U(x)^+}} \right.
\\ \label{eq:repeated_even}
&+ 
\sum_{x\in \twoU, \ind_{U(x)}(x)~\mathrm{even}}  \partial_{\{x\}\cup U(x)^-}\bwedge_\bdiamond^x \partial_{U(x)^+} 
\\ \label{eq:repeated_odd}
&+ 
\left.\sum_{x\in \twoU, \ind_{U(x)}(x)~\mathrm{odd}} \partial_{U(x)^-}\bwedge_\bdiamond^x \partial_{\{x\}\cup U(x)^+}\right)
\end{align}
\end{proposition}
\begin{proof} The first term is the definition of $\partial^{2n-q}\circ\nabla^{\binom{n}{q}}$. The first equality holds by Lemma \ref{lemma:nondiagonalpart_aux}, by sending the summand indexed by $U$ and $y$ to the summand indexed by $U$ and either $x=\gamma_{U^-}(y)$ or $x=\gamma_{U^+}(y)$, depending on the case (recall that $\xi_{U^\pm}(y)=\{x\}\cup U^\pm$). The second equality is given by sending the summand indexed by $U$ and $x$ to the summand indexed by $U\cup \{x\}$ and $x$.
\end{proof}
We immediately have:
\begin{lemma}\label{lemma:repeated} The summands \eqref{eq:repeated_even} and \eqref{eq:repeated_odd} add up to:
\[\sum_{U\in \APower_{q+1}(n)} \partial_{U^-}\bwedge_\bdiamond\partial_{U^+}.\]
\end{lemma}
\begin{lemma}\label{lemma:non_repeated} The summand \eqref{eq:non_repeated} is $\bF_2$-equivalent to:
\[\sum_{U\in \APower_{q+1}(n)} \partial_{U^-}\bwedge\partial_{U^+} + \partial_{U^+}\bwedge\partial_{U^-}.\]
\end{lemma} 
\begin{proof} The summand \eqref{eq:non_repeated} is empty unless it is indexed by a sequence $U\in \APower_{q+1}(n)$ such that $\oneU\neq \emptyset$. Let $U$ be such a sequence and let $x\in \oneU$. In order to lighten the notation, write $U[x]^\pm = \{x\}\cup U(x)^{\pm}$ and define 
\begin{align*}
\oneU^\pm_{< x} &= \{u\in \oneU^\pm\mid u<x\} & \oneU^\pm_{> x} &= \{u\in \oneU^\pm\mid u>x\},
\end{align*} 
and note that
\begin{gather}\label{eq:85}
\begin{aligned}
U(x)^+ &= \oneU^+_{<x}\cup \oneU^-_{>x}\cup \twoU, &  U[x]^+ &= \oneU^+_{<x}\cup\{x\}\cup \oneU^-_{>x}\cup \twoU, \\
U(x)^- &= \oneU^-_{<x}\cup \oneU^+_{>x}\cup \twoU, &  U[x]^- &= \oneU^-_{<x}\cup\{x\}\cup \oneU^+_{>x}\cup \twoU. 
\end{aligned}\end{gather}
Assume first that $x$ is either the first or the last number in $\oneU$ (which need not to be different). Then
\begin{align}\label{eq:mult1}
\partial_{U^-}\bwedge \partial_{U^+} &\igual \begin{cases}
\partial_{U[x]^-}\bwedge \partial_{U(x)^+} &\text{ if $\ind_U(x)$ is odd and the last number in $\oneU$}\\
\partial_{U(x)^-}\bwedge \partial_{U[x]^+} &\text{ if $\ind_U(x)$ is even and the last number in $\oneU$}
\end{cases} \\ \label{eq:mult2}
\partial_{U^+}\bwedge \partial_{U^-} &\igual \begin{cases}
\partial_{U[x]^-}\bwedge \partial_{U(x)^+} &\text{ if $\ind_U(x)$ is even and the first number in $\oneU$}\\
\partial_{U(x)^-}\bwedge \partial_{U[x]^+} &\text{ if $\ind_U(x)$ is odd and the first number in $\oneU$}
\end{cases}
\end{align}
In every other case not treated in \eqref{eq:mult1} or \eqref{eq:mult2}, write $l(x)$ for the element in $\oneU$ that precedes $x$ and $r(x)$ for the element in $\oneU$ that succeeds $x$. Then we have that 
\begin{align*}
\oneU^+_{<l(x)}\cup \oneU^-_{>x}&\subset U(l(x))^+ & \oneU^+_{<l(x)}\cup \oneU^-_{>x}&\subset U(x)^+\\
\oneU^-_{<l(x)}\cup \oneU^+_{>x}&\subset U(l(x))^- & \oneU^-_{<l(x)}\cup \oneU^+_{>x}&\subset U(x)^-
\end{align*}
 and from \eqref{eq:85} it follows that
\begin{align*}
U(l(x))^+\smallsetminus (\oneU^+_{<l(x)}\cup \oneU^-_{>x}) &=\begin{cases}
x & \text{if $x\in U^+$} \\
\emptyset & \text{if $x\in U^-$}
\end{cases} \\
U(x)^+\smallsetminus (\oneU^+_{<l(x)}\cup \oneU^-_{>x})&=\begin{cases}
\emptyset & \text{if $l(x)\in U^+$} \\
l(x) & \text{if $l(x)\in U^-$}
\end{cases} \\
U(l(x))^-\smallsetminus (\oneU^-_{<l(x)}\cup \oneU^+_{>x})&=\begin{cases}
\emptyset & \text{if $x\in U^+$} \\
x & \text{if $x\in U^-$}
\end{cases} \\
U(x)^-\smallsetminus (\oneU^-_{<l(x)}\cup \oneU^+_{>x})&=\begin{cases}
l(x) & \text{if $l(x)\in U^+$} \\
\emptyset & \text{if $l(x)\in U^-$}.
\end{cases}
\end{align*}
As a consequence,
\begin{align*}
(a)&& U[l(x)]^- &= U[x]^-& U(l(x))^+ &=U(x)^+&\text{ if $l(x)\in U^-$ and $x\in U^+$} \\
(b)&& U(l(x))^- &=U[x]^- & U[l(x)]^+ &= U(x)^+&\text{ if $l(x)\in U^+$ and $x\in U^+$} \\
(c)&& U[l(x)]^- &=U(x)^- & U(l(x))^+ &= U[x]^+&\text{ if $l(x)\in U^-$ and $x\in U^-$} \\ 
(d)&& U(l(x))^- &=U(x)^- & U[l(x)]^+ &= U[x]^+&\text{ if $l(x)\in U^+$ and $x\in U^-$} 
\end{align*}
and therefore, all the terms of \eqref{eq:non_repeated} indexed by $U$ that do not appear in \eqref{eq:mult1} or \eqref{eq:mult2} are paired as follows:
\begin{align*}
\partial_{U[x]^-}\bwedge \partial_{U(x)^+} =
\begin{cases}
\partial_{U[l(x)]^-}\bwedge \partial_{U(l(x))^+} & \text{if $l(x)\in U^-$ and $x\in U^+$ (by $(a)$)}\\
\partial_{U[r(x)]^-}\bwedge \partial_{U(r(x))^+} & \text{if $x\in U^-$ and $r(x)\in U^+$ (by $(a)$)} \\
\partial_{U(l(x))^-}\bwedge \partial_{U[l(x)]^+} & \text{if $l(x)\in U^+$ and $x\in U^+$ (by $(b)$)} \\
\partial_{U(r(x))^-}\bwedge \partial_{U[r(x)]^+} & \text{if $x\in U^-$ and $r(x)\in U^-$ (by $(c)$)}
\end{cases}
\end{align*}
and 
\begin{align*}
\partial_{U(x)^-}\bwedge \partial_{U[x]^+} = 
\begin{cases}
\partial_{U[l(x)]^-}\bwedge \partial_{U(l(x))^+} & \text{if $l(x)\in U^-$ and $x\in U^-$ (by $(c)$)} \\
\partial_{U[r(x)]^-}\bwedge \partial_{U(r(x))^+} & \text{if $x\in U^+$ and $r(x)\in U^+$ (by $(b)$)} \\
\partial_{U(l(x))^-}\bwedge \partial_{U[l(x)]^+} & \text{if $l(x)\in U^+$ and $x\in U^-$ (by $(d)$)} \\
\partial_{U(r(x))^-}\bwedge \partial_{U[r(x)]^+} & \text{if $x\in U^+$ and $r(x)\in U^-$ (by $(d)$).}
\end{cases}
\end{align*}\normalsize 
Iterating twice these equalities returns the same summand, so every summand except the ones treated in \eqref{eq:mult1} and \eqref{eq:mult2} appears twice. As a consequence, $\eqref{eq:non_repeated}$ is $\bF_2$-equivalent to the sum of \eqref{eq:mult1} and \eqref{eq:mult2}.
\end{proof}

\subsection{Summing everything up} For each $U\in \APower_{q+1}(n)$, let $\cS(U)$ be the sum of those summands labeled by $U$ in Propositions \ref{prop:diagonalpart} and \ref{prop:nondiagonalpart}. Then the equality \eqref{eq:proof} that we want to prove translates into
\begin{equation}\label{eq:proof2}
\sum_{U\in \APower_{q+1}(n)}{\cS(U)} \igualdos \sum_{U\in \APower_{q+1}(n)}\partial_{U^-}\wedge\partial_{U^+} + \partial_{U^+}\bwedge\partial_{U^-}.
\end{equation}
Using the simplifications of Lemmas \ref{lemma:repeated} and \ref{lemma:non_repeated}, together with the identity 
\[\partial_U\bwedge_\udiamond\partial_V + \partial_U\bwedge_\bdiamond\partial_V = \partial_U\bwedge_\diamond\partial_V,\]
the summand indexed by $U$ in the left hand-side of \eqref{eq:proof2} becomes
\begin{align*}
\label{eq:Anibal}\cS(U) &\igualdos \partial_{U^-}\bwedge\partial_{U^+} + \partial_{U^+}\bwedge\partial_{U^-} &\text{if $\twoU=\emptyset$}\\
\cS(U) &\igual \partial_{U^-}\bwedge_\diamond\partial_{U^+} &\text{if $\oneU=\emptyset$} \\
\cS(U) &\igualdos  \partial_{U^-}\bwedge_\diamond\partial_{U^+} + \partial_{U^-}\bwedge\partial_{U^+} + \partial_{U^+}\bwedge\partial_{U^-} &\text{if $\twoU\neq\emptyset$ and $\oneU\neq\emptyset$}.
\end{align*}

%
%
If $\twoU=\emptyset$, then $\partial_{U^-}\bwedge\partial_{U^+} = \partial_{U^-}\times\partial_{U^+} = \partial_{U^-}\wedge\partial_{U^+}$ so \eqref{eq:proof2} holds. Otherwise, \eqref{eq:proof2} becomes
\begin{align*}
\partial_{U^-}\bwedge_\diamond\partial_{U^+} &\igualdos \partial_{U^-}\wedge\partial_{U^+} + \partial_{U^+}\bwedge\partial_{U^-} & \text{if $\oneU=\emptyset$}
\\
\partial_{U^-}\bwedge_\diamond\partial_{U^+} + 
\partial_{U^-}\bwedge\partial_{U^+} + \partial_{U^+}\bwedge\partial_{U^-} &\igualdos \partial_{U^-}\wedge\partial_{U^+} + \partial_{U^+}\bwedge\partial_{U^-} & \text{if $\oneU\neq\emptyset$} 
\end{align*}
We may now cancel both appearances of $\partial_{U^+}\bwedge\partial_{U^-}$ in the second line, whereas in the first line,  $U^- = U = U^+$, so $\partial_{U^+}\bwedge\partial_{U^-} = \partial_{U^-}\bwedge\partial_{U^+}$. Therefore we are left with the same equation in both cases:
\begin{align*}
\partial_{U^-}\bwedge_\diamond\partial_{U^+} + 
\partial_{U^-}\bwedge\partial_{U^+} + \partial_{U^-}\wedge\partial_{U^+} &\igualdos\emptyset.
\end{align*}
We will prove that this equation holds in Proposition \ref{prop:final}, finishing the proof of the theorem.

\subsection{Colored graphs}\label{ssection:coloredgraphs}
Let $\Gamma$ be a graph (i.e., a simplicial complex of dimension $1$) with set of vertices $V(\Gamma)$ and set of edges $E(\Gamma)$.

\begin{df}\label{df:coloredgraph} A \emph{$k$-coloring} on $\Gamma$ is a function $\epsilon\colon E(\Gamma)\to \{1,\ldots,k\}$ such that 
 each vertex has $k$ incident edges, all labeled differently.
If $v\in V(\Gamma)$ and $i=1,\ldots,k$, let $\edge{v}{i}$ be the unique edge incident to $v$ colored by $i$. If $(\Gamma,\epsilon)$ and $(\Gamma',\epsilon')$ are $k$-colored graphs, a simplicial map $f\colon \Gamma\to \Gamma'$ is \emph{color-preserving} if for each edge $e$ of $\Gamma$ such that $f(e)$ is an edge, $\epsilon'(f(e))=\epsilon(e)$.
\end{df}
\begin{lemma}\label{lemma:degree} Let $f\colon \Gamma\to \Gamma'$ be a color-preserving map between two $k$-colored graphs $(\Gamma,\epsilon)$ and $(\Gamma',\epsilon')$ and let $v,v'$ be vertices of $\Gamma'$. If $\Gamma'$ is connected, then 
\[|V(f^{-1}(v))|\equiv |V(f^{-1}(v'))|\mod 2.\]
\end{lemma}
\begin{proof}
Let $\sd \Gamma$ be the barycentric subdivision of $\Gamma$. Its vertices are the vertices of $\Gamma$ together with a new vertex $\flat_e$ at the barycenter of each edge $e$ of $\Gamma$. There is an edge $e_v$ between an old vertex $v$ and a barycenter $\flat_{e}$ whenever the edge $e$ is incident to $v$ in $\Gamma$.

Define the following simplicial map $g\colon \sd \Gamma\to \sd \Gamma'$:
\begin{align*}
g(v) &= f(v) \\
g(\flat_e) &= \begin{cases}
\flat_{f(e)} & \text{if $f(e)$ is an edge} \\
\flat_{\edge{f(e)}{\epsilon(e)}} & \text{if $f(e)$ is a vertex}
\end{cases}
\end{align*}
whose action on edges is $g(e_v) = \edge{f(v)}{\epsilon(e)}_{f(v)}$. It is straightforward to check that this map is well-defined and that
\begin{enumerate}
\item for each $v\in \Gamma$, $g(v) = f(v)$ and the image of a barycenter is a barycenter.
\item\label{itg:2} for each $v'\in \Gamma'$, $V(g^{-1}(v')) = V(f^{-1}(v'))$,
\item\label{itg:3} for each $v\in \Gamma$, the restriction $g_{|\operatorname{star}(v)}\colon \operatorname{star}(v)\to \operatorname{star}(g(v))$ is a simplicial isomorphism.
\end{enumerate}
Let $v'_0,v'_1$ be two vertices in $\Gamma'$. Since $\Gamma'$ is connected, there is an embedded simplicial path $\gamma\subset \Gamma'$ from $v'_0$ to $v'_1$. We claim that the inverse image $g^{-1}(\gamma)$ is also a simplicial manifold of dimension $1$ whose boundary is $g^{-1}(v'_0)\cup g^{-1}(v'_1)$. Once this is proven, the lemma follows from \eqref{itg:2} and the fact that the boundary of a manifold of dimension $1$ has even cardinality. 

By \eqref{itg:3}, the star of an old vertex $v$ in $g^{-1}(\gamma)$ is isomorphic to the star of $g(v)$ in $\gamma$, therefore $g^{-1}(\gamma)$ is locally the boundary or the interior of a $1$-manifold at its old vertices depending on whether they belong or not to $g^{-1}(v)\cup g^{-1}(v')$. 

Note now that if a barycenter belongs to $\gamma$, then its star in $\sd\Gamma'$ belongs to $\gamma$ as well. Therefore, if a barycenter belongs to $g^{-1}(\gamma)$, then its star in $\sd\Gamma$ belongs to $g^{-1}(\gamma)$ as well. Since $\sd\Gamma$ is already locally the interior of a $1$-manifold at its barycenters, the same holds for $g^{-1}(\gamma)$.
\end{proof}

\subsection{Application}\label{ssection:application} Take $k=r-1$, and, for each $(s,t)\in \partial_{U^-}\times \partial_{U^+}$ construct a $(r-1)$-colored graph $\Gamma(s,t)$ whose vertices are the maximal chains of $(s,t)$-good pairs and there is an edge between two different maximal chains $\{(W_j^\shortparallel,W_j^\smallwedge)\}_{j=1}^r$ and $\{(V_j^\shortparallel,V_j^\smallwedge)\}_{j=1}^r$ if there is an $i\in 1\ldots	r-1$ such that $(W_j^\shortparallel,W_j^\smallwedge)=(V_j^\shortparallel,V_j^\smallwedge)$ for all $j\neq i$. We label such an edge with the color $i$.

\begin{lemma}\label{lemma:gammacolor} $\Gamma(s,t)$ is a $(r-1)$-colored graph.
\end{lemma}
\begin{proof}
We need to prove that every chain $\{(W_j^\shortparallel,W_j^\smallwedge)\}_{j=1}^r$ has a unique incident edge labeled by $i$. Assume first that this edge exists and let the other end be the chain $\{(V_j^\shortparallel,V_j^\smallwedge)\}_{i=1}^r$. Then, we have that 
\begin{align}\label{eq:998}
V_j &= W_j \quad \text{for all $j\neq i$}.
\end{align} Write 
\begin{equation*}
\begin{aligned}[c]
w_{i+1} &= W_{i+1}^\smallwedge\smallsetminus W_{i}^\smallwedge \\
v_{i+1} &=  V_{i+1}^\smallwedge\smallsetminus V_{i}^\smallwedge 
\end{aligned}
\quad 
\begin{aligned}[c]
w_i &= W_{i}^\smallwedge\smallsetminus W_{i-1}^\smallwedge \\
 v_i &= V_i^\smallwedge\smallsetminus  V_{i-1}^\smallwedge.
\end{aligned}
\end{equation*}
Then, because of \eqref{eq:998} and because we assume that both chains are different, we have necessarily that $v_i = w_{i+1}$ and $w_{i} = v_{i+1}$. Therefore, we have that 
\[
V_i^\smallwedge = V_{i-1}^\smallwedge\cup \{v_i\},\quad V_i^\shortparallel = \begin{cases} V_{i-1}^\shortparallel & \text{ if $v_i\notin V_{i-1}$} \\
V_{i-1}^\shortparallel\smallsetminus \{v_i\} & \text{ if $v_i\in V_{i-1}$.}
\end{cases}
\] 
Therefore, if it exists, then it is uniquely given by this formula. To see that it exists, we have to see that all $V_j$'s are $(s,t)$-good. That is clear for all $j\neq i$ because $W_j$ is $(s,t)$-good, and also for $j=i$, because of Lemma \ref{lemma:parallelicity}.
\end{proof}

Let $I(k)$ be the $1$-skeleton of the $k$-dimensional cube, whose set of vertices is $\map(\{1,\ldots,k\},\{0,1\})$ and there is an edge from a vertex $v$ to a vertex $u$ if there is an $i\in 1\ldots r$ such that $v[j]=u[j]$ for all $j\neq i$. Such an edge is labeled with the color $i$. Let $f\colon \Gamma(s,t)\to I(r-1)$ be the color-preserving map that sends a vertex $(W_j^\shortparallel,W_j^\smallwedge)\}_{j=1}^r$ to its positiveness function:
\begin{equation*}
f(\{(W_j^\shortparallel,W_j^\smallwedge)\}_{j=1}^r)[i] = \begin{cases}
0 & \text{if the pair $(W_i^\shortparallel,W_i^\smallwedge)$ is positive} \\
1 & \text{if the pair $(W_i^\shortparallel,W_i^\smallwedge)$ is negative}
\end{cases},\quad i=1,\ldots,r-1.
\end{equation*}

\begin{proposition}\label{prop:final} Let $U\in \APower_q(n)$ be a sequence with $\oneU\neq U$. Then 
\begin{align*}
\partial_{U^-}\bwedge_\diamond\partial_{U^+} + 
\partial_{U^-}\bwedge\partial_{U^+} + \partial_{U^-}\wedge\partial_{U^+} &\igualdos\emptyset. 
\end{align*}
\end{proposition}
\begin{proof}

The left hand-side is, by definition, 
\begin{equation}\label{eq:4}
\sum_{(s,t)\in \partial_{U^-}\times \partial_{U^+}} \cO_{(s,t)}(\twoU)_\diamond^- + \cO_{(s,t)}(\twoU)^- + \cO_{(s,t)}(\twoU)^+
\end{equation}
and using the $k$-coloring of Lemma \ref{lemma:gammacolor}, 
\begin{align*}
\cO_{(s,t)}(\twoU)_\diamond^- &= V(f^{-1}(1,1,\ldots,1,1)) \\
\cO_{(s,t)}(\twoU)^+ &= \begin{cases} V(f^{-1}(0,0,\ldots,0,0)) &\text{ if $(s,t)$ is positive}, \\
\emptyset &\text{ if $(s,t)$ is negative},
\end{cases} \\
\cO_{(s,t)}(\twoU)^- &= \begin{cases} V(f^{-1}(1,1,\ldots,1,1)) &\text{ if $(s,t)$ is negative}, \\
\emptyset & \text{ if $(s,t)$ is positive}
\end{cases}
\end{align*}
Hence the sum
\begin{equation}\label{eq:3}
\cO_{(s,t)}(\twoU)_\diamond^- + \cO_{(s,t)}(\twoU)^- + \cO_{(s,t)}(\twoU)^+
\end{equation}
is the union of the inverse image under $f$ of two vertices of $I(r-1)$. Therefore, by Lemma \ref{lemma:degree}, we have that the cardinality of both preimages have the same parity, hence \eqref{eq:3} is even, and so \eqref{eq:4} is $\bF_2$-equivalent to the empty span.
\end{proof}

\section{Steenrod squares}\label{section:well-defined}

If $\alpha$ is a cocycle in a cochain complex $C^*$ of $\bF_2$-modules with a symmetric multiplication $\smile_i\colon C^*\otimes C^*\to C^*$, then it follows from Condition \eqref{eq:2} of the definition of stable symmetric multiplication in page \pageref{eq:2} that $\alpha\smile_i\alpha$ is also a cocycle and that if $\alpha$ and $\beta$ are cohomologous, then $\alpha\smile_i\alpha$ is cohomologous to $\beta\smile_i\beta$. Therefore, we obtain for each $i\geq 0$ a well-defined operation
\begin{equation}\label{eq:squares}
\Sq^i\colon  H^n(C^*)\lra H^{n+i}(C^*),\quad [\alpha]\mapsto [\alpha\smile_{n-i}\alpha]
\end{equation}
which is called the \emph{$i$-th Steenrod square} of $C^*$. 

Recall now from Definition \ref{df:suspension}, the suspension of an \sob. 


\begin{proposition}\label{prop:suspension} If $\alpha,\beta\in C^*(X_\bullet;\bF_2)$ and $\Sigma (\alpha),\Sigma (\beta)\in C^*(\Sigma X_\bullet;\bF_2)$ denote their suspensions, then $\Sigma(\alpha\smile_i \beta) = \Sigma(\beta)\smile_{i+1}\Sigma(\alpha)$. As a consequence, if $[\alpha]\in H^*(X_\bullet;\bF_2)$, then $\Sigma\Sq^i([\alpha]) = \Sq^i\Sigma([\alpha])$.
\end{proposition}
\begin{proof} Let $\nabla_i$ denote the symmetric comultiplication on $C^*(X_\bullet;\bF_2)$ and let $\bar{\nabla}_i$ denote the symmetric comultiplication on $C^*(\Sigma X_\bullet;\bF_2) = \Sigma C^*(X_\bullet;\bF_2)$. Then
\[(\Sigma\otimes \Sigma)\circ\nabla^{\binom{n}{q}}=T\cdot \bar{\nabla}^{\binom{n+1}{q}}\circ \Sigma\]
because, writing $\bpartial_U$ for the $U$th generalised face map of $\Sigma X_\bullet$,
\begin{align*}
\sum_{U\in \APower_q(n+1)}\bpartial_{U^-}\wedge \bpartial_{U^+} &= 
\sum_{U\in \APower_q(n)}\bpartial_{\psi_0(U)^-}\wedge \bpartial_{\psi_0(U)^+}= 
\sum_{U\in \APower_q(n)}\partial_{U^+}\bwedge \partial_{U^-},
\end{align*}
where the first equality holds because the span $\bpartial_U$ is non-empty only if $U$ is in the image of $\psi_0$, and the second equality holds because if $U=\{u_1,\ldots,u_q\}$, then $\psi_0(U)=\{u_1+1,\ldots,u_q+1\}$, therefore the index of every entry changes by one. Additionally, the sum $n+|U^+|+|U^-|$ used to define the positivity of a $(s,t)$-good pair in Definition \ref{df:good} also changes by one.
\end{proof}
\begin{proposition} The first Steenrod square is the Bockstein homomorphism.
\end{proposition}
\begin{proof} Let $X_\bullet$ be an \osob and let $d^i$ be the \abelianisation{\bZ} of $\partial^{n+1}_i$, whose coefficients are 
\[d^i(y) = \sum_{x\in X_n} d^i_{x,y} x.\] 
Define 
\begin{align*}
d_{x,y} &= \sum_{i} d^i_{x,y},&
d^\even_{x,y} &= \sum_{i~\even} d^i_{x,y}, &
d^\odd_{x,y} &= \sum_{i~\odd} d^i_{x,y}, &
\end{align*}
so that if $\bar{\alpha}\in C^n(X_\bullet;\bZ)$ is a cocycle, then
\[\delta(\bar{\alpha}) = \sum_{x\in X_n} \bar{\alpha}(x)d_{x,y}\cdot y^*.\]

The Bockstein homomorphism
\[\beta\colon H^n(X_\bullet;\bF_2)\lra H^{n+1}(X_\bullet;\bF_2)\]
is defined on a cohomology class represented by a cocycle $\alpha$ as follows: define a cochain $\bar{\alpha}$ in $C^n(X_\bullet;\bZ)$ by $\bar{\alpha}(x) = \alpha(x)$ and observe that, as $\alpha$ is a cocycle with $\bF_2$-coefficients, the coboundary of $\bar{\alpha}$ is divisible by $2$, i.e., for every $y\in X_{n+1}$, the number 
\[d_{\alpha,y}=\sum_{\alpha(x)=1} d_{x,y}\]
is divisible by $2$. Define the class $\beta([\alpha])\in H^{n+1}(X_\bullet;\bF_2)$ as
\begin{equation}\label{eq:bockstein}
\beta([\alpha]) = \left[\sum_y\frac{d_{\alpha,y}}{2} y^*\right].
\end{equation}

On the other hand, the first Steenrod square of $\alpha$ is
\[\Sq^1(\alpha) = \alpha\smile_{n-1} \alpha.\]
The product $\smile_{n-1}$ is dual to $\nabla_{n-1}$, and its restriction to $C^n(X_\bullet;\bF_2)\otimes C^n(X_\bullet;\bF_2)$ is dual to the homomorphism $\nabla^{\binom{n+1}{2}}[1,1]$ defined as 
\[\bF_2\langle X_{n+1}\rangle \overset{\nabla^{\binom{n+1}{2}}}{\lra} 
\begin{array}{c}
\bF_2\langle X_{n+1}\rangle\otimes \bF_2\langle X_{n-1}\rangle \\
\oplus\\
 \bF_2\langle X_{n}\rangle\otimes \bF_2\langle X_{n}\rangle \\
\oplus \\
 \bF_2\langle X_{n-1}\rangle\otimes \bF_2\langle X_{n+1}\rangle
\end{array}\lra  \bF_2\langle X_{n}\rangle\otimes \bF_2\langle X_{n}\rangle,\]
This homomorphism is given by
\begin{align*}
\nabla_{n-2}[1,1](y) &= \btoab{\bF_2}{\undersumtwo{U\in \APower_{2}(n+1)}{U^+\neq\emptyset\neq U^-}{\partial_{U^-}\wedge\partial_{U^+}}}(y)  \\
&= \btoab{\bF_2}{\undersumtwo{U\in \APower_{2}^0(n+1)}{U^+\neq\emptyset\neq U^-}{\partial_{U^-}\wedge \partial_{U^+}} + \sum_{U\in \APower_{2}^1(n+1)} \partial_{U^-}\wedge\partial_{U^+}}(y) \\
&= \btoab{\bF_2}{\undersumtwo{0\leq u<v\leq n+1}{u,v \text{ even}} \partial_{u}\times \partial_{v} + \undersumtwo{ 0\leq u<v\leq n+1}{u,v \text{ odd}}{\partial_{v}\times \partial_{u}} + \sum_{0\leq u\leq n+1} \partial_u\wedge \partial_u}(y)\\
&=  \undersumtwo{0\leq u<v\leq n+1}{u,v \text{ even}}{\sum_{x,x'\in X_n} (d^u_{x,y}d^v_{x',y})\cdot x\otimes x'} +\\
&+ \undersumtwo{ 0\leq u<v\leq n+1}{u,v \text{ odd}}{\sum_{x,x'\in X_n}(d^u_{x,y}d^v_{x',y})\cdot x\otimes x'} +\\ &+ \sum_{0\leq u\leq n+1}{\sum_{x<x'}(d^u_{x,y}d^u_{x',y})\cdot x\otimes x'} + \sum_{0\leq u\leq n+1}\sum_x\binom{d^u_{x,y}}{2}x\otimes x
\end{align*}
whose dual takes $\alpha$ to
\begin{align*}
&\sum_y\left(\undersumtwo{0\leq u<v\leq n+1}{u,v \text{ even}}{\undersumtwo{x,x'\in X_n}{\alpha(x)=\alpha(x')=1}{(d^u_{x,y}d^v_{x',y})}} + \undersumtwo{ 0\leq u<v\leq n+1}{u,v \text{ odd}}{\undersumtwo{x,x'\in X_n}{\alpha(x)=\alpha(x')=1}{(d^u_{x,y}d^v_{x',y})}}\right. +\\
 &+ \left.\sum_{0\leq u\leq n+1}{\undersumtwo{x<x'}{\alpha(x)=\alpha(x')=1}{(d^u_{x,y}d^u_{x',y})}} + \sum_{0\leq u\leq n+1}\sum_{\alpha(x)=1}\binom{d^u_{x,y}}{2}\right)\cdot y^*
\end{align*}
which, rearranging
\begin{align*}
&\sum_y\left(\undersumtwo{x,x'\in X_n}{\alpha(x)=\alpha(x')=1}{\undersumtwo{0\leq u\leq v\leq n+1}{u,v \text{ even}}{d^u_{x,y}d^v_{x',y}}}  + \sum_{\alpha(x)=1} \undersumtwo{0\leq u\leq n+1}{u~\even}{\binom{d^u_{x,y}}{2}}\right.+
\\
&+\left.\undersumtwo{x,x'\in X_n}{\alpha(x)=\alpha(x')=1}{\undersumtwo{ 0\leq u\leq v\leq n+1}{u,v \text{ odd}}{d^u_{x,y}d^v_{x',y}}} + \sum_{\alpha(x)=1} \undersumtwo{0\leq u\leq n+1}{u~\odd}{\binom{d^u_{x,y}}{2}}\right)\cdot y^*
\end{align*}
and using that $\binom{a+b}{2} = ab+\binom{a}{2}+\binom{b}{2}$, this becomes
\begin{align*}
& \sum_y\left(\binom{d^\even_{\alpha,y}}{2} + \binom{d^\odd_{\alpha,y}}{2}\right)\cdot y^*
\end{align*}
Let $(d^\even_{\alpha,y})_0$ and $(d^\even_{\alpha,y})_1$ be the first two digits in the binary expansion of $d^\even_{\alpha,y}$ and and let $(d^\odd_{\alpha,y})_0$ and $(d^\odd_{\alpha,y})_1$ be the first two digits in the binary expansion of $d^\odd_{\alpha,y}$. As $\alpha$ is a cocycle with $\bF_2$ coefficients, $\delta(\alpha)$ is divisible by $2$, and therefore $(d^\even_{\alpha,y})_0 = (d^\odd_{\alpha,y})_0$ for all $y\in X_{n+1}$, hence the coefficient in \eqref{eq:bockstein} is
\[\frac{d_{\alpha,y}}{2} = \frac{d^\even_{\alpha,y} - d^{\odd}_{\alpha,y}}{2} \equiv (d^\even_{\alpha,y})_1 + (d^\odd_{\alpha,y})_1\mod 2.\] 
Finally, by Lucas' Theorem, we have that 
\begin{align*}
(d^\even_{\alpha,y})_1 &= \binom{d^\even_{\alpha,y}}{2} 
&
(d^\odd_{\alpha,y})_1 &= \binom{d^\odd_{\alpha,y}}{2}.
\end{align*}
\end{proof}

In the remaining of this section we prove a Cartan formula for these Steenrod squares. The natural product operation on \sobs is the join, that we introduce now. If $U\in \APower(n)$, and $m\leq n$, define the following sequences in $\APower(m)$ and $\APower(n-m-1)$:
\begin{align*}
U_{\leq m} &= \{u\in U\mid u\leq m\} & U_{> m} &= \{u\geq 0\mid u+m+1\in U\}.
\end{align*}
If $n=n_1+n_2+1$, then there is a bijection $P(n)\to P(n_1)\times P(n_2)$ given by sending a sequence $U$ to the pair of sequences $(U_{\leq \m},U_{>\m})$. If $u\in U_{\leq \m}$, then its index in $U_{\leq \m}$ coincides with its index in $U$, whereas if $u\in U_{> \m}$, then its index in $U_{> \m}$ and its index in $U$ coincide if $\m$ and $|U_{\leq \m}|$ have different parity, and are opposite if $\m$ and $|U_{\leq \m}|$ have the same parity, therefore, writing $U_{\leq \m}^\pm$ for $\left(U_{\leq \m}\right)^\pm$ and $U_{> \m}^\pm$ for $\left(U_{> \m}\right)^\pm$, we have
\begin{equation}\label{eq:cases}
\begin{aligned}
U_{\leq \m}^-&= (U^-)_{\leq \m } &  U_{>\m}^- &= \begin{cases} (U^-)_{>\m} & \text{if $\m+ |U_{\leq \m}|$ is odd} \\ (U^+)_{>\m} & \text{if $\m+ |U_{\leq \m}|$ is even} \end{cases} 
\\
U_{\leq \m}^+&= (U^+)_{\leq \m} & U_{>\m}^+ &= \begin{cases} (U^+)_{>\m}& \text{if $\m+ |U_{\leq \m}|$ is odd} \\ (U^-)_{>\m} & \text{if $\m+ |U_{\leq \m}|$ is even.}  \end{cases}
\end{aligned}
\end{equation}

\begin{df}\label{df:join} The join product  of two \sobs $X_\bullet,Y_\bullet$ is the \sob $(X* Y)_{\bullet}$ given by
\begin{align*}
(X* Y)_n &= \coprod_{n_1+n_2=n-1}X_{n_1}\times Y_{n_2}\\
\bpartial_U^n &= \coprod_{n_1+n_2=n-1}\partial_{U_{\leq \m}}^{n_1}\times \partial^{n_2}_{U_{>\m}},\\
\bmu^{n}_{V,W} &= \coprod_{n_1+n_2=n-1} \mu^{n_1}_{V_{\leq \m},W_{\leq \m}}\times \mu^{n_2}_{V_{>\m},W_{> \m}}.
\end{align*}
\end{df}
\begin{notation} During the rest of this section, we will write $\partial^n_U$ and $\mu^n_{V,W}$ for the structural maps of $X_\bullet$ and $Y_\bullet$ and $\bpartial^n_U$ and $\bmu^n_{V,W}$ for the structural maps of the join $X_\bullet*Y_\bullet$.
\end{notation}
There is a canonical isomorphism 
\begin{equation}\label{eq:chainjoin}
C^*((X*Y)_\bullet;R)\cong \Sigma \left( C^*(X_\bullet;R)\otimes C^*(Y_\bullet;R)\right),
\end{equation}
therefore, when $R=\bF_2$, the K\"unneth formula gives an isomorphism
\begin{equation}\label{eq:cohjoin}H^*((X*Y)_\bullet;\bF_2)
\cong \Sigma \left(H^*(X_\bullet;\bF_2)\otimes H^*(Y_\bullet;\bF_2)\right).\end{equation} 
\begin{remark} The product of two cubes in the Burnside category (\cite[Definition 5.4]{LLS-cube}, \cite[Def.~4.20]{LLS2015}) 
\begin{align*}
F\colon& \cube{n_1}\to \cB,&G\colon& \cube{n_2}\to \cB
\end{align*} 
is another cube $F\times G\colon \cube{n_1+n_2}\to \cB$ in the Burnside category. If $$|\cdot|\colon \preshcube{\cube{n}}{\cB_{\mathrm{fin}}}\to \Sp$$ denotes the realisation functor \cite{LLS2015} from the finite Burnside category (see Warning in page \pageref{warning-burnside}) to the category of spectra, then \cite[Prop. 4.23]{LLS2015}
\[|F\times G| \simeq |F|\wedge |G|.\]
It is easy to check that $\Lambda(F\times G) = \Lambda(F)* \Lambda(G)$, and one could prove that, writing again $|\cdot |_{\bS}\colon \presh{\Deltainjaug}{\cB}\to \Sp$ for the functor $*$ in Diagram \ref{eq:picture}.
\begin{equation}\label{eq:smash}
|(X*Y)_\bullet|_{\bS} \simeq \Sigma\left(|X_\bullet|_{\bS}\wedge |Y_\bullet|_{\bS}\right),
\end{equation}
though this computation is beyond the scope of this paper.
\end{remark}
\begin{df} If $X_\bullet$ and $Y_\bullet$ are ordered, then we order $(X*Y)_\bullet$ lexicographically, i.e., if $(s_1,s_2)$ and $(t_1,t_2)$ are elements in $\partial^n_U$, and 
\begin{align*}
(s_1,s_2)&\in \partial^{n_1}_{U_{\leq \m}}\times  \partial^{n_2}_{U_{>\m}} & (t_1,t_2)&\in \partial^{n'_1}_{U_{\leq\m'}}\times  \partial^{n'_2}_{U_{>\m'}},
\end{align*}
 then $(s_1,s_2)<(t_1,t_2)$ if and only if one of the following holds:
\begin{align*}
n_1&<n'_1 & &\text{or} & n_1&= n'_1, s_1<t_1 & &\text{or} & n_1&= n'_1, s_1=t_1, s_2<t_2. 
\end{align*}
Elements of $\partial^{n_1}_{U_{\leq \m}}$ are called \emph{strong} and elements of $\partial^{n_2}_{U_{> \m}}$ are called \emph{weak}.
\end{df}

Let us write $\{\bnabla_k\}$ for the stable symmetric comultiplication on $X_\bullet\times Y_\bullet$ and $\{\nabla_k\}$ for the stable symmetric comultiplications on $X_\bullet$ and $Y_\bullet$.

\begin{proposition}\label{prop:cartan} Let $n=n_1+n_2+1$ and let $q\geq 0$. The span 
\begin{equation*}
\bnabla^{\binom{n}{q}}|_{X_{n_1}\times Y_{n_2}} = \sum_{U\in \APower_q(n)} \bpartial^n_{U^-}\wedge\bpartial^n_{U^+}
\end{equation*} 
is $\bF_2$-equivalent to:
\begin{align*}
\sum_{q_1+q_2=q}\sum_{r_1,r_2}\sum_{U_1\in \APower^{r_1}_{q_1}(n_1)}\sum_{U_2\in \APower^{r_2}_{q_2}(n_2)} &\left(\partial^{n_1}_{U^-_{1}}\wedge \partial^{n_1}_{U^+_{1}}\right)\times T^{q_1+n_1+r_1+1}\left(\partial^{n_2}_{U^-_2}\wedge \partial^{n_2}_{U^+_{2}}\right)
\end{align*}
\end{proposition}
So, in general, it is not true that if $\alpha,\alpha'\in C^*(X_\bullet;\bF_2)$ and $\beta,\beta'\in C^*(Y_\bullet;\bF_2)$, then
\[\alpha\otimes\beta\smile_i \alpha'\otimes\beta' = \sum_{j=0}^i (\alpha\smile_j\alpha')\otimes(\beta\smile_{i-j}\beta'),\]
so the symmetric comultiplication of Theorem \ref{thm:main} does not satisfy a Cartan formula. But if we forget the twist $T^{\m+q_1+r_1+1}$, Proposition \ref{prop:cartan} becomes
\begin{equation}\label{eq:cartan1}
\bnabla^{\binom{n}{q}}|_{X_{n_1}\times Y_{n_2}}\igualdos \sum_{q_1+q_2= q}{\nabla^{\binom{n_1}{q_1}}\times \nabla^{\binom{n_2}{q_2}}},\end{equation}
and since the twist is irrelevant if $\alpha=\alpha'$ and $\beta=\beta'$, we do have a Cartan formula for the squaring operations:
\[\alpha\otimes\beta\smile_i \alpha\otimes\beta = \sum_{j=0}^i (\alpha\smile_j\alpha)\otimes(\beta\smile_{i-j}\beta).\]

\begin{corollary}\label{cor:cartan} If $\alpha\otimes\beta$ is a cohomology class in $H^*((X*Y)_\bullet;\bF_2)$ via the isomorphism \eqref{eq:cohjoin}, then $$\Sq^i(\alpha\otimes\beta) = \sum_{j=0}^i \Sq^j(\alpha)\otimes \Sq^{i-j}(\beta).$$
\end{corollary}

\subsection{Proof of Proposition \ref{prop:cartan}}

There is a bijection between the indexing sets of both summations of the proposition given by sending $U\in \APower_q(n)$ to 
\begin{align*}
U_1 &= U_{\leq \m} &  q_1 &= |U_1| & r_1 &= |\twoU_1|,
U_2 &= U_{>\m}     &  q_2 &= |U_2| & r_2 &= |\twoU_2|.
\end{align*}
Therefore the proposition will follow if for every $U\in \APower^r_q(n)$ with $|\twoU_{\leq\m}|=r_1$,
\begin{equation}\label{eq:cartan2}
\bpartial^n_{U^-}\wedge \bpartial^n_{U^+} \igualdos \left(\partial^{n_1}_{U^-_{\leq \m}}\wedge \partial^{n_1}_{U^+_{\leq \m}}\right)\times T^{q_1+n_1+r_1+1}\left(\partial^{n_2}_{U^-_{> \m}}\wedge \partial^{n_2}_{U^+_{> \m}}\right),\end{equation}
which, developing each term, is the same as
\begin{eqnarray*} 
\sum_{(s,t)\in \partial^n_{U^-}\times \partial^n_{U^+}}\cO_{s,t}(\twoU)^+ \igualdos
\undersumtwo{(s_1,t_1)\in \partial^{n_1}_{U^-_{\leq \m}}\times \partial^{n_1}_{U^+_{\leq \m}}}{(s_2,t_2)\in \partial^{n_2}_{U^-_{> \m}}\times \partial^{n_2}_{U^+_{> \m}}}{\cO_{s_1,t_1}(\twoU_{\leq \m})^+\times \cO_{s_2,t_2}(\twoU_{> \m})^\pm,} 
\end{eqnarray*}
where the sign $\pm$ is positive if $q_1+n_1+r_1$ is odd and negative otherwise. Therefore, it will be enough to prove that 
 for each $s,t\in \partial^n_{U^-}\times \partial^n_{U^+}$,
\begin{equation}\label{eq:cartan3}|\cO_{s,t}(\twoU)^+|\equiv |\cO_{s_1,t_1}(\twoU_{\leq \m})^+\times \cO_{s_2,t_2}(\twoU_{> \m})^\pm|\mod 2.\end{equation}
From now on we assume that $q_1+n_1$ is odd, so that \eqref{eq:cases} gives
\begin{align*}
U_{\leq \m}^-&= (U^-)_{\leq \m } &  U_{>\m}^- &= (U^-)_{>\m} 
\\
U_{\leq \m}^+&= (U^+)_{\leq \m} & U_{>\m}^+ &= (U^+)_{>\m}.
\end{align*} 
The proof of the other case is formally the same.
\begin{df} Let $X_\bullet$ be an \osob, let $U\in \APower_q(n)$, and let $(s,t)\in \partial^n_{U^-}\times \partial^n_{U^+}$. Say that a pair $\W = (W^\shortparallel,W^\smallwedge)$ of disjoint subsets of $\twoU$ is \emph{$(s,t)$-parallel} if 
\begin{align*}
\lambda^n_{W^\shortparallel}(s) &= \lambda^n_{W^\shortparallel}(t) &
\lambda^n_{W^\shortparallel,W^\smallwedge}(s) &= \lambda^n_{W^\shortparallel,W^\smallwedge}(t).
\end{align*}
Note that being $(s,t)$-parallel is downwards closed: if $\W\prec\W'$ and $\W'$ is $(s,t)$-parallel, then $\W$ is $(s,t)$-parallel too. Note also that if $\lambda^n_{W^\shortparallel}(s) = \lambda^n_{W^\shortparallel}(t)$ and $W^\circ$ is empty, then $\W$ is $(s,t)$-parallel. 
\end{df}

Let $(s,t)\in \partial^n_{U^-}\times \partial^n_{U^+}$ with $s=(s_1,s_2)$ and $t=(t_1,t_2)$. An $(s,t)$-good pair $\W=(W^\shortparallel,W^\smallwedge)$ decomposes as
\begin{align*}
W^\shortparallel &= (W^\shortparallel)_{\leq \m}\cup (W^\shortparallel)_{> \m} & W^\circ &= (W^\circ)_{\leq \m}\cup (W^\circ)_{> \m},
\end{align*}
and we refer to the pairs 
\begin{align*}
\W_{\leq \m} &:= ( (W^\shortparallel)_{\leq \m},(W^\circ)_{\leq \m}) & \W_{> \m} &:= ( (W^\shortparallel)_{> \m},(W^\circ)_{> \m})
\end{align*}
as the \emph{strong} and \emph{weak} parts of the pair $\W$. Since the spans of $(X*Y)_\bullet$ are ordered lexicographically, a pair $\W=(W^\shortparallel,W^\smallwedge)$ is $(s,t)$-positive if either of the following holds:
\begin{equation}\label{1}
\begin{array}{l}
\text{$\W_{\leq \m}$ is $(s_1,t_1)$-positive and $\W_{> \m}$ is either $(s_2,t_2)$-positive,}\\ \text{or $(s_2,t_2)$-negative or $(s_2,t_2)$-parallel.} \\[.15cm]
\text{$\W_{\leq \m}$ is $(s_1,t_1)$-parallel and $\W_{> \m}$ is $(s_2,t_2)$-positive.}
\end{array}
\end{equation}

Let $r=|\twoU|, r_1 = |\twoU_{\leq \m}|, r_2= |\twoU_{>\m}|$, so in particular $r=r_1+r_2$. If $\WW$ is a maximal chain of $(s,t)$-good pairs
$$\W_1\prec\ldots\prec \W_r,$$  then for each $i\leq r$, the set $W_i^\smallwedge\smallsetminus W_{i-1}^\smallwedge$ has a single element (see Lemma \ref{lemma:parallelicity} and comment afterward), which we call $w_i$. Say that the pair $\W_i$ is \emph{strong in $\WW$} if $w_i\leq \m$ and that it is \emph{weak in $\WW$} if $w_i>\m$. Say that the pair $\W_i$ is \emph{\semiparallel, \semipositive} or \emph{\seminegative} if $\W_i$ is strong (weak) and the strong (weak) part of $\W_i$ is parallel, positive or negative.

Let $\cO_{s,t}(\twoU)^+_{=}\subset \cO_{s,t}(\twoU)^+$ be the subset of those maximal chains of $(s,t)$-positive pairs that contain a \semiparallel pair, and let $\cO_{s,t}(\twoU)^+_{\neq}$ be its complement.
\begin{lemma}
The subset $\cO_{s,t}(\twoU)^+_{=}\subset \cO_{s,t}(\twoU)^+$ has even cardinality.
\end{lemma}
\begin{proof}
If $\WW\in \cO_{s,t}(\twoU)^+_{=}$, let $i\leq r$ be the smallest index such that 
$\W_i$ is \semiparallel. This index cannot be $1$ because in that case $\W_1$ would not be $(s,t)$-positive. On the other hand, if $\W_i$ is strong (alternatively, weak), then $i$ is also the smallest index such that $\W_i$ is strong (alternatively, weak), because the property of being parallel is downwards closed. Observe now that the following are equivalent,
\begin{enumerate}
\item $w_i\in W^\shortparallel_j$ for all $j<i$,
\item $w_i\in W^\shortparallel_j$ for some $j<i$,
\end{enumerate}
and define another maximal chain $\overline{\WW}\in \cO_{s,t}(\twoU)^+_{=}$ as follows:
\begin{align*}
\overline{W}_j^\smallwedge &= W_j^\smallwedge &\text{ for all $j$} 
\\
\overline{W}_j^\shortparallel &= W_j^\shortparallel &\text{ for all $j\geq i$}
\\
\overline{W}_j^\shortparallel &= \begin{cases} 
W_j^\shortparallel\smallsetminus \{w_i\} &\text{if $w_i\in W_{1}^\shortparallel$} \\
W_j^\shortparallel\cup \{w_i\} &\text{if $w_i\notin W_{1}^\shortparallel$}
\end{cases}
&\text{ for all $j< i$}
\end{align*}
This defines a free involution on $\cO_{s,t}(\twoU)^+_{=}$, hence this set has even cardinality.
\end{proof}
\begin{lemma} The cardinality of the subset $\cO_{s,t}(\twoU)^+_{\neq}\subset \cO_{s,t}(\twoU)^+$ has the same parity as the cardinality of $\cO_{s_1,t_1}(\twoU_{\leq \m})^+\times \cO_{s_2,t_2}(\twoU_{> \m})^\pm$, where the sign $\pm$ is positive if $r_1$ is even and negative if $r_1$ is odd.
\end{lemma}
\begin{proof}
Let us construct first an $(r-1)$-colored graph $\Theta_{r_1,r_2}$ (see Definition \ref{df:coloredgraph}). Its vertices are pairs of functions $(\varphi,\nu)$, where
\begin{align*}
\varphi\colon \{1,\ldots,r\}&\lra \{\weak,\strong\} &
\nu\colon \{1,\ldots,r\}&\lra \{0,1\}
\end{align*}
and $|\varphi^{-1}(\strong)|=r_1$ and $|\varphi^{-1}(\weak)|=r_2$. The $i$th incident edge to the vertex $(\varphi,\nu)$ connects $(\varphi,\nu)$ with $(\bar{\varphi},\bar{\nu})$, where
\begin{enumerate}
\item If $\varphi(i)=\varphi(i+1)$, then 
\begin{align*}
\bar{\varphi}(j)&=\varphi(j)  \text{ for all $j$},\\ \bar{\nu}(i)&\neq \nu(i),  \\ \bar{\nu}(j)&=\nu(j) \text{ if $j\neq i$}.
\end{align*}
\item If $\varphi(i)\neq \varphi(i+1)$, then  
\begin{align*}
\bar{\varphi}(i) &= \varphi(i+1),& \bar{\varphi}(i+1)&= \varphi(i),& \bar{\varphi}(j)&=\varphi(j)\text{ if $j\neq i,i+1$},\\
\bar{\nu}(i)&=\nu(i+1),& \bar{\nu}(i+1)&=\nu(i), & \bar{\nu}(j)&=\nu(j) \text{ if $j\neq i,i+1$}.
\end{align*}
\end{enumerate}
Observe that $\Theta_{r_1,r_2}$ has $4$ connected components: if we write 
\begin{align*}
\maxstrong{\varphi} &= \max \{i\mid \varphi(i)=\strong\}, & \maxweak{\varphi} &= \max\{i\mid \varphi(i)=\weak\},\\
\minstrong{\varphi} &= \min \{i\mid \varphi(i)=\strong\}, & \minweak{\varphi} &= \min\{i\mid \varphi(i)=\weak\},
\end{align*}
then the vertices of $(\varphi,\nu)$ and $(\bar{\varphi},\bar{\nu})$ are in the same component if and only if 
\begin{align*}
\nu(\maxstrong{\varphi}) &= \bar{\nu}(\maxstrong{\bar{\varphi}}) &
\nu(\maxweak{\varphi}) &= \bar{\nu}(\maxweak{\bar{\varphi}}).
\end{align*}
Recall the definition of the graph $\Gamma(s,t)$ from Section \ref{ssection:application}. The vertices of $\Gamma(s,t)$ that are in $\cO_{s,t}(\twoU)_{\neq}$ span a connected component of $\Gamma(s,t)$ that we call $\Gamma(s,t)_{\neq}$. There is a color-preserving graph map
\[f\colon \Gamma(s,t)_{\neq}\lra \Theta_{r_1,r_2}\]
that sends a non-\semiparallel $(s,t)$-good maximal chain $\WW$ to the vertex $(\varphi,\nu)$ with 
\begin{align*}
\varphi(i)&= \begin{cases}
\strong & \text{if $\W_i$ is strong} \\
\weak & \text{if $\W_i$ is weak}
\end{cases}
&
\nu(i)&= \begin{cases}
0 & \text{if $\W_i$ is \semipositive} \\
1 & \text{if $\W_i$ is \seminegative}
\end{cases}
\end{align*}
Let $\Theta^+_{r_1,r_2}$ be the set of those vertices of $\Theta_{r_1,r_2}$ such that
\begin{enumerate}
\item $\nu(j)=0$ if $\varphi(j)=\strong$, 
\item $\nu(j)=0$ for all $j<\minstrong{\varphi}$.
\end{enumerate} 
Then, using \eqref{1}, we have that
\begin{equation}\label{eq:sol1}
V\left(f^{-1}(\Theta^+_{r_1,r_2})\right) = \cO_{s,t}(\twoU)^+.
\end{equation}

If we let 
\[\Theta^{+,\ell}_{r_1,r_2} = \{(\varphi,\nu)\in \Theta^+_{r_1,r_2}\mid \minstrong{\varphi}=\ell\},\]
and we let 
\begin{align*}
\almaxweak{\varphi} &= \max_i\{\varphi(i)=\weak,i\neq \maxweak{\varphi}\}
\end{align*}
then $\almaxweak{\varphi} > \minstrong{\varphi}$ if and only if $\minstrong{\varphi}<r_2$. If $\ell>r_2+1$, then $\Theta^{+,\ell}_{r_1,r_2}$ is empty. When $\ell<r_2$, construct a free involution of $\Theta^{+,\ell}_{r_1,r_2}$ by sending a vertex $(\varphi,\nu)$ to the vertex $(\bar{\varphi},\bar{\nu})$ with
\begin{align*}
\bar{\varphi}(j)&= \varphi(j) \text{ for all $j$}
\\
\bar{\nu}(j) &\neq \nu(j) \text{ if $j= \almaxweak{\varphi}$}
\\
\bar{\nu}(j) &= \nu(j)\text{ if $j\neq \almaxweak{\varphi}$.}
\end{align*}
When $\ell=r_2$, the set $\Theta^{+,r_2}_{r_1,r_2}$ has $2r_1$ elements, classified by the pair $(\maxweak{\varphi}-r_2,\nu(\maxweak{\varphi}))\in \{1,\ldots,r_1\}\times \{0,1\}$, and given by
\begin{align*}
\varphi(i)&=\begin{cases}
\weak &\text{if $i<r_2$ or $i=\maxweak{\varphi}$} \\
\strong &\text{if $i\geq r_2$ and $i\neq \maxweak{\varphi}$}
\end{cases} \\
\nu(i) &= \begin{cases} 
0 &\text{if $i\neq \maxweak{\varphi}$} \\
0\text{ or }1 & \text{if $i=\maxweak{\varphi}$}.
\end{cases}
\end{align*}
On the other hand, $\Theta^{+,r_2+1}_{r_1,r_2}$ has a single element given by 
\begin{align*}
\varphi(i)&=\begin{cases}
\weak &\text{if $i\leq r_2$} \\
\strong &\text{if $i> r_2$}
\end{cases} &
\nu(i) &= 0 \text{ for all $i$.}
\end{align*}
Define an involution on $\Theta^{+,r_2}_{r_1,r_2}\cup \Theta^{+,r_2+1}_{r_1,r_2}$ as follows: if $r_1$ is even, then send an element $(\varphi,\nu)$ of $\Theta^{+,r_2}_{r_1,r_2}$ classified by $(\maxweak{\varphi}-r_2,\nu(\maxweak{\varphi}))$ to the element $(\bar{\varphi},\bar{\nu})$ of $\Theta^{+,r_2}_{r_1,r_2}$ with 
\begin{align*}
\maxweak{\bar{\varphi}} &= 2\cdot \left\lceil\frac{n_\varphi}{2}\right\rceil
&
\nu(\maxweak{\bar{\varphi}}) &= \nu(n_{\varphi})
\end{align*}
and send the unique element of $\Theta^{+,r_2+1}_{r_1,r_2}$ to itself. Let us denote this latter element by $a$. If $r_1$ is odd, then send an element $(\varphi,\nu)$ of $\Theta^{+,r_2}_{r_1,r_2}$ classified by $(\maxweak{\varphi}-r_2,\nu(\maxweak{\varphi}))$ with $\maxweak{\varphi}>r_2+1$ to the element $(\bar{\varphi},\bar{\nu})$ of $\Theta^{+,r_2}_{r_1,r_2}$ with 
\begin{align*}
\maxweak{\bar{\varphi}} &= 2\cdot \left\lfloor \frac{n_\varphi}{2}\right\rfloor
&
\nu(\maxweak{\bar{\varphi}}) &= \nu(n_{\varphi})
\end{align*}
and send the element classified by $(r_2+1,0)$ to the unique element of $\Theta^{+,r_2+1}_{r_1,r_2}$, and send the element classified by $(r_2+1,1)$ to itself. Let us denote this latter element by $b$.

Altogether, we have defined an involution on $\Theta_{r_1,r_2}^+$ with a unique fixed point, and such that $(\varphi,\nu)$ and $(\bar{\varphi},\bar{\nu})$ lie in the same connected component of $\Theta_{r_1,r_2}$. 

Hence, by Lemma \ref{lemma:degree}, we have that if $(\varphi,\nu)\in \Theta_{r_1,r_2}^+$, then
\[\left|V\left(f^{-1}(\varphi,\nu)\right)\right|\equiv \left|V\left(f^{-1}(\bar{\varphi},\bar{\nu})\right)\right|\mod 2,\]
therefore 
\begin{equation}\label{eq:sol2}
\left|V\left(f^{-1}(\Theta^{+}_{r_1,r_2})\right)\right|\equiv \begin{cases}
\left|V\left(f^{-1}(a)\right)\right|\mod 2 \text{  if $r_1$ is even}\\
\left|V\left(f^{-1}(b)\right)\right|\mod 2  \text{  if $r_1$ is odd.}\end{cases}
\end{equation}
Moreover, the vertex $b$ lies in the same connected component of $\Theta_{r_1,r_2}$ as the vertex $c$ defined as
\begin{align*}
\varphi(i)&=\begin{cases}
\weak &\text{if $i\leq r_2$} \\
\strong &\text{if $i> r_2$}
\end{cases} &
\nu(i) &= 1 \text{ for all $i$}
\end{align*}
and therefore 
\begin{equation}\label{eq:sol2b}
\left|f^{-1}(b)\right|\equiv\left|f^{-1}(c)\right|\mod 2.
\end{equation} Finally, there are bijections
\begin{equation}\label{eq:sol3}
\begin{aligned}
f^{-1}(a) &\lra   \cO_{s_1,t_1}(\twoU_{\leq \m})^+\times \cO_{s_2,t_2}(\twoU_{> \m})^+\\
f^{-1}(c) &\lra   \cO_{s_1,t_1}(\twoU_{\leq \m})^+\times \cO_{s_2,t_2}(\twoU_{> \m})^-
\end{aligned}
\end{equation}
given by sending a maximal chain $\W_1\prec\ldots\prec\W_r$ to the pair of maximal chains
\begin{align*}
\W_{r_2+1}&\prec\ldots\prec\W_{r} & \W_{1}&\prec\ldots\prec\W_{r_2}.
\end{align*}
The statement now follows from \eqref{eq:sol1}, \eqref{eq:sol2}, \eqref{eq:sol2b} and \eqref{eq:sol3}.
\end{proof}

\section{Naturality}\label{section:functoriality}

The \cupi{i}products constructed are not natural under maps of \sobs. This is no surprise, because the realisation functor $|\cdot|_{\bF_2}\colon \cB^{\Deltainjaug^{\op}}\to \Ch(\bF_2)$ is not faithful (compare to the case of semi-simplicial sets, where the functor is faithful). The following is a simple situation in which naturality fails:
\begin{example}
Let $X_\bullet$ be given by $X_{-1}=\{a\}, X_0=\{b\}$ and $X_i=\emptyset$ otherwise, and let $\partial^0_0$ be the span $\{b\}\la Q\to \{a\}$ with $Q=\{q_1,q_2\}$ and $q_1<q_2$. Let $Y_\bullet$ be given by $Y_{-1}=\{c\}, Y_0=\{d\}$ and $Y_i=\emptyset$ otherwise, and let $\bpartial^0_0$ be the span $\{d\}\la P \to \{c\}$ with $P=\{p_1,p_2,p_3\}$ and $p_1<p_2<p_3$. We claim that $\smile_{-2}$ is not natural for any non trivial map from $X_\bullet$ to $Y_\bullet$.

Any such map $f\colon X_\bullet\to Y_\bullet$ will consist on spans
\begin{align*}
\{b\}&\la S\to \{d\} & \{a\}&\la T\to \{c\}
\end{align*}
and some bijection between $S\times P$ and $Q\times T$. If $f$ is non-trivial (i.e., $S$ and $T$ are non-empty), then $S$ has cardinality $2$ and $T$ has cardinality $3$. 
To compute $\smile_{-2}$ in $X_\bullet$ and $Y_\bullet$ we start with their duals:
\begin{align*}
\nabla_{-2}(b) &= \btoab{\bF_2}{\partial^0_0\wedge\partial^0_0\circ \Delta}(b) &
\nabla_{-2}(d) &= \btoab{\bF_2}{\bpartial^0_0\wedge\bpartial^0_0\circ\Delta}(d) 
\end{align*}
and as the spans $\partial_0^0\wedge\partial_0^0$ and $\bpartial_0^0\wedge\bpartial_0^0$ are
\begin{align*}
\{b\}&\longleftarrow \{(q_1,q_2)\}\lra \{a\} &
\{d\}&\longleftarrow \{(p_1,p_2),(p_1,p_3),(p_2,p_3)\}\lra \{c\}
\end{align*}
we have that
\begin{align*}
\nabla_{-2}(b) &= |\{(q_1,q_2)\}|\cdot a = 1\cdot a = a \\
 \nabla_{-2}(d) &= |\{(p_1,p_2),(p_1,p_3),(p_2,p_3)\}|\cdot c = 3\cdot c = c
\end{align*}
and therefore, writing $a^*,b^*,c^*,d^*$ for the duals of $a,b,c,d$, we have
\begin{align*}
a^*\smile_{-2} a^* &= b^* & c^*\smile_{-2} c^* &=  d
\end{align*}
and hence 
\begin{align*}
f^*(c^*\smile_{-2} c^*) &= f^*(d^*)= 2b^*=0\\
f^*(c^*)\smile_{-2}f^*(c^*)&=a^*\smile_{-2}a^* = b^*.
\end{align*}
\end{example}

\subsection{Naturality of the \cupi{i}products}
Let $X_\bullet,Y_\bullet$ be a pair of \sobs and let $f\colon X_\bullet\to Y_\bullet$ be a map as in Section \ref{ssection:sobmaps}, so the face maps in $X_\bullet$ will be denoted $\partial^n_U$ and the face maps in $Y_\bullet$ will be denoted $\bpartial^n_U$.
\begin{df}\label{def:freemap}  The map $f$ is \emph{free} if for each $n\geq -1$, $f_n\colon X_n\to Y_n$ is a free span.
\end{df}

Now, if $X\overset{\source}{\lla} Q\overset{\target}{\lra} Y$ is a locally finite span $Q$ and $x\in X$ and $y\in Y$, define the \emph{restriction set $Q|_{x,y}$} as $\source^{-1}(x)\cap \target^{-1}(y)$. In order to check whether two spans $Q,Q'$ from $X$ to $Y$ are equivalent, it is enough to check that for each $x\in X$ and each $y\in Y$, the restriction sets $Q_{x,y}$ and $Q'_{x,y}$ are in bijection.

Let us now endow both $X_\bullet$ and $Y_\bullet$ with an order, and endow the spans $ f_{n-q}\circ\partial^n_U$ and $\bpartial^n_U\circ f_n $ with the partial order induced by the projection maps
\begin{align*}
f_{n-q}\circ\partial^n_U&\lra \partial^n_U &
\bpartial^n_U\circ f_n&\lra \bpartial^n_U.
\end{align*}
If $x\in X_n$ and $y\in Y_{n-q}$, then this partial order restricts to a total order on the restriction sets
\begin{align}\label{eq:alex}
&f_{n-q}\circ\partial^n_U|_{x,y} & &\bpartial^n_U\circ f_n|_{x,y}.
\end{align}
\begin{df} 
A free map from $X_\bullet$ to $Y_\bullet$ is \emph{order-preserving} if for each $n\geq -1$ and for each $U\in \Power_q(n)$, the $2$-morphism $f^n_U\colon f_{n-q}\circ\partial^n_U\to\bpartial^n_U\circ f_n $ is the unique fibrewise bijection that restricts to an order-preserving bijection on each restriction set of \eqref{eq:alex}. 
\end{df}
\begin{proposition}\label{prop:naturality-cup} The \cupi{i}products are natural with respect to order-preserving free maps.
\end{proposition}
\begin{proof}
Let $f\colon X_\bullet\to Y_\bullet$ be such order-preserving free map and let $\alpha,\beta$ be a $p$-cochain and a $q$-cochain in $Y_\bullet$, and let $\sigma$ be a $(p+q-i)$ simplex in $X_\bullet$. Then
\begin{align*}
(f^*(\alpha\smile_i\beta))(\sigma) &=
(\alpha\otimes\beta)(\bar{\nabla}_i(f_*(\sigma))) \\
&\overset{*}{=} (\alpha\otimes\beta)\left((f_*\otimes f_*)(\nabla_i(\sigma))\right) \\
&= (f^*\alpha\otimes f^*\beta)\left(\nabla_i(\sigma)\right) \\
&= (f^*\alpha\smile_i f^*\beta)(\sigma),
\end{align*}
where every equality is formal except for $*$, which we prove true now: We claim that if $f$ is free and order-preserving, there is an equivalence of spans
\[\bar{\nabla}^{\binom{n}{q}}\circ \bar{\Delta}\circ f_n\igual(f\times f)\circ {\nabla}^{\binom{n}{q}}\circ\Delta.\]
Note first that if $f$ is a free map, then
\[\bar{\Delta}\circ f_n  = (f_n\times f_n)\circ \Delta,\]
so we have to prove,
\begin{align*}
\bar{\nabla}^{\binom{n}{q}}\circ (f_n\times f_n)\circ \Delta  &= (f\times f)\circ {\nabla}^{\binom{n}{q}}\circ\Delta
\end{align*}
that is, 
\begin{align*}
\sum_{U\in \APower_q(n)}\bpartial_{U^-}\wedge\bpartial_{U^+}\circ (f_n\times f_n)\circ \Delta  &= \sum_{U\in \APower_q(n)}(f\times f)\circ (\partial_{U^-}\wedge \partial_{U^+})\circ \Delta.
\end{align*}
 Hence, it is enough to check that for each $U\in \APower_q(n)$ with $n-|U^-|=m$ and $n-|U^+|=\ell$, each $x\in X_n$ and each $(y,y')\in Y_{m}\times Y_{\ell}$, the following sets are in bijection
\begin{align}\label{eq:f1}
 &\left.\bpartial_{U^-}\wedge\bpartial_{U^+}\circ (f_n\times f_n)\right|_{x,(y,y')} & &\left.(f_m\times f_\ell)\circ (\partial_{U^-}\wedge \partial_{U^+})\right|_{x,(y,y')}.
\end{align}
If $f$ is a free map, $U\in \Power_q(n), x\in X_n$ and $y\in Y_{n-q}$, either the projections
\begin{align}\label{eq:f}
\begin{aligned}
 \bpartial^n_U\circ \left.f_n\right|_{x,y}&\lra\  \left.\bpartial^n_U\right|_{f(x),y}
\\
f_{n-q}\circ \left.\partial^n_U\right|_{x,y}&\lra \coprod_{f(z)=(y)}\left.\partial^n_U\right|_{x,z} 
\end{aligned}
\end{align}
that forget the component of $f$ are bijections, or the domain of each projection is empty. In the second situation both sides of \eqref{eq:f} are empty, so we assume the first situation, in which case there are bijections
\begin{align}\label{eq:f2}
\varphi\colon
\left.\bpartial^n_U\right|_{f(x),y} 
\cong 
\bpartial^n_U\circ \left.f_n\right|_{x,y} 
\cong 
f_{n-q}\circ \left.\partial^n_U\right|_{x,y}
\cong
\coprod_{f(z)=(y)} \left.\partial^n_U\right|_{x,z}.
\end{align}
Under the isomorphisms \eqref{eq:f}, each side of \eqref{eq:f1} becomes isomorphic to 
\begin{align*}&\undercoprodthree{(s,t)\in \bpartial^n_{U^-}\times \bpartial^n_{U^+}}{\source(s)=\source(t)=f(x)}{\target(s)=y, \target(t)=y'}{\cO_{s,t}(\twoU)}& &\undercoprodthree{(s,t)\in \partial^n_{U^-}\times \partial^n_{U^+}}{\source(s)=\source(t)=x}{f_m(\target(s))=y, f_\ell(\target(t))=y'}{\cO_{s,t}(\twoU)} \end{align*}
Now, the bijection $\varphi$ induces an isomorphism between these two sets because 
\begin{itemize}
\item a pair $(W^\shortparallel,W^\smallwedge)$ is $(s,t)$-good if and only if it is $(\varphi(s),\varphi(t))$-good. 
\item a pair $(W^\shortparallel,W^\smallwedge)$ is $(s,t)$-positive if and only if it is $(\varphi(s),\varphi(t))$-positive. 
\end{itemize}
These two assertions follow from diagrams \eqref{eq:sobmaps_lambda1} and \eqref{eq:sobmaps_lambda2}, using that $f$ is order-preserving and noting that, after taking restriction sets, the upper and bottom rows in those diagrams are the bijections \eqref{eq:f2}.
\end{proof}
\begin{corollary}\label{cor:nat_free} If $f\colon X_\bullet\to Y_\bullet$ is an order-preserving free map of \osobs, then $f^*\Sq^i = \Sq^if^*$.
\end{corollary}
\subsection{A mapping cylinder construction} Recall that if $f\colon C_*\to C'_*$ is a homomorphism of chain complexes of $R$-modules, the mapping cylinder of $f$ is the chain complex $M(f)_*$ with $M(f)_n = C_{n-1}\oplus C_n\oplus C'_n$ and differential $d''$ given by 
\[d''(x,y,z) = (-d(x),d(y)+x,d'(c)-f(x)).\]
This chain complex comes with maps
\[C_*\overset{i_f}{\lra} M(f)_*\overset{g_f}{\lra} C'_*\]
given by including $C_*$ as the second summand and by sending an element $(x,y,z)$ to $f(y)+z$. Moreover, the last map has a homotopy inverse $h_f\colon C'_*\to M(f)_*$ given by including $C'_*$ as the third summand, and therefore $g_f$ is a chain homotopy equivalence. Additionally, $f=g_f\circ i_f$.

This construction has the following counterpart $M(\Sigma^2 f)_\bullet$ for the double suspension of a map $f\colon X_\bullet\to Y_\bullet$ of \sobs(again, we use the notation of Section \ref{ssection:sobmaps}): $M(\Sigma^2 f)_\bullet$ is the \sob whose set of $n$-simplices is 
\[M(\Sigma^2 f)_n = X_{n-3}\cup X_{n-2}\cup Y_{n-2}. \]
Write $\psi_\cerouno$ and $\psi_\cerodos$ for $\psi_{\{0,1\}}$ and $\psi_{\{0,1,2\}}$. If $U\in \Power_q(n)$, the generalised face map $\bbpartial^n_U$ of $M(\Sigma^2 f)_\bullet$ is the sum of the following three spans (note that the second and third spans are precisely the generalised face maps of $\Sigma^2 X_\bullet$ and $\Sigma^2 Y_\bullet$):
\begin{align*}
\bbpartial^n_U|_{X_{n-3}} &= \begin{cases}
\partial^{n-3}_{\psi_\cerodos(U)} & \text{if $\{0,1,2\}\cap U=\emptyset$} \\
\emptyset & \text{if $2\in U$} \\
\partial^{n-3}_{\psi_{\cerodos}(U\smallsetminus \{0\})} & \text{if $0\in U$ and $\{1,2\}\cap U=\emptyset$} \\
\bpartial^{n-3}_{\psi_{\cerodos}(U\smallsetminus \{1\})}\circ f_{n-3} &\text{if $1\in U$ and $\{0,2\}\cap U=\emptyset$}\\
\emptyset & \text{if $\{0,1\}\subset U$}
\end{cases} 
\\
\bbpartial^n_U|_{X_{n-2}} &= \begin{cases}
\partial^{n-2}_{\psi_\cerouno(U)} & \text{if $\{0,1\}\cap U=\emptyset$} \\
\emptyset & \text{if $0\in U$ or $1\in U$}
\end{cases}
\\
\bbpartial^n_U|_{Y_{n-2}} &= \begin{cases}
\bpartial^{n-2}_{\psi_\cerouno(U)} & \text{if $\{0,1\}\cap U=\emptyset$} \\
\emptyset & \text{if $0\in U$ or $1\in U$}
\end{cases}
\end{align*}

Let $U=V_1\cup V_2$ with $U\in \Power_q(n)$ and $V_1\in \Power_{p}(n)$, and write as usual $W_2=\psi_{V_1}(V_2)$ and observe that
\begin{align}\label{pescadilla1}
\psi_{\psi_A(B)}(\psi_A(C)) &= \psi_A(\psi_B(C)),
\end{align}
and if $i\in \{0,1\}$ and $U\cap \{0,1\}=\{i\}$, then, depending on whether $i\in V_1$ or $i\in V_2$, we have:
\begin{align}\label{pescadilla2}
\psi_{\cerouno}(\psi_{V_1}(V_2)) = \psi_{\cerodos}(\psi_{V_1\smallsetminus \{i\}}(V_2))
\\ \label{pescadilla3}
\psi_{\cerodos}(\psi_{V_1}(V_2)\smallsetminus\{i\}) = \psi_{\cerodos}(\psi_{V_1}(V_2\smallsetminus \{i\})).
\end{align}
Define the structural $2$-morphisms 
\begin{equation}\label{eq:mu}
\bar{\bar{\mu}}_{V_1,V_2}\colon \bbpartial^n_U\to \bbpartial^{n-p}_{W_2}\circ \bbpartial^n_{V_1}
\end{equation}
of $M(\Sigma^2 f)_\bullet$ as follows:
\begin{enumerate}
\item if $2\in U$ and $\{0,1\}\cap U\neq \emptyset$, then at least two of the three spans are empty, therefore both sides of \eqref{eq:mu} are empty, so $\bar{\bar{\mu}}^n_{V_1,V_2}$ is the unique function from the empty set to itself.
\item if $2\in U$ and $\{0,1\}\cap U = \emptyset$, then
\begin{align*}
\bar{\bar{\mu}}^n_{V_1,V_2} &= \mu^{n-2}_{\psi_\cerouno(V_1),\psi_\cerouno(V_2)} \cup \bmu^{n-2}_{\psi_\cerouno(V_1),\psi_\cerouno(V_2)}
\end{align*}
\item if $2\notin U$, and $\{0,1\}\cap U=\emptyset$ then
\begin{align*}
\bar{\bar{\mu}}^n_{V_1,V_2} &= \mu^{n-3}_{\psi_\cerodos(V_1),\psi_\cerodos(V_2)} \cup \mu^{n-2}_{\psi_\cerouno(V_1),\psi_\cerouno(V_2)} \cup \bmu^{n-2}_{\psi_\cerouno(V_1),\psi_\cerouno(V_2)}
\end{align*}
\item if $2\notin U$ and $1\notin U$ and $0\in V_1$, then
\begin{align*}
\bbpartial^n_U &= \partial^{n-3}_{\psi_\cerodos(U\smallsetminus \{0\})}
\\
\bbpartial^n_{V_1} &= \partial^{n-3}_{\psi_{\cerodos}(V_1\smallsetminus \{0\})} 
\\
\bbpartial^{n-p}_{W_2} &= \partial^{n-p-3}_{\psi_\cerodos(W_2)} \cup \partial^{n-p-2}_{\psi_\cerouno(W_2)} \cup \bpartial^{n-p-2}_{\psi_\cerouno(W_2)}
\end{align*}
and since the target of the span $\bbpartial^n_{V_1}$ is contained in $X_{n-p-2}$ (the second factor of $M(\Sigma^2 f)_{n-p}$), we have that
\[\bbpartial^{n-p}_{W_2}\circ \bbpartial^n_{V_1}=\partial^{n-p-2}_{\psi_\cerouno(W_2)}\circ \partial^{n-3}_{\psi_{\cerodos}(V_1\smallsetminus\{0\})}\]
and setting $U'=\psi_\cerodos(U\smallsetminus \{0\}), V_1' = \psi_{\cerodos}(V_1\smallsetminus \{0\}), V_2'=\psi_\cerodos(V_2)$, define \eqref{eq:mu} as follows, using the identities \eqref{pescadilla1} and \eqref{pescadilla2}: 
\[\mu^{n}_{V_1',V_2'}\colon \partial^{n-3}_{U'}\lra \partial^{n-p-2}_{W_2'}\circ \partial^{n-3}_{V_1'}\]

\item if $2\notin U$, $1\notin U$ and $0\in V_2$, then
\begin{align*}
\bbpartial^n_U &= \partial^{n-3}_{\psi_\cerodos(U\smallsetminus \{0\})}
\\
 \bbpartial^n_{V_1} &=\partial^{n-3}_{\psi_\cerodos(V_1)} \cup \partial^{n-2}_{\psi_\cerouno(V_1)} \cup \bpartial^{n-2}_{\psi_\cerouno(V_1)}  \\
 \bbpartial^{n-p}_{W_2} &= \partial^{n-p-3}_{\psi_\cerodos(W_2\smallsetminus \{0\})}
\end{align*}
and since the source of the span $\bbpartial^{n-p}_{W_2}$ is contained in $X_{n-p-3}$ (the first factor of $M(\Sigma^2 f)_{n-p}$), we have that
\[\bbpartial^{n-p}_{W_2}\circ \bbpartial^n_{V_1}=\partial^{n-p-3}_{\psi_\cerodos(W_2\smallsetminus \{0\})}\circ \partial^{n-3}_{\psi_\cerodos(V_1)}\]
and setting $U'=\psi_\cerodos(U\smallsetminus \{0\}), V_1' = \psi_\cerodos(V_1), V_2'=\psi_\cerodos(V_2\smallsetminus\{0\})$, define \eqref{eq:mu} as follows, using the identities \eqref{pescadilla1} and \eqref{pescadilla3}:
\[\mu^{n}_{V_1',V_2'}\colon \partial^{n-3}_{U'}\lra \partial^{n-p-2}_{W_2'}\circ \partial^{n-3}_{V_1'}\]

\item if $2\notin U$, $0\notin U$ and $1\in V_1$, then
\begin{align*}
\bbpartial^n_U &= \bpartial^{n-3}_{\psi_\cerodos(U\smallsetminus \{1\})}\circ f_{n-3}      \\ 
\bbpartial^n_{V_1} &= \bpartial^{n-3}_{\psi_\cerodos(V_1\smallsetminus \{1\})}\circ f_{n-3} \\
 \bbpartial^{n-p}_{W_2} &= \partial^{n-p-3}_{\psi_\cerodos(W_2)}\cup \partial^{n-p-2}_{\psi_\cerouno(W_2)} \cup \bpartial^{n-p-2}_{\psi_\cerouno(W_2)}
\end{align*}
and since the target of the span $\bbpartial^n_{V_1}$ is contained in $Y_{n-p-2}$ (the third factor of $M(\Sigma^2 f)_{n-p}$), we have that
\[\bbpartial^{n-p}_{W_2}\circ \bbpartial^n_{V_1} = 
\bpartial^{n-p-2}_{\psi_\cerouno(W_2)} \circ \bpartial^{n-3}_{\psi_\cerodos(V_1\smallsetminus \{1\})}\circ f_{n-3} 
\]
and setting $U'=\psi_\cerodos(U\smallsetminus \{1\}), V_1' = \psi_\cerodos(V_1\smallsetminus\{1\}), V_2'=\psi_\cerodos(V_2)$, define \eqref{eq:mu} as follows, using the identities \eqref{pescadilla1} and \eqref{pescadilla2}:
\begin{align*}
\bmu^{n-2}_{V_1',V_2'}\circ_2 \Id\colon \bpartial^{n-3}_{U'}\circ f_{n-3}&\lra \bpartial^{n-p-2}_{W_2'}\circ \bpartial^{n-3}_{V_1'}\circ f_{n-3}
\end{align*}

\item if $2\notin U$, $0\notin U$ and $1\in V_2$, then
\begin{align*}
\bbpartial^n_U         &= \bpartial^{n-3}_{\psi_\cerodos(U\smallsetminus\{1\})}\circ f_{n-3}     \\ 
\bbpartial^n_{V_1}     &= \partial^{n-3}_{\psi_\cerodos(V_1)}                \cup \partial^{n-2}_{\psi_\cerouno(V_1)}   \cup \bpartial^{n-2}_{\psi_\cerouno(V_1)} 
\\
\bbpartial^{n-p}_{W_2} &= \bpartial^{n-p-3}_{\psi_\cerodos(W_2\smallsetminus \{1\})}\circ f_{n-p-3}
\end{align*}
and since the source of the span $\bbpartial^{n-p}_{W_2}$ is contained in $X_{n-p-3}$ (the first factor of $M(\Sigma^2 f)_{n-p}$), we have that
\begin{align*}
\bbpartial^{n-p}_{W_2}\circ \bbpartial^n_{V_1} &= \bpartial^{n-p-3}_{\psi_\cerodos(W_2\smallsetminus\{1\})}\circ f_{n-p-3}\circ \partial^{n-3}_{\psi_\cerodos(V_1)}
\end{align*}
and setting $U'=\psi_\cerodos(U\smallsetminus\{1\}), V_1' = \psi_\cerodos(V_1), V_2'=\psi_\cerodos(V_2\smallsetminus\{1\})$, define \eqref{eq:mu} as follows, using the identities \eqref{pescadilla1} and \eqref{pescadilla3}:
\[
\xymatrixcolsep{1.4cm}\xymatrix{
\bpartial^{n-3}_{U'}\circ f_{n-3}\ar[r]^-{\mu_{V_1',V_2'}}& \bpartial^{n-p-3}_{W_2'}\circ \bpartial^{n-3}_{V_1'}\circ  f_{n-3} &\ar[l]_-{f^{n-3}_{V_1'}}  \bpartial^{n-p-3}_{W_2'}\circ f_{n-p-3}\circ \partial^{n-3}_{V_1'} 
}\]

\item If $2\notin U$ and either $\{0,1\}\subset V_1$ or $\{0,1\}\subset V_2$, then at least two of the three spans in \eqref{eq:mu} are the empty span, and therefore both sides are empty.
\item If $0\in V_1$ and $1\in V_2$
\begin{align*}
\bbpartial^n_U &= \emptyset & \bbpartial^n_{V_1} &= \partial^{n-3}_{\psi_\cerodos(V_1\smallsetminus \{0\})} &\bbpartial^{n-p}_{W_2} &= \bpartial^{n-p-3}_{\psi_\cerodos(W_2\smallsetminus \{1\})}\circ f_{n-p-3}
\end{align*}
and since the source of the span $\bbpartial^{n-p}_{W_2}$ is contained in $X_{n-p-3}$ while the target of the span $\bbpartial^n_{V_1}$ is contained in $X_{n-p-2}$, 
\[\bbpartial^{n-p}_{W_2}\circ\bbpartial^n_{V_1} = \emptyset\]
so $\eqref{eq:mu}$ is the unique bijection between the empty spans.
\item If $2\notin U$, $1\in V_1$ and $0\in V_2$
\begin{align*}
\bbpartial^n_U &= \emptyset & \bbpartial^n_{V_1} &= \bpartial^{n-3}_{\psi_\cerodos(V_1\smallsetminus\{1\})}\circ f_{n-3} &\bbpartial^{n-p}_{W_2} &= \partial^{n-p-3}_{\psi_{\cerodos}(W_2\smallsetminus\{0\})}
\end{align*}
and since the source of the span $\bbpartial^{n-p}_{W_2}$ is contained in $X_{n-p-3}$ while the target of the span $\bbpartial^n_{V_1}$ is contained in $Y_{n-p-2}$, 
\[\bbpartial^{n-p}_{W_2}\circ\bbpartial^n_{V_1} = \emptyset\]
so $\eqref{eq:mu}$ is the unique bijection between the empty spans.
\end{enumerate}
The verification of Conditions \eqref{sob:4}, \eqref{sob:5} and \eqref{sob:6} in the definition of \sob are not made explicit, because in each case they will be compositions of $2$-morphisms $\mu_{V_1,V_2},\bmu_{V_1,V_2}$ and $f_V$, and the diagrams will commute because $f$ satisfies \eqref{eq:sobmaps2} and $X_\bullet$ and $Y_\bullet$ satisfy the aforementioned Condition \eqref{sob:6}.

The \sobs $\Sigma^2 X_\bullet$ and $\Sigma^2 Y_\bullet$ are included in $M(\Sigma^2 f)_\bullet$ as the second and third factors. Let us denote by
\[\xymatrix{\Sigma^2 X_\bullet\ar[r]^-{i_{\Sigma^2 f}}& M(\Sigma^2 f)_\bullet& \ar[l]_-{h_{\Sigma^2 f}} Y_\bullet}\]
these inclusions. Then, by construction, the $R$-realisation of $M(\Sigma^2f)_\bullet$ will be the mapping cylinder of the $R$-realisation of $f$:
\begin{align*}
|M(\Sigma^2 f)_\bullet|_R &= M(|\Sigma^2 f|_R)_* & |i_{\Sigma^2 f}|_R &= i_{|\Sigma^2 f|_R} & |h_{\Sigma^2 f}|_R &= h_{|\Sigma^2 f|_R}
\end{align*}
in particular $h_{\Sigma^2 f}$ is an equivalence of \sobs.

\subsection{Naturality of Steenrod squares}
\begin{theorem}\label{thm:naturality-Steenrod}
If $f\colon X_\bullet\to Y_\bullet$ is a map of \sobs, and both $X_\bullet$ and $Y_\bullet$ are ordered, then $f^*\Sq^i = \Sq^if^*$.
\end{theorem}
\begin{proof} 
Let $|\cdot|$ denote the $\bF_2$-realisation $|\cdot|_{\bF_2}\colon \presh{\Deltainjaug}{\cB}\to \Ch(\bF_2)$. The diagram 
\[\xymatrix{
&M(\Sigma^2 f)_\bullet & \\
\Sigma^2 X_\bullet\ar[ur]^{i_{\Sigma^2 f}}\ar[rr]^{\Sigma^2 f} && \Sigma^2 Y_\bullet\ar[ul]_{h_{\Sigma^2 f}}
}\]
does not commute, but after taking $\bF_2$-realisations, the diagram of chain complexes
\[\xymatrix{
&|M(\Sigma^2 f)_\bullet| & \\
|\Sigma^2 X_\bullet|\ar[ur]^{|i_{\Sigma^2 f}|}\ar[rr]^{|\Sigma^2 f|} && |\Sigma^2 Y_\bullet|\ar[ul]_{|h_{\Sigma^2 f}|}
}\]
does commute up to homotopy because, by construction, this is the mapping cylinder diagram of $|\Sigma^2 f|$. Moreover, there is a chain map $g\colon |M(\Sigma^2 f)_\bullet|\to |\Sigma^2 Y_\bullet|$ that is a homotopy inverse of $|h_{\Sigma^2 f}|$, and therefore
$|\Sigma^2 f|\simeq g\circ |i_{\Sigma^2 f}|$. Let $M(\Sigma^2 f_\bullet)$ be endowed with any order extending the order of its subobjects $\Sigma^2X_\bullet$ and $\Sigma^2Y_\bullet$. Then, the inclusion $h_{\Sigma^2 f}$ is free and order-preserving, so $h_{\Sigma^2 f}$ preserves Steenrod squares, and therefore $g$ preserves Steenrod squares: for every cohomology class $\alpha\in H^*(\Sigma^2 Y_\bullet;\bF_2)$,
\begin{align*}
\Sq^i(|h_f|^*g^*\alpha) = |h_f|^*\Sq^i(g^*\alpha)\quad &\Rightarrow\quad g^*\Sq^i(|h_f|^*g^*\alpha) = g^*|h_f|^*\Sq^i(g^*\alpha)\\
&\Rightarrow\quad g^*\Sq^i(\alpha) = \Sq^i(g^*\alpha).
\end{align*}
On the other hand, since $i_{\Sigma^2 f}$ is free and order-preserving we have that $|i_{\Sigma^2 f}|^*\circ\Sq^i = \Sq^i\circ|i_{\Sigma^2 f}|^*$, and since $|\Sigma^2 f|$ is homotopy equivalent to $g\circ |i_{\Sigma^2 f}|$, we deduce that $|\Sigma^2 f|^*\circ\Sq^i = \Sq^i\circ|\Sigma^2 f|^*$. As by Proposition \ref{prop:suspension} Steenrod squares commute with suspension, we have $|f|^*\circ\Sq^i = \Sq^i\circ|f|^*$.
\end{proof}
As a consequence, if $X_\bullet$ is \ansob, the Steenrod squares $\Sq^i\colon H^n(X_\bullet;\bF_2)\to H^{n+i}(X_\bullet;\bF_2)$ are independent of the order on $X_\bullet$ used to define them, hence are well-defined natural operations on the cohomology of any \sob.

%

\section{Cubes and Khovanov homology}\label{section:khovanov}
Using the construction of Lawson, Lipshitz and Sarkar and the functor $\Lambda$, the \cupi{i}products of Theorem \ref{thm:main} may be defined on the Khovanov cochain complex associated to an oriented link diagram $D$ with ordered crossings. In this section we prove that the Steenrod squares $\Sq^i$ do not depend on the order of the crossings and are invariant under Reidemeister moves.

\subsection{Stable equivalences of stable functors} We start by reviewing the concept of stable functor \cite[Section 5]{LLS-cube}, and proving that stable equivalences of stable functors induce, after applying $\Lambda$, zig-zags of equivalences between iterated suspensions of \sobs.

\begin{df} A \emph{face inclusion of degree $r$} is a functor $\iota\colon \cube{n}\to \cube{n'}$ that is injective on objects and for every $A\subset \{1,\ldots,n\}$, $|\iota(A)|=|A|+r$. 
A face inclusion is \emph{sequential} if, writing $\bar{a}$ for the unique element of the singleton $\iota(\{a\})\smallsetminus \iota(\emptyset)$, we have that $\bar{a}<\bar{b}$ if $a<b$, and $b<\bar{1}$ if $b\in \iota(\emptyset)$. If $n=n'$, then all face inclusions have degree $0$ and are induced by permutations of $\{1,\ldots,n\}$, so we will refer to them by the associated permutation.
\end{df}

\begin{df}
If $\iota\colon \cube{n}\to \cube{n'}$ is a face inclusion, and $F\colon \cube{n}\to\cB$, then there is a unique functor $\iota_*F$ such that $F = (\iota_*F)\circ \iota$ and $\iota_*F(A)=\emptyset$ if $A\in \cube{n'}\smallsetminus \cube{n}$. This assignment is natural and defines a functor $\iota_*\colon \preshcube{\cube{n}}{\cB}\to \preshcube{\cube{n'}}{\cB}$.
\end{df}
\begin{lemma}\label{lemma:multiple} Let $\iota\colon \cube{n}\to \cube{n'}$ be a face inclusion of degree $r$ and let $F,G\colon \cube{n}\to \cB$ be cubes in the Burnside category.
\begin{enumerate} 
\item\label{cond:reordering} There is a permutation $\omega$ of $\{1,\ldots,n'\}$ such that the composition $\omega\circ \iota$ is sequential.
\item\label{cond:extra} If $\omega$ is a permutation of $\{1,\ldots,n'\}$, then $(\iota_*F)\circ\omega=(\omega\circ\iota)_*(F)$.
\item\label{cond:seqface} If $\iota$ is sequential, then $\Lambda(\iota_*F)=\Sigma^r\Lambda(F)$.
\item\label{cond:totreal} If $f\colon F\to G$ is an equivalence of cubes in the Burnside category, then $\Lambda(f)\colon \Lambda(F)\to \Lambda(G)$ is an equivalence of \sobs.
\end{enumerate}
\end{lemma}
\begin{proof}
The first two assertions are clear. The third follows because 
\[\Lambda(\iota_*F)_n = 	\Lambda(F)_{n-r}\]
and, if $\bpartial_U$ denotes the generalised face map on the left and $\partial_{U}$ the generalised face map on the right, then 
\begin{align*}
\bpartial^n_U &=
\begin{cases} 
\partial^{n-r}_{\psi_{\{0,\ldots,r-1\}}(U)} &\text{if $\{0,\ldots,r-1\}\cap U=\emptyset$} \\
\emptyset &\text{otherwise} 
\end{cases}
\\
\bmu^{n}_{V,W} &= \begin{cases}
\bmu^{n-r}_{\psi_{\{0,\ldots,r-1\}}(V),\psi_{\{0,\ldots,r-1\}}(W)} &\text{if $\{0,\ldots,r-1\}\cap (V\cup W)=\emptyset$}\\
\emptyset&\text{otherwise}
\end{cases}
\end{align*}
The fourth is a consequence of the commutativity of the right square in \eqref{eq:picture} explained in Section \ref{ssection:realisations}.
\end{proof}
The proof of the following proposition is given in Section \ref{ssection:reordering} and is based on \cite[Section 7]{Steenrod}.
\begin{proposition}\label{prop:zigzag} If $\omega$ is a permutation of $\{1,\ldots,n\}$, there are sequential face inclusions $\Phi,\Psi^\omega\colon \cube{n}\to \cube{2n}$ of degree $0$ such that for every $F\colon \cube{n}\to \cB$ there is a $G\colon \cube{2n}\to \cB$ and a zig-zag of equivalences of cubes in the Burnside category
\[\Phi_*F\lra G\lla \Psi^\omega_*(F\circ \omega).\]
%
\end{proposition}
\begin{corollary}\label{cor:seqgood} If $\iota\colon \cube{n}\to \cube{n'}$ is a face inclusion of degree $r$ and $F\colon \cube{n}\to \cB$ is a cube in the Burnside category, then there is a zig-zag of equivalences of \sobs
\[\Lambda(\iota_*F)\lra \Lambda(G)\lla \Sigma^r\Lambda(F).\]
\end{corollary}
\begin{proof} By Lemma \ref{lemma:multiple} \eqref{cond:reordering}, there is a permutation $\omega$ of $\{1,\ldots,n\}$ such that $\omega\circ\iota$ is a sequential face inclusion, and by Proposition \ref{prop:zigzag}, there is a pair of sequential inclusions $\Phi_*,\Psi^\omega_*\colon \cube{n}\to \cube{2n}$ and a $G\colon \cube{2n}\to \cB$ and a zig-zag of equivalences between cubes in the Burnside category
\[\Phi_*(\iota_* F_i)\lra G\lla \Psi^\omega_*((\iota_*F)\circ \omega).\]
Therefore, by Lemma \ref{lemma:multiple}, we have equivalences of \sobs
\begin{align*}
\Lambda(\iota_* F_i)&=\Lambda(\Phi_*(\iota_* F_i))\simeq \Lambda(G)\simeq \Lambda(\Psi^\omega_*((\iota_*F)\circ \omega)) =\\ &= \Lambda((\iota_*F)\circ \omega)=  \Lambda((\omega\circ\iota)_*F) =  \Sigma^r\Lambda(F).
\end{align*}

\end{proof}
\begin{df} A \emph{stable functor} is a triple $(F,n,r)$ where $F\colon \cube{n}\to \cB$ for some $n\geq 0$ and $r\in \bZ$. A \emph{stable map} from a stable functor $(F,n,r)$ to a stable functor $(F',n',r')$ is a pair $(f,\iota)$, where $\iota\colon \cube{n}\to \cube{n'}$ is a face inclusion of degree $r-r'$ and $f\colon \iota_*F\to F'$ is a map of cubes. The composition $(f',\iota')\circ (f,\iota)$ is $(f'\circ (\iota'_*f),\iota'\circ\iota)$. A stable map $(f,\iota)\colon (F,n,r)\to (F',n',r')$ is a \emph{stable equivalence} if $f\colon \iota_*F\to F'$ is an equivalence of cubes in the Burnside category. Two cubes in the Burnside category are \emph{stably equivalent} if there is a zig-zag of stable equivalences between them.
\end{df}
From Corollary \ref{cor:seqgood} and Lemma \ref{lemma:multiple} \eqref{cond:totreal}, we deduce
\begin{corollary}\label{cor:steq} If $(F,r)$ and $(F',r')$ are stable functors that are stably equivalent, then there exists some $k\geq 0$ such that $\Sigma^{r+k}\Lambda(F)$ and $\Sigma^{r'+k}\Lambda(F')$ are equivalent.
\end{corollary}
\subsection{Proof of Proposition \ref{prop:zigzag}}\label{ssection:reordering}


Let $F\colon \cube{n}\to \cB$ be a cube in the Burnside category, with morphisms $F(B\subset A)\colon F(A)\to F(B)$ for each $B\subset A\subset \{1,\ldots,n\}$ and higher morphisms $F(B\subset A,C\subset B)\colon F(C\subset A)\to F(C\subset B)\circ F(B\subset A)$. Let $\omega$ be a permutation of $\{1,\ldots,n\}$, which induces a maximal face inclusion $\omega\colon \cube{n}\to \cube{n}$ and let $F^\omega = F\circ\omega$. Let 
\begin{align*}
\Phi\colon \{1,\ldots,n\}&\lra \{1,\ldots,2n\} & \Phi(i)&= i \\
\Psi\colon \{1,\ldots,n\}&\lra \{1,\ldots,2n\} &\Psi(i)&=\omega^{-1}(i)+n\\
\Gamma\colon \{1,\ldots,2n\}&\lra \{1,\ldots,n\} & \Gamma(i)&= \begin{cases}
i &\text{if $i\leq n$} \\
\omega(i-n)&\text{if $i> n$}
\end{cases}
\end{align*}
and observe that $\Gamma\circ \Phi = \Id$ and $\Gamma\circ \Psi= \Id$ and that both $\Phi$ and $\Psi^\omega:=\Psi\circ\omega$ induce sequential face inclusions of degree $0$ and that $\Gamma$ induces a functor from $\cube{2n}$ to $\cube{n}$. If $A,B\subset \{1,\ldots,2n\}$, write $A<B$ (resp.\ $A\leq B$) if for every $a\in A$ and every $b\in B$, $a<b$ (resp.\ $a\leq b$). A vertex $A\subset \{0,\ldots,2n\}$ of $\cube{2n}$ is \emph{\wellordered} if $\Phi^{-1}(A)\leq \Psi^{-1}(A)$ and it is \emph{\very \wellordered} if $\Phi^{-1}(A)<\Psi^{-1}(A)$. Note that these properties are hereditary: if $B\subset A$ and $A$ is (\very) \wellordered, then $B$ is (\very) \wellordered too. Define a new cube $G\colon \cube{2n}\to \cB$ as 
\[G(A)= \begin{cases}
F\circ \Gamma(A) & \text{if $A$ is \wellordered} \\
\emptyset &\text{otherwise.}
\end{cases}
\]
with 
\begin{align*}
G(B\subset A) &= F\circ \Gamma(B\subset A) & G(B\subset A,C\subset B)&= F\circ \Gamma(B\subset A,C\subset B)
\end{align*}
whenever $A$ is \wellordered, and the empty span and the unique 2-morphism between empty spans otherwise. Since $F=G\circ \Phi$ and $F^\omega=G\circ \Psi^\omega$, we have that the sequential face inclusions $\Phi$ and $\Psi^\omega$ induce maps of cubes
\begin{align*}
f\colon \Phi_*F&\lra G & f^\omega\colon (\Psi^\omega)_*F&\lra G.
\end{align*}
The induced maps $f_*\colon C_*(F;R)\to C_*(G;R)$ and $f^\omega_*\colon C_*(F^\omega;R)\to C_*(G;R)$ have left inverses 
\begin{align*}
g\colon C_*(G;R)&\lra C_*(F;R) & g^\omega\colon C_*(G;R)&\lra C_*(F^\omega;R)
\end{align*} whose value on a generator $\sigma\in G(A)$ is
\begin{align*}
g(\sigma) &= \begin{cases}
\sigma\in F(\Gamma(A)) & \text{if $A$ is \very \wellordered} \\
0 & \text{otherwise}
\end{cases}\\
g^\omega(\sigma) &= \begin{cases}
\sigma\in F^{\omega}(\omega^{-1}\circ\Gamma(A)) & \text{if $A$ is \very \wellordered} \\
0 & \text{otherwise}
\end{cases}
\end{align*}
We now construct a chain homotopy $D$ from the identity on $C_*(G;\bF_2)$ to $f_*\circ g_*$ and a chain homotopy $D^\omega$ from the identity to $f^\omega_*\circ g^\omega_*$. As a consequence, both $f$ and $f^\omega$ induce isomorphisms on homology, so $f$ and $f^\omega$ are equivalences of cubes.

If $A=(a_1,\ldots,a_m)\subset \{1,\ldots,2n\}$ is a vertex of $\cube{2n}$ and $j\in \{1,\ldots, m\}$, define $A_{\leq j} = (a_1,\ldots,a_j)$ and $A_{\geq j} = (a_j,\ldots,a_m)$. If $A$ is \very \wellordered, then the value of $G$ at $A$ and at $\Theta_j(A)=\Phi\circ \Gamma(A_{\leq j})\cup \Psi\circ\Gamma(A_{\geq j})$ coincides, hence we may define a homomorphism
\begin{align*}
D_j\colon C_*(G;R)&\lra C_{*+1}(G;R)
\end{align*}
whose value on a generator $\sigma\in G(A)$ is $\sigma\in G(\Theta_j(A))$ if $A$ is \very \wellordered and $0$ otherwise. Finally, define the chain homotopies
\begin{align*}
D\colon C_*(G;R)&\lra C_{*+1}(G;R) & D^\omega\colon C_*(G;R)&\lra C_{*+1}(G;R) 
\end{align*}
whose value on a generator $\sigma\in G(A)$ with $A=\{a_1,\ldots,a_m\}$ is
\begin{align*}
D(\sigma)&=\sum_{j=1}^{\ell}(-1)^{j}D_j(\sigma)&
D^\omega(\sigma)&=\sum_{j=\ell+1}^{m}(-1)^{j+1}D_j(\sigma),
\end{align*}
where $\ell$ is the number such that $a_\ell\leq n$ and $a_{\ell+1}>n$. We have
\begin{align*}
\partial\circ D + D\circ\partial  &= f\circ g-\Id &\partial\circ D^\omega+ D^\omega\circ\partial  &=  f^\omega\circ g^{\omega}-\Id.
\end{align*}
The first equation holds because the summands in both sides are paired as follows
\begin{align*}
(-1)^{k+j} \partial_k\circ D_j &= -(-1)^{j+k-1}D_{j-1}\circ\partial_k &\text{if $k< j,j+1$} \\
(-1)^{k+j} \partial_k\circ D_j &= -(-1)^{j+k-1}D_{j}\circ\partial_{k-1} &\text{if $k> j,j+1$} \\
(-1)^{2j} \partial_j\circ D_j &= -(-1)^{2j-1}\partial_j\circ D_{j-1} &\text{if $j\neq 1$}\\
(-1)^{2}\partial_1\circ D_1 &= f\circ g&\\
(-1)^{2\ell+1}\partial_{\ell+1}\circ D_\ell &= -\Id,&
\end{align*}
and the second equation is obtained similarly.

\subsection{Khovanov homology}\label{ssection:khovanov-khovanov}

Let $D$ be an oriented knot diagram with $c$ crossings and $n_-$ negative crossings whose crossings have been ordered. Let $\cF_{D}\colon \cube{c}\to \cB$ be the Khovanov functor of Lawson, Lipshitz and Sarkar \cite{LLS2015,LLS-cube}, from which one obtains the stable functor $(\cF_D,-n_-)$. If we write $C_*(D;\bF_2)$ for the dual of the Khovanov cochain complex of $D$ with $\bF_2$ coefficients, then 
\[C_*(D;\bF_2) = \Sigma^{-n_-}C_*(\cF_D;\bF_2).\] 
Applying the functor $\Lambda$ of Section \ref{ssection:lambda} to $\cF$, we obtain \ansob $X_\bullet$ and the following chain complexes are equal (see Section \ref{ssection:realisations})
\[C_*(X_\bullet;\bF_2) = \Sigma^{-1}C_*(\cF_D;\bF_2).\]
Now, after choosing and order on $X_\bullet$, the cochain complex $C^*(X_\bullet;\bF_2)$ becomes endowed with a symmetric multiplication via the \cupi{i}-products of Theorem \ref{thm:main}, and therefore so does the Khovanov cochain complex of $D$:
\[\smile_i\colon C^{p-n_-+1}(D;\bF_2)\otimes C^{q-n_-+1}(D;\bF_2)\lra C^{p+q-i-n_-+1}(D;\bF_2)\quad i\in \bZ.\]
As a consequence, the Khovanov homology of $D$ is enhanced with the Steenrod squares associated to this symmetric multiplication
\[\Sq^i\colon \Kh^{n-n_-+1}(D;\bF_2)\lra \Kh^{n+i-n_-+1}(D;\bF_2), \quad \Sq^i([\alpha]) = [\alpha\smile_{n-i} \alpha], \quad i\geq 0.\]

\begin{proposition}\cite[p.\ 14]{LLS-cube} If $D$ and $D'$ are two oriented link diagrams with ordered crossings and $n_-$ and $n'_-$ negative crossings, and $D$ and $D'$ are related by a Reidemeister move, then $(\cF_D,-n_-)$ and $(\cF_{D'},-n_-')$ are stably equivalent. 
\end{proposition}
Note also that if $D$ is a link diagram with $c$ ordered crossings and $n_-$ negative crossings, and $D'$ is the same diagram but whose crossings have been ordered differently, then there is a permutation $\omega$ of $\{1,\ldots,c\}$ such that $\cF_{D'} = \omega_*\cF_{D'}$, hence $(\cF_{D'},n_-)$ and $(\cF_{D'},n_-)$ are stably equivalent. Therefore, by Corollary \ref{cor:steq}:
\begin{corollary} If $D$ and $D'$ are two oriented link diagrams with ordered crossings and $n_-$ and $n'_-$ negative crossings, and $D$ and $D'$ are related by a Reidemeister move or by a reordering of the crossings, then there is some $k\geq 0$ such that $\Sigma^{k-n_-}\Lambda(\cF_D)$ and $\Sigma^{k-n_-'}\Lambda(\cF_{D'})$ are equivalent.
\end{corollary}
\begin{corollary}\label{cor:citable} If $D$ is a link diagram with ordered crossings and $n_-$ negative crossings, then the Steenrod squares $\Sq^i$ applied to the \sob $\Lambda(\cF_D)$ give operations
\[\Sq^i\colon \Kh^{n-n_-+1}(D;\bF_2)\lra \Kh^{n+i-n_-+1}(D;\bF_2), \quad \Sq^i([\alpha]) = [\alpha\smile_{n-i} \alpha]\]
that are independent of the chosen diagram and the ordering of the crossings. 
\end{corollary}

\section{Examples}\label{section:examples}
Write $|\cdot|_\bS$ for the functor labeled * in diagram \eqref{eq:picture}. We will use this functor together with the commutativity of the bottom square of diagram \eqref{eq:picture} to be able to discuss the expected values of the Steenrod squares in spectra (though the computations are independent of these discussions). 

\begin{example}\label{example:rp2} Consider the ordered span $f$ given by $\{x\}\la \{a,b\}\to\{y\}$, and construct the \osob $X_\bullet$ by declaring
\begin{align*}
X_{-1} &= \{y\} & X_0 &= \{x\} & X_i&=\emptyset\quad \text{for all } i>0\end{align*}
and that
\[\partial_0^0 =f\colon X_0\lra X_{-1}.\]
Then, we have that $C_*(X_\bullet;\bZ)$ is the chain complex
\[\ldots\lra 0\lra \bZ\langle x\rangle \overset{\cdot 2}{\lra} \bZ\langle y\rangle\lra 0\]
where $x$ is in degree $0$ and $y$ is in degree $-1$. The homology of this complex is concentrated in degree $-1$ and $H_{-1}(X_\bullet;\bZ) = \bZ_2$, therefore $|X_\bullet|_\bS$ is a Moore spectrum $M(\bZ_2,-1)$ for the group $\bZ_2$ in degree $-1$, and, by the uniqueness of Moore spectra, we have determined the homotopy type of $|X_\bullet|_\bS$. A model of $M(\bZ_2,-1)$ is, for example, $\Sigma^{-2}\Sigma^\infty\bR P^2$. 

The complex of cochains with $\bF_2$ coefficients is
\[0\lra \bF_2\langle y^*\rangle \overset{\cdot 2}{\lra} \bF_2\langle x^*\rangle\lra 0\lra\ldots\]
Let us compute $\Sq^1(y^*)$. As $y^*$ has degree $-1$, we have by definition that
\[\Sq^1(y^*) = y^*\smile_{-2}y^*\]
and 
\[\smile_{-2}\colon \bF_2\langle y^*\rangle\otimes \bF_2\langle y^*\rangle\lra \bF_2\langle x^*x\rangle\] is the dual of
\[\nabla_{-2}\colon \bF_2\langle x\rangle\to \bF_2\langle y\rangle\otimes \bF_2\langle y\rangle.\]
Now, $\nabla_{-2}$ restricted to $X_0$ is the \abelianisation{\bF_2} of 
\[\nabla^{\binom{0}{2}}=\sum_{U\in \APower_2(0)}\partial_{U^-}\wedge\partial_{U^+}\circ \Delta\]
Since $\APower_2(0)$ contains only the sequence $(0,0)$, we have 
\[\nabla_{-2} = \btoabfunctor{\bF_2}\left(\partial_0\wedge\partial_0\circ \Delta\right)\]
and, since $a<b$ and $n+|U^+|+|U^-|=0+1+1$ is even, we have that $\cO_{a,b}(\twoU)^+$ has a single element $(\emptyset,\{0\})$ and that $\cO_{b,a}(\twoU)^+$ is empty, therefore $\partial_0\wedge\partial_0$ is equivalent to the span 
\[\{x\}\times \{x\}\lla \{\star\}\lra \{y\}\times\{y\},\]
whose $\bF_2$-realisation is the homomorphism $(x,x)\mapsto (y,y)$. Therefore
\begin{align*}
\nabla_{-2}(x) &= \btoabfunctor{\bF_2}\left(\partial_0\wedge\partial_0\circ \Delta\right)(x) \\
&= \btoabfunctor{\bF_2}\left(\partial_0\wedge\partial_0\right)(x,x) \\
&= (y,y)
\end{align*}
so $y^*\smile_{-2}y^* = x^*$ and $\Sq^1(y^*) = x^*$, as it should be.
\end{example}
\begin{example}\label{example:join} Let us take now another copy $X'_\bullet$ of the object $X_\bullet$ studied in Example \ref{example:rp2} and take the join product $X_\bullet* X'_\bullet$, which is (Definition \ref{df:join}):
\[X_{1} = \{xx'\},\quad\quad X_{0} = \{yx',xy'\},\quad\quad  X_{-1} = \{yy'\}\]
with generalised face maps
\[
 \xymatrixrowsep{-.2cm}\xymatrix{
\partial^1_0 & xx' & \ar@{<->}[l]\{a,b\}\ar@{<->}[r]& yx' \\
 & & & \\
\partial^1_1 & xx' & \ar@{<->}[l]\{a',b'\}\ar@{<->}[r]& xy' \\
&&&\\
&&a'\ar@{|->}[dddr] \ar@{|->}[dl] & \\
&yx' && \\
&&b'\ar@{|->}[ul] \ar@{|->}[dr] & \\
\partial^0_0 &&&yy' \\
&&a\ar@{|->}[ur]\ar@{|->}[dl] & \\
&xy' & & \\
&&b\ar@{|->}[ul]\ar@{|->}[uuur] \\
\partial^1_{\{0,1\}} & xx' & \{aa',ab',ba',bb'\}\ar@{|->}[l]\ar@{|->}[r]& yy'}
\]
and we endow each span with the order that results from reading the entries left to right or up to down. This is summarised in the following diagram:

\[\xymatrix{
1&&xx'\ar@{<->}[dl]|-{\caja{\{a,b\}}}\ar@{<->}[dr]|-{\{a',b'\}} & & xx'\ar@{<->}[dd]|-{\{aa',ab',ba',bb'\}}\\
0&yx'\ar@{<->}[dr]|-{\{a',b'\}}&&xy'\ar@{<->}[dl]|-{\{a,b\}} &\\
-1&&yy'&& yy'.}
\]
The $2$-morphisms are:
\[\xymatrixrowsep{.15cm}\xymatrixcolsep{1.5cm}\xymatrix{\partial^1_0\times_{X_0}\partial^0_0 &\ar[l]_-{\mu^1_{0,1}} \partial_{\{0,1\}}\ar[r]^-{\mu^1_{1,0}} & \partial^1_1\times_{X_0}\partial^0_0 \\
(a,a')& aa' \ar@{|->}[l]\ar@{|->}[r]& (a',a) \\
(a,b')& ab' \ar@{|->}[l]\ar@{|->}[r]& (b',a) \\
(b,a')& ba' \ar@{|->}[l]\ar@{|->}[r]& (a',b) \\
(b,b')& bb' \ar@{|->}[l]\ar@{|->}[r]& (b',b)
}\]
By \eqref{eq:smash}, $|X_\bullet* X_\bullet'|_\bS\simeq \Sigma^{-3}\Sigma^\infty \bR P^2\wedge \bR P^2$, so we expect that $\Sq^2((yy')^*) = (xx')^*$. By definition, we have that \eqref{eq:squares}
\[\Sq^2((yy')^*) = (yy')^*\smile_{-3}(yy')^*\]
and that the operation $\smile_{-3}$ is dual to the operation $\nabla_{-3}$, whose restriction to $\bF_2\langle X_1\rangle$ is the \abelianisation{\bF_2} of 
\[\nabla^{\binom{1}{4}} = \sum_{U\in \APower_4(1)} \partial_{U^-}\wedge\partial_{U^+}.\]
Since $\APower_4(1)$ has only one element $U=(0,0,1,1)$, writing $\partial_{01}$ for $\partial^1_{\{0,1\}}$
\[\nabla_{-3} = \btoabfunctor{\bF_2}\left(\partial_{01}\wedge \partial_{01}\circ \Delta\right).\]
now, to compute this wedge product, 
\[\partial_{01}\wedge \partial_{01} = \sum_{(s,t)\in \partial_{01}\times\partial_{01}} \cO_{s,t}(\twoU)^+\]
we need to compute for each $(s,t)\in \partial_{01}\times\partial_{01}$
\[\cO_{s,t}(\twoU)^+ = \{\text{positive maximal chains that are $(s,t)$-good}\}\]
All maximal chains have length $2$, so they take the form $(W_1^\shortparallel, W_1^\smallwedge)\prec(\emptyset,\{0,1\})$. Every element $(s,t)$ with $s\neq t$ has two $(s,t)$-good maximal chains: one in which either $W_1^\shortparallel = \{0\}$ or $W^\shortparallel_1 = \emptyset$ and $W_1^\smallwedge = \{0\}$; and another in which either $W_1^\shortparallel = \{1\}$ or $W^\shortparallel_1 = \emptyset$ and $W_1^\smallwedge = \{1\}$. We call the first maximal chain ``left'' and the second ``right''. By definition, since $n+|U^-|+|U^+| = 1+2+2 =5$, we have that $(W^\shortparallel,W^\smallwedge)$ is a positive $(s,t)$-good pair if
\begin{align*}
\lambda_{W^\shortparallel}(s) &= \lambda_{W^\shortparallel}(t) & \lambda_{W^\shortparallel, W^\smallwedge}(s) &> \lambda_{W^\shortparallel,W^\smallwedge}(t)
\end{align*}
Now, $\partial_{01}\times \partial_{01}$ has $4^2$ elements. The elements in the diagonal have no maximal chains, so $\cO_{s,s}(\twoU)^+=\emptyset$ for them. For the remaining $12$, we have that for half of them (those $(s,t)$ such that $s<t$) the pair $(\emptyset,\{0,1\})$ is negative, so $\cO_{s,t}(\twoU)^+=\emptyset$ for them too. For the remaining $6$ elements, the pair $(\emptyset,\{0,1\})$ is positive, and in the following table we give the positivity of the middle pair $(W_1^\shortparallel, W_1^\smallwedge)$ of the left and right maximal chains of the remaining $6$ elements. As each maximal chain is determined by its first pair $(W_1^\shortparallel,W_1^\smallwedge)$, we give only this datum.
\[\begin{array}{|cccc|} \hline
(s,t) & \text{max. chain (left)} & \text{max. chain (right)}&\text{positivity}  \\ \hline
(bb',ba')& (\{0\},\{1\}) & (\emptyset,\{1\}) & +,+ \\
(bb',ab')& (\emptyset, \{0\}) & (\{1\},\{0\}) & +,+ \\
(bb',aa')& (\emptyset, \{0\}) & (\emptyset, \{1\}) & +,+ \\
(ba',ab')& (\emptyset, \{0\}) & (\emptyset, \{1\}) & +,- \\
(ba',aa')& (\emptyset, \{0\}) & (\{1\}, \{0\}) & +,+ \\
(ab',aa')& (\{0\},\{1\}) & (\emptyset, \{1\}) & +,+ \\ \hline
\end{array}\]
For example, for $(s,t) = (bb',ab')$, we have that
\begin{align*}
\lambda_0(bb') &= b >a =\lambda_0(ab') & &\Rightarrow\quad (\emptyset,\{0\}) \text{ is good} \\
\lambda_1(bb') &= b'=b'=\lambda_1(ab'), \lambda_{1,0}(bb') = b>a=\lambda_{1,0}(ab')& &\Rightarrow \quad (\{1\},\{0\})\text{ is good}
\end{align*}
and both are positive, and if $(s,t) = (ba',ab')$, we have that
\begin{align*}
\lambda_0(ba') &= b >a =\lambda_0(ab') & &\Rightarrow\quad (\emptyset,\{0\}) \text{ is good and positive} \\
\lambda_1(ba') &= a'<b'=\lambda_1(ab'), &&\Rightarrow \quad (\emptyset,\{1\})\text{ is good and negative}.
\end{align*}
As a consequence, there are $11$ positive maximal chains, so $\partial_{01}\wedge\partial_{01}$ has odd cardinality, so $\btoab{\bF_2}{\partial_{01}\wedge\partial_{01}\circ \Delta}$ is the homomorphism that sends $xx'$ to $yy'\otimes yy'$. Therefore, 
\begin{align*}
\nabla_{-3}(xx') &= yy', & (yy')^*\smile_{-3}(yy')^* &= (xx')^*, & \Sq^2((yy')^*) &= (xx')^*.\end{align*}
\end{example}
\begin{example}
Iterating the construction of the previous example, we obtain models of $\Sigma^{-k-1}\Sigma^\infty\underbrace{\bR P^2\wedge \bR P^2\wedge\ldots \bR P^2}_{k}$, for which the operation $\Sq^k$ applied to the element in degree $-1$ is non-trivial. This operation comes from a \cupi{i}product of degree $-2k$, so these iterations give examples of non-trivial \cupi{i}operations of all negative degrees.
\end{example}

\begin{example} Let us consider now an ordered \sob $X_\bullet$ with
\begin{align*}X_{-1}&=\{a\} &X_{0} &= \{b_1,\ldots,b_k\}& X_{1}&=\{c\}& X_i&=\emptyset \quad i>1\end{align*}
Endow $\partial^0_0\circ \partial^1_0$ and $\partial^0_0\circ \partial^1_1$ with the lexicographic order. Let $N$ be the cardinal of $\partial^1_{\{0,1\}}$. Using the ordering, we can identify $\mu^1_{0,1}$ and $\mu^1_{1,0}$ as elements of the symmetric group on $N$ letters, and we let $\sign(\mu^1_{0,1})$ and $\sign(\mu^1_{1,0})$ be the sign of these permutations. We claim that
\begin{equation}\label{eq:last}\nabla_{-3}(c) =  \left( \sign(\mu^1_{0,1}) + \sign(\mu^1_{1,0})\right)\cdot a\otimes a.\end{equation}
so
\begin{equation}\Sq^2(a^*) =  \left( \sign(\mu^1_{0,1}) + \sign(\mu^1_{1,0})\right)\cdot c^*.\end{equation}
To prove \eqref{eq:last}, note first that, as in Example \ref{example:join}, $\APower_4(1)$ has only one sequence $U = (0,0,1,1)$, so
\[\nabla_{-3} = \btoab{\bF_2}{\nabla^{\binom{1}{4}}\circ \Delta_1} = \btoab{\bF_2}{\partial^1_{\{0,1\}}\wedge\partial^1_{\{0,1\}}\circ\Delta_1}\]
To compute $\partial^1_{\{0,1\}}\wedge\partial^1_{\{0,1\}}$, note that, for each $(s,t)\in \partial^1_{\{0,1\}}\times \partial^1_{\{0,1\}}$ the set $\cO_{s,t}(\twoU)$ has exactly two elements (cf.\ Example \ref{example:join}): the left and the right. The left element is positive if $s>t$ and $\mu^1_{0,1}(s)>\mu^1_{0,1}(t)$ in the lexicographic order. The right element is positive if $s>t$ and $\mu^1_{1,0}(s)>\mu^1_{1,0}(t)$. 

Assume first that $\mu^1_{0,1}$ and $\mu^1_{1,0}$ are the identity permutations. Then, for each $(s,t)$, either $s>t$, in which case both the left and right chains are positive, or $s<t$, in which case none of them is positive. Therefore, $\cO_{s,t}(\twoU)^+$ has even cardinality and $\partial^1_{\{0,1\}}\wedge\partial^1_{\{0,1\}}\igualdos \emptyset$, so \eqref{eq:last} holds in this case.

Assume now that \eqref{eq:last} holds for some permutations $\mu^1_{0,1}$ and $\mu^1_{1,0}$, and let $\bmu^1_{0,1}$ be the result of changing $\mu^1_{0,1}$ by a transposition of two consecutive elements $s,t\in \partial_{\{0,1\}}$ with $s<t$, i.e.,
\begin{align*}
\bmu^1_{0,1}(s) &= \mu^1_{0,1}(t) &
\bmu^1_{0,1}(t) &= \mu^1_{0,1}(s) &
\bmu^1_{0,1}(r) &= \mu^1_{0,1}(r)\quad r\neq s,t.
\end{align*}
Then, for any $(r,r')$ different from $(s,t)$ or $(t,s)$, the positivity of the maximal $(r,r')$-good chains remains the same. Additionally, no $(t,s)$-good maximal chain in $\cO_{t,s}(\twoU)^+$ was positive for, and none of them becomes positive, because $t>s$. On the other hand, the right $(s,t)$-good maximal chain is as positive as before whereas the left $(s,t)$-good maximal chain changes its positivity. Therefore the parity of $\partial^1_{\{0,1\}}\wedge\partial^1_{\{0,1\}}$ changes by one, so the cup-$i$ product in the new semi-simplicial object changed by one, as does the sign of the permutation $\bmu^1_{0,1}$. A symmetric argument shows that the same holds for $\mu^1_{1,0}$, and since the symmetric group is generated by transpositions of consecutive elements, we conclude that \eqref{eq:last} holds always.

This simple formula in terms of the sign of the permutations raises the following question:
\begin{quote} \emph{Is it possible to give an interpretation of the products $\partial_U\wedge\partial_V$ in terms of the homology of the permutation groups $\Sigma_k$ with $\bF_2$ coefficients?}
\end{quote}

\end{example}

\begin{example} We know compute a second Steenrod square in the Khovanov homology of the disjoint union of two right-handed trefoils $T \amalg T$ (cf.\ \cite[p.\ 60]{LLS2015}). In this example and the next one we assume that the reader is familiar with Khovanov homology and with the Khovanov functor of Lawson, Lipshitz and Sarkar \cite{LLS2015,LLS-cube}. We will use the following knot diagram $D$ of $T\amalg T$:

\bigskip

\begin{center}
\def\svgwidth{.4\columnwidth}
\begingroup%
  \makeatletter%
  \providecommand\color[2][]{%
    \errmessage{(Inkscape) Color is used for the text in Inkscape, but the package 'color.sty' is not loaded}%
    \renewcommand\color[2][]{}%
  }%
  \providecommand\transparent[1]{%
    \errmessage{(Inkscape) Transparency is used (non-zero) for the text in Inkscape, but the package 'transparent.sty' is not loaded}%
    \renewcommand\transparent[1]{}%
  }%
  \providecommand\rotatebox[2]{#2}%
  \newcommand*\fsize{\dimexpr\f@size pt\relax}%
  \newcommand*\lineheight[1]{\fontsize{\fsize}{#1\fsize}\selectfont}%
  \ifx\svgwidth\undefined%
    \setlength{\unitlength}{201.2421118bp}%
    \ifx\svgscale\undefined%
      \relax%
    \else%
      \setlength{\unitlength}{\unitlength * \real{\svgscale}}%
    \fi%
  \else%
    \setlength{\unitlength}{\svgwidth}%
  \fi%
  \global\let\svgwidth\undefined%
  \global\let\svgscale\undefined%
  \makeatother%
  \begin{picture}(1,0.46919232)%
    \lineheight{1}%
    \setlength\tabcolsep{0pt}%
    \put(0,0){\includegraphics[width=\unitlength,page=1]{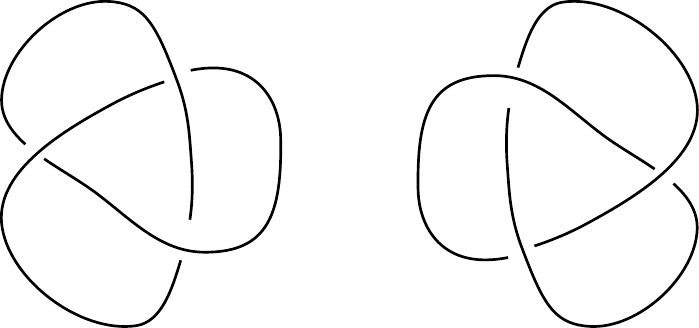}}%
  \end{picture}%
\endgroup%

\end{center}

\bigskip

\noindent 
After choosing an ordering of the crossings, its $1$-resolution is 

\bigskip

\begin{center}
\centering
\def\svgwidth{.4\columnwidth}
\begingroup%
  \makeatletter%
  \providecommand\color[2][]{%
    \errmessage{(Inkscape) Color is used for the text in Inkscape, but the package 'color.sty' is not loaded}%
    \renewcommand\color[2][]{}%
  }%
  \providecommand\transparent[1]{%
    \errmessage{(Inkscape) Transparency is used (non-zero) for the text in Inkscape, but the package 'transparent.sty' is not loaded}%
    \renewcommand\transparent[1]{}%
  }%
  \providecommand\rotatebox[2]{#2}%
  \newcommand*\fsize{\dimexpr\f@size pt\relax}%
  \newcommand*\lineheight[1]{\fontsize{\fsize}{#1\fsize}\selectfont}%
  \ifx\svgwidth\undefined%
    \setlength{\unitlength}{868.38847657bp}%
    \ifx\svgscale\undefined%
      \relax%
    \else%
      \setlength{\unitlength}{\unitlength * \real{\svgscale}}%
    \fi%
  \else%
    \setlength{\unitlength}{\svgwidth}%
  \fi%
  \global\let\svgwidth\undefined%
  \global\let\svgscale\undefined%
  \makeatother%
  \begin{picture}(1,0.36892177)%
    \lineheight{1}%
    \setlength\tabcolsep{0pt}%
    \put(0,0){\includegraphics[width=\unitlength,page=1]{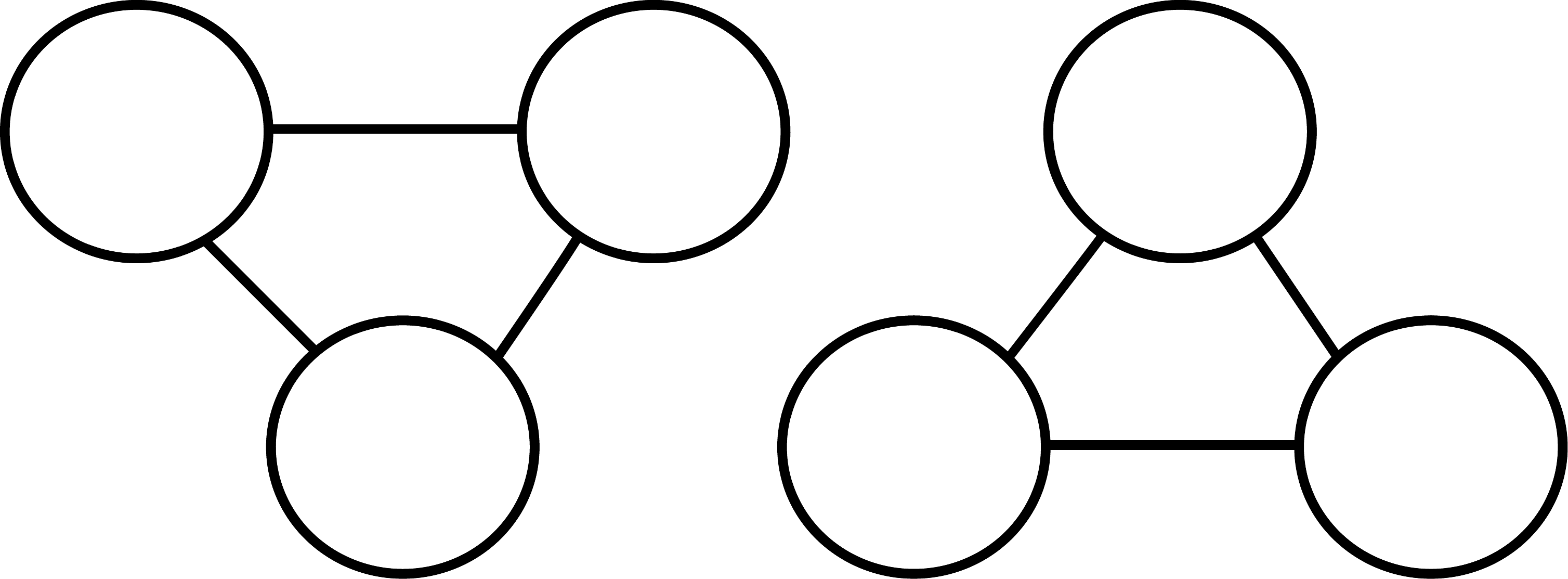}}%
    \put(0.23645619,0.31004205){\color[rgb]{0,0,0}\makebox(0,0)[lt]{\lineheight{12.375}\smash{\begin{tabular}[t]{l}$0$\end{tabular}}}}%
    \put(0.10954382,0.14563285){\color[rgb]{0,0,0}\makebox(0,0)[lt]{\lineheight{12.375}\smash{\begin{tabular}[t]{l}$1$\end{tabular}}}}%
    \put(0.36116884,0.14565915){\color[rgb]{0,0,0}\makebox(0,0)[lt]{\lineheight{12.375}\smash{\begin{tabular}[t]{l}$2$\end{tabular}}}}%
    \put(0.72563703,0.02647227){\color[rgb]{0,0,0}\makebox(0,0)[lt]{\lineheight{12.375}\smash{\begin{tabular}[t]{l}$3$\end{tabular}}}}%
    \put(0.6237241,0.1844318){\color[rgb]{0,0,0}\makebox(0,0)[lt]{\lineheight{12.375}\smash{\begin{tabular}[t]{l}$4$\end{tabular}}}}%
    \put(0.84136863,0.1844318){\color[rgb]{0,0,0}\makebox(0,0)[lt]{\lineheight{12.375}\smash{\begin{tabular}[t]{l}$5$\end{tabular}}}}%
    \put(0.06406672,0.27093469){\color[rgb]{0,0,0}\makebox(0,0)[lt]{\lineheight{12.375}\smash{\begin{tabular}[t]{l}$a$\end{tabular}}}}%
    \put(0.23766417,0.0646663){\color[rgb]{0,0,0}\makebox(0,0)[lt]{\lineheight{12.375}\smash{\begin{tabular}[t]{l}$b$\end{tabular}}}}%
    \put(0.40089757,0.27093469){\color[rgb]{0,0,0}\makebox(0,0)[lt]{\lineheight{12.375}\smash{\begin{tabular}[t]{l}$c$\end{tabular}}}}%
    \put(0.73772838,0.27093469){\color[rgb]{0,0,0}\makebox(0,0)[lt]{\lineheight{12.375}\smash{\begin{tabular}[t]{l}$e$\end{tabular}}}}%
    \put(0.55635794,0.06514605){\color[rgb]{0,0,0}\makebox(0,0)[lt]{\lineheight{12.375}\smash{\begin{tabular}[t]{l}$d$\end{tabular}}}}%
    \put(0.89318876,0.06495997){\color[rgb]{0,0,0}\makebox(0,0)[lt]{\lineheight{12.375}\smash{\begin{tabular}[t]{l}$f$\end{tabular}}}}%
  \end{picture}%
\endgroup%

\end{center}

\bigskip

Let $C^{*,*}(D;\bF_2)$ be the Khovanov complex with $\bF_2$ coefficients associated to this diagram, which is concentrated in degrees $0,1,2,3,4,5$ and $6$ because the number of negative crossings $n_-$ of this diagram is $0$. Figure \ref{figure:2trefoil} shows the cube of resolutions in degrees $4,5$ and $6$, and below are the ranks of the subcomplex $C^{14,*}(D;\bF_2)$ generated by the generators of quantum grading $14$ and its homology:
\[\begin{array}{|r|llll|}
\hline
& 6 & 5 & 4 & 3 \\ \hline 
&&&&\\   [-.3cm]
C^{14,*}(D;\bF_2) &\bF_2^{15}&\bF_2^{30}&\bF_2^{15}&\bF_2^2\\
\Kh^{14,*}(D;\bF_2)&\bF_2&\bF_2^4&\bF_2&0 \\ \hline
\end{array}
\]
 Consider now the Khovanov functor $\cF^{14}\colon \cube{6}\to \cB$ for $D$ in quantum grading $14$, and let $X_\bullet = \Lambda(\cF^{14})$ be the associated \sob (see Section \ref{ssection:lambda}). Let $C_*$ be the $\bF_2$-realisation of $X_\bullet$, which by construction is the one-fold desuspension of the chain complex $\Tot_{\bF_2}(\cF^{14})$, whose $n_-$-desuspension is the Khovanov complex. Since $n_-=0$, we have that the dual complex of $C_*$ is: 
\[C^* \cong \Sigma^{-1}C^{14,*}(D;\bF_2).\]

\begin{center}
\begin{figure}
\footnotesize
\def\svgwidth{.84\columnwidth}
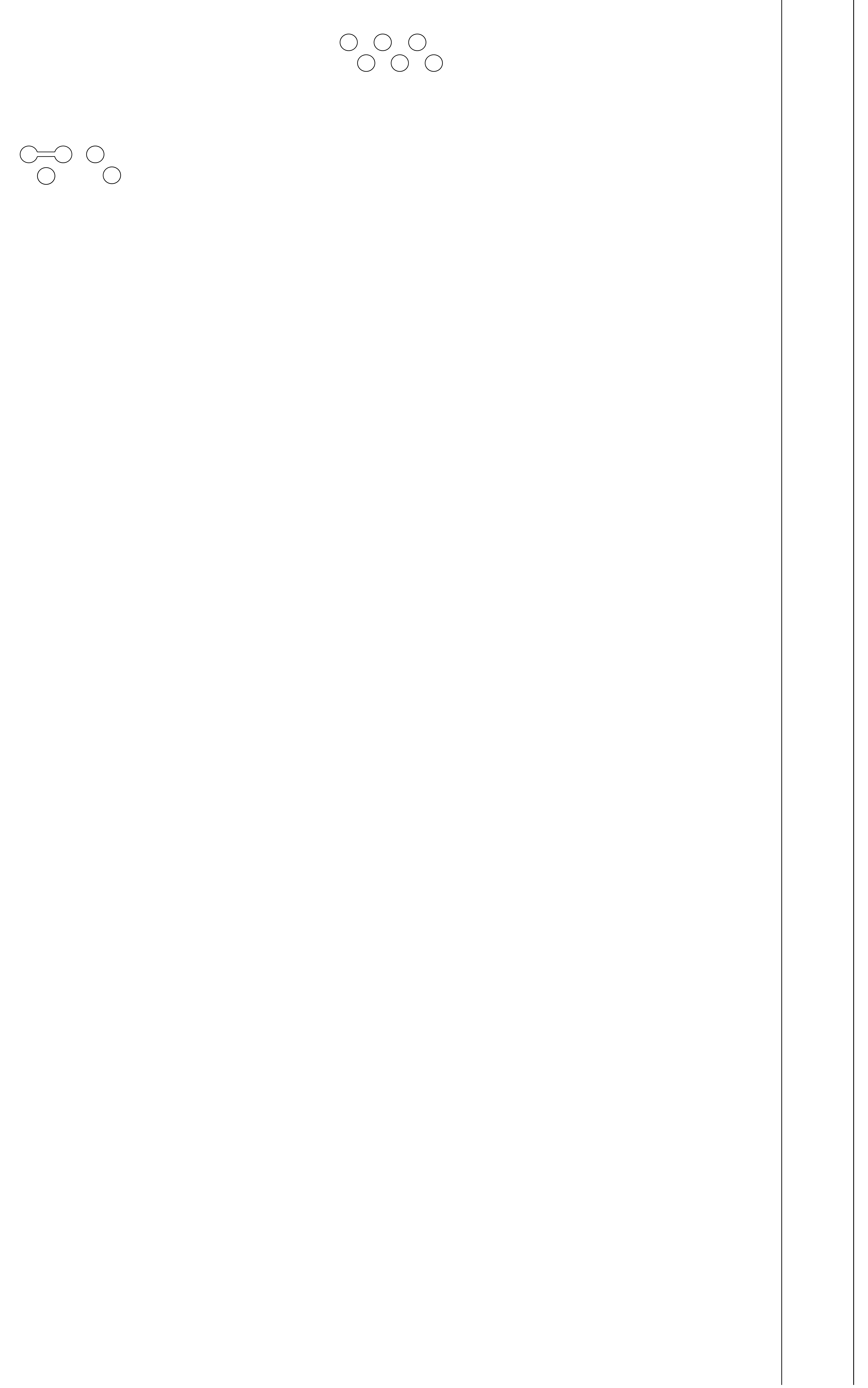
\caption{Cube of resolutions of $D$ in homological degrees $4,5$ and $6$ (column $h$), which under the semi-simplicial convention are degrees $3,4$ and $5$ (column $n$). Each arrow $a<b$ is labeled with the face map $\Lambda(a<b)$. The cube is splitted in four rows to avoid cluttering. The edges are oriented downwards instead of upwards, as in \cite{LLS2015} and \cite{LLS-cube}. We name the six circles in $D$ as $a,b,c,d,e,f$ as shown in the bottom left part of the figure.}
\label{figure:2trefoil}
\end{figure}
\end{center}

The generators of $C^{14,*}(D;\bF_2)$ are as follows:
\begin{itemize}
\item a generator in homological degree $6$ (semi-simplicial degree $5$) enhances two circles with $x_-$ and four circles with $x_+$.
\item a generator in homological degree $5$ (semi-simplicial degree $4$) enhances one circle with $x_-$ and four circles with $x_+$.
\item a generator in homological degree $4$ (semi-simplicial degree $3$) enhances each of the four circles with $x_+$.
\end{itemize}
Hence, there is a unique generator $\z_{u}$ in quantum grading $14$ in each vertex $u$ with semi-simplicial grading $3$, which enhances every circle with $x_+$. Let us compute the second Steenrod square of the following cocycle
\[\alpha= \z_{011110}^* + \z_{011011}^* + \z_{110011}^*+ \z_{110110}^*.\]
As $\alpha$ has semi-simplicial degree $3$, by definition
\[\Sq^2([\alpha]) = [\alpha\smile_{3-2}\alpha].\]
Now, the $\smile_1$ product is dual to $\nabla_{1}$, and we want to compute it on the generators of $C_*$ of semi-simplicial degree $5$ (which is where the second Steenrod square of $\alpha$ lives), and again, by definition,
\[\nabla_1|_{C_5} = \btoab{\bF_2}{\nabla^{\binom{5}{4}}\circ \Delta_5}.\]
and
\[\nabla^{\binom{5}{4}} = \sum_{U\in \APower_4(5)}{\partial^5_{U^-}\wedge\partial^5_{U^+}}.\]
We only need to look at those $U$'s such that $|U^-|=|U^+|=2$ because we are feeding both sides of $\smile_1$ with the cochain $\alpha$ of semi-simplicial degree $3$. Moreover, every face map in the cube is a merging, so it defines a span that is of the form $A\overset{=}{\leftarrow}A\to B$, i.e., a function of sets. Therefore, by Corollary \ref{cor:sset}, every summand indexed by a $U\in \APower_4(5)$ with $\twoU\neq\emptyset$ is trivial. There are exactly three sequences in $\APower_4(5)$ for which $|U^-|=|U^+|=2$ and $\twoU=\emptyset$, which are:
\begin{align*}
U\quad && U^- && U^+ \\
(0,1,3,4)&& (0,1) && (3,4) \\
(0,2,3,5)&& (0,5) && (2,3) \\
(1,2,4,5)&& (4,5) && (1,2).
\end{align*}
Write $\partial^5_{ij}$ for $\partial^5_{\{i,j\}}$. First, note that the span $\partial^5_{01}$ is non-trivial on a generator $\z$ if and only if $\z$ enhances two of the circles $a,b,c$ with a $x_-$. Similarly, the span $\partial^5_{34}$ is non-trivial on a generator $\z$ if and only if $\z$ enhances two of the circles $d,e,f$ with a $x_-$. Since these two conditions are mutually excluding, we have that $\left(\partial^5_{01}\wedge\partial^5_{34}\right)\circ\Delta_5$ is the empty span. A similar argument shows that $\partial^5_{45}\wedge\partial^5_{12}\circ\Delta_5$ is trivial as well. The remaining span is $\partial^5_{05}\wedge\partial^5_{23}\circ\Delta_5$. The span $\partial^5_{05}$ is non-trivial on a generator $\z$ if and only if $\z$ enhances one of the circles $a,c$ and one of the circles $e,f$ with a $x_-$. The span $\partial^5_{23}$ is non-trivial on a generator $\z$ if and only if $\z$ enhances two of the circles $b,c$ and one of the circles $d,f$ with an $x_-$. As a consequence, the span $\partial^5_{05}\wedge\partial^5_{23}\circ\Delta_5$ is non-trivial only on the generator $\x_{c,f}$ that enhances the circles $c,f$ with an $x_-$ and every other circle with an $x_+$. The target of the span is precisely the singleton $\{(\z_{u}, \z_v)\}$, with $u = 011110$ and $v= 110011$. Therefore we have that for each generator $\x$ of $C_5$,
\[\nabla_{1}(\x) = \begin{cases}
0& \text{ if $\x\neq \x_{c,f}$} \\ \z_u\otimes \z_v&\text{ if $\x=\x_{c,f}$.}\end{cases}\]
Since $\alpha(\z_u) = \alpha(\z_v) = 1$, we have that
\[\left(\alpha\smile_1\alpha\right)(\z) = \begin{cases}
0 &\text{ if $\x\neq \x_{c,f}$} \\
1 &\text{ if $\x= \x_{c,f}$} 
\end{cases}\]
and therefore letting $\beta$ be the dual of $\x_{c,f}$, we have
\[\Sq^2([\alpha]) = [\beta].\]
Since $[\beta]$ is a non-zero generator of $\Kh^{14,6}(D;\bF_2)$, we additionally deduce that $\Sq^2$ is non-trivial on $[\alpha]$.

\end{example}

\begin{example} In the previous example, every map involved was merging. That simplified drastically the computations because the span induced by a merging is free. In this example we introduce several splittings, which will give rise to non-free spans. Consider the following diagram $D$ of the unlink:

\begin{center}
\footnotesize
\def\svgwidth{.2\columnwidth}
\begingroup%
  \makeatletter%
  \providecommand\color[2][]{%
    \errmessage{(Inkscape) Color is used for the text in Inkscape, but the package 'color.sty' is not loaded}%
    \renewcommand\color[2][]{}%
  }%
  \providecommand\transparent[1]{%
    \errmessage{(Inkscape) Transparency is used (non-zero) for the text in Inkscape, but the package 'transparent.sty' is not loaded}%
    \renewcommand\transparent[1]{}%
  }%
  \providecommand\rotatebox[2]{#2}%
  \newcommand*\fsize{\dimexpr\f@size pt\relax}%
  \newcommand*\lineheight[1]{\fontsize{\fsize}{#1\fsize}\selectfont}%
  \ifx\svgwidth\undefined%
    \setlength{\unitlength}{327.55474011bp}%
    \ifx\svgscale\undefined%
      \relax%
    \else%
      \setlength{\unitlength}{\unitlength * \real{\svgscale}}%
    \fi%
  \else%
    \setlength{\unitlength}{\svgwidth}%
  \fi%
  \global\let\svgwidth\undefined%
  \global\let\svgscale\undefined%
  \makeatother%
  \begin{picture}(1,1.12482503)%
    \lineheight{1}%
    \setlength\tabcolsep{0pt}%
    \put(0,0){\includegraphics[width=\unitlength,page=1]{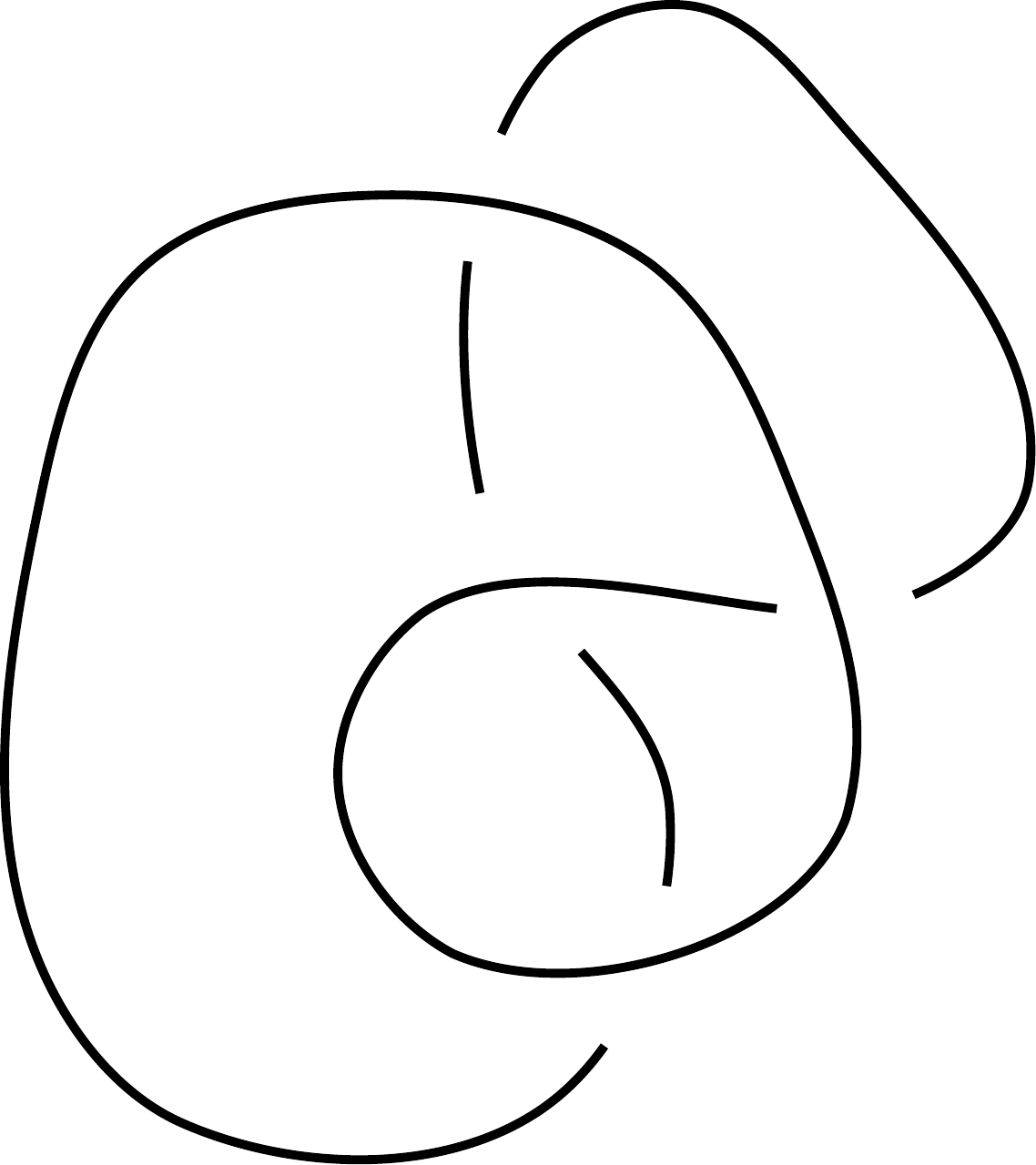}}%
  \end{picture}%
\endgroup%

\end{center}
whose $1$-resolution is
\begin{center}
\footnotesize
\def\svgwidth{.1\columnwidth}
\begingroup%
  \makeatletter%
  \providecommand\color[2][]{%
    \errmessage{(Inkscape) Color is used for the text in Inkscape, but the package 'color.sty' is not loaded}%
    \renewcommand\color[2][]{}%
  }%
  \providecommand\transparent[1]{%
    \errmessage{(Inkscape) Transparency is used (non-zero) for the text in Inkscape, but the package 'transparent.sty' is not loaded}%
    \renewcommand\transparent[1]{}%
  }%
  \providecommand\rotatebox[2]{#2}%
  \newcommand*\fsize{\dimexpr\f@size pt\relax}%
  \newcommand*\lineheight[1]{\fontsize{\fsize}{#1\fsize}\selectfont}%
  \ifx\svgwidth\undefined%
    \setlength{\unitlength}{37.76276501bp}%
    \ifx\svgscale\undefined%
      \relax%
    \else%
      \setlength{\unitlength}{\unitlength * \real{\svgscale}}%
    \fi%
  \else%
    \setlength{\unitlength}{\svgwidth}%
  \fi%
  \global\let\svgwidth\undefined%
  \global\let\svgscale\undefined%
  \makeatother%
  \begin{picture}(1,1.35661388)%
    \lineheight{1}%
    \setlength\tabcolsep{0pt}%
    \put(0,0){\includegraphics[width=\unitlength,page=1]{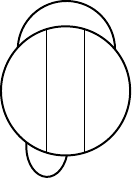}}%
  \end{picture}%
\endgroup%

\end{center}
We take the following order on the chords of the $1$-resolution:
\begin{center}
\footnotesize
\def\svgwidth{.1\columnwidth}
\begingroup%
  \makeatletter%
  \providecommand\color[2][]{%
    \errmessage{(Inkscape) Color is used for the text in Inkscape, but the package 'color.sty' is not loaded}%
    \renewcommand\color[2][]{}%
  }%
  \providecommand\transparent[1]{%
    \errmessage{(Inkscape) Transparency is used (non-zero) for the text in Inkscape, but the package 'transparent.sty' is not loaded}%
    \renewcommand\transparent[1]{}%
  }%
  \providecommand\rotatebox[2]{#2}%
  \newcommand*\fsize{\dimexpr\f@size pt\relax}%
  \newcommand*\lineheight[1]{\fontsize{\fsize}{#1\fsize}\selectfont}%
  \ifx\svgwidth\undefined%
    \setlength{\unitlength}{44.19055529bp}%
    \ifx\svgscale\undefined%
      \relax%
    \else%
      \setlength{\unitlength}{\unitlength * \real{\svgscale}}%
    \fi%
  \else%
    \setlength{\unitlength}{\svgwidth}%
  \fi%
  \global\let\svgwidth\undefined%
  \global\let\svgscale\undefined%
  \makeatother%
  \begin{picture}(1,1.58390712)%
    \lineheight{1}%
    \setlength\tabcolsep{0pt}%
    \put(0,0){\includegraphics[width=\unitlength,page=1]{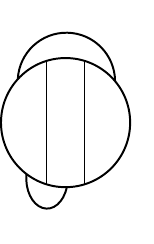}}%
    \put(0.14238561,0.68239837){\color[rgb]{0,0,0}\makebox(0,0)[lt]{\lineheight{12.375}\smash{\begin{tabular}[t]{l}$0$\end{tabular}}}}%
    \put(0.5836568,0.68239837){\color[rgb]{0,0,0}\makebox(0,0)[lt]{\lineheight{12.375}\smash{\begin{tabular}[t]{l}$1$\end{tabular}}}}%
    \put(0.4090877,1.40976782){\color[rgb]{0,0,0}\makebox(0,0)[lt]{\lineheight{12.375}\smash{\begin{tabular}[t]{l}$3$\end{tabular}}}}%
    \put(0.21209171,0.03988851){\color[rgb]{0,0,0}\makebox(0,0)[lt]{\lineheight{12.375}\smash{\begin{tabular}[t]{l}$2$\end{tabular}}}}%
  \end{picture}%
\endgroup%

\end{center}
Let $C^{*,*}(D;\bF_2)$ be the Khovanov complex with $\bF_2$ coefficients associated to this diagram, which is concentrated in degrees $-2,-1,0,1$ and $2$ because the number of negative crossings $n_-$ of this diagram is $2$. Figure \ref{figure:unlink} shows its cibe of resolutions. Let $C^{3,*}(D;\bF_2)$ be the subcomplex generated by the generators of quantum grading $1$, which is concentrated in homological degrees $-1,0,1,2$ where it attains the following ranks:
\begin{align*}
C^{1,2}(D;\bF_2) &\cong \bF_2& C^{1,1}(D;\bF_2) &\cong \bF_2^{8}& C^{1,0}(D;\bF_2) &\cong \bF_2^{12} & C^{1,-1}(D;\bF_2) &\cong \bF_2^{4}.
\end{align*}
The homology of this subcomplex is the Khovanov homology of the unlink in quantum grading $1$, which takes the following values:
\begin{align*}
\Kh^{1,i}(D;\bF_2) &\cong \begin{cases}
\bF_2 & \text{if $i=0$}\\
0 & \text{if $i\neq 0$}
\end{cases} 
\end{align*}
Consider now the Khovanov functor $\cF^{1}\colon \cube{4}\to \cB$ for $D$ in quantum grading $1$, and let $X_\bullet = \Lambda(\cF^{1})$ be the associated \sob (see Section \ref{ssection:lambda}). Let $C_*$ be the $\bF_2$-realisation of $X_\bullet$, which by construction is the one-fold desuspension of the chain complex $\Tot_{\bF_2}(\cF^{1})$, whose $n_-$-desuspension is the Khovanov complex. Since $n_-=2$, we have that the dual complex of $C_*$ is: 
\[C^* \cong \Sigma C^{1,*}(D;\bF_2).\]

Figure \ref{figure:unlink} shows the cube of resolutions of $D$. We order the circles at each vertex of the cube as follows: if the circles are not nested, order them left-to-right, if the circles are nested, order them outside-to-inside. This rule is well-defined except at the vertices $1001$ and $1100$, where we choose any order, as it will be irrelevant in the computation. 

\begin{center}
\begin{figure}
\footnotesize
\def\svgwidth{\columnwidth}
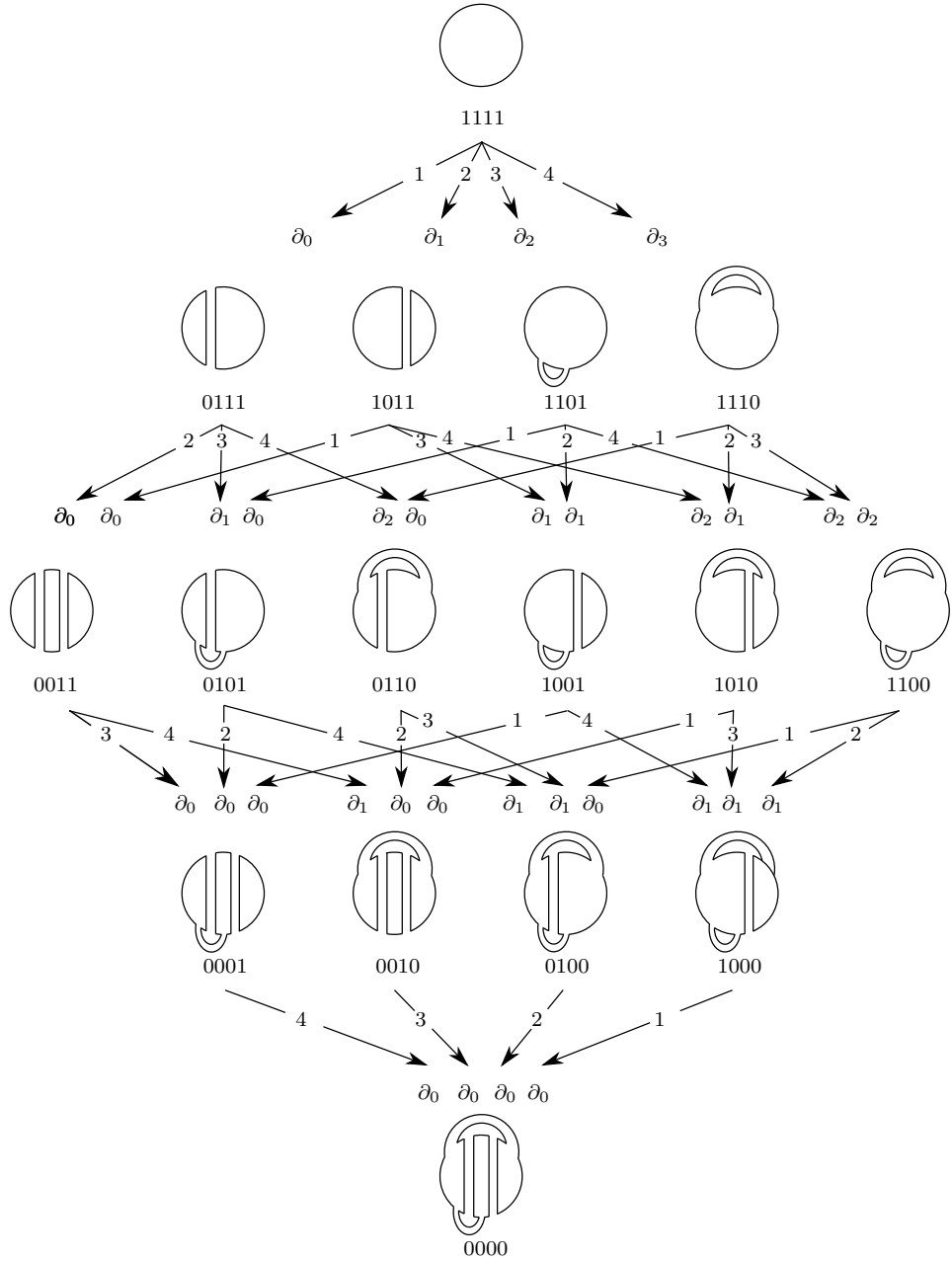
\caption{Cube of resolutions of $D$. We label each edge $u<v$ with the face map $\Lambda(u<v)$ and with a number $k$ that indicates which coordinate of the cube is changed along each edge. Note that we write the cube following the convention of \cite{LLS2015} and \cite{LLS-cube}, so the face maps go downwards instead of upwards.}
\label{figure:unlink}
\end{figure}

\end{center}

The generators of $C^{1,*}(D;\bF_2)$ are as follows:
\begin{itemize}
\item there is a single generator $\gen{x}$ in homological degree $2$ (semi-simplicial degree $3$) that enhances the only circle with $x_-$. 
\item Each vertex $u$ in homological degree $1$ (semi-simplicial degree $2$) contributes with two generators, each of them enhances one circle with $x_-$ and the other circle with $x_+$. For the $j$th vertex (read left to right, $j=1\ldots 4$), let $\gen{y}_{j,1}$ be the generator that enhances the first circle with $x_-$ and let $\gen{y}_{j,2}$ be the other generator.
\item Read left to right, the first, fourth and sixth vertices in homological degree $0$ (semi-simplicial degree $1$) contribute with three generators, each of them enhances one circle with $x_-$ and the remaining two circles with $x_+$. For the $j$th vertex ($j=1,4,6$), let $\gen{z}_{j,1}, \gen{z}_{j,2}$ and $\gen{z}_{j,3}$ be the generators that enhance the first, second and third circles with an $x_-$, respectively.
\item Read left to right, the second, third and fifth vertices in homological degree $0$ (semi-simplicial degree $1$) contribute with one generator, that enhances the only circle with $x_+$. For the $j$th vertex ($j=2,3,5$), let $\gen{z}_j$ denote that generator.
\item Each vertex in homological degree $-1$ (semi-simplicial degree $0$) contributes with a single generator, that labels every circle with $x_+$.
\end{itemize}
Let $\alpha = \gen{z}_{1,1}^* + \gen{z}_{1,3}^* + \gen{z}_{3}^* + \gen{z}_5^*\in C^1$, which is a cocycle because
\begin{align*}
\delta(\gen{z}_{1,1}^*) &= \gen{y}_{1,1}^* + \gen{y}_{2,1}^* & \delta(\gen{z}_{3}^*) &= \gen{y}_{1,1}^*+\gen{y}_{1,2}^* + \gen{y}_{4,1}^*+\gen{y}_{4,2}^*\\
\delta(\gen{z}_{1,3}^*) &= \gen{y}_{1,2}^* + \gen{y}_{2,2}^* & \delta(\gen{z}_{5}^*) &= \gen{y}_{2,1}^*+\gen{y}_{2,2}^* + \gen{y}_{4,1}^*+\gen{y}_{4,2}^*.
\end{align*}
Let us compute the second Steenrod square of the class $[\alpha]$ of degree $1$, using the formula
\[\Sq^2([\alpha]) = [\alpha\smile_{1-2}\alpha] = [\alpha\smile_{-1}\alpha].\]
The $\smile_{-1}$ product is dual to $\nabla_{-1}$, and we want to compute it on the generators of $C_*$ of semi-simplicial degree $3$ (which is where the second Steenrod square of $\alpha$ lives), so we have:
\[\nabla_{-1}|_{C_3} = \btoab{\bF_2}{\nabla^{\binom{3}{4}}\circ \Delta_3}.\]
and, by definition,
\[\nabla^{\binom{3}{4}} = \sum_{U\in \APower_4(3)}{\partial^3_{U^-}\wedge\partial^3_{U^+}}.\]
The target of this span is the union of $X_{-1}\times X_3$, $X_0\times X_2$, $X_1\times X_1$, $X_2\times X_0$ and $X_3\times X_{-1}$. We are interested only on its restriction to $X_1\times X_1$, so we only need to consider the summands indexed by the following $U\in \APower_4(3)$:
\begin{align*}
(0,0,1,1)&&
(0,0,2,2)&&
(0,0,3,3)&&
(1,1,2,2)&&
(1,1,3,3)\\
(2,2,3,3)&&
(0,0,1,3)&&
(0,1,1,3)&&
(0,2,2,3)&&
(0,2,3,3)
\end{align*}
which are characterised by the property that $|U^-| = |U^+| = 2$:
\begin{align*}
(0,0,1,1)^- &= (0,1) & (2,2,3,3)^- &= (2,3) \\
(0,0,1,1)^+ &= (0,1) & (2,2,3,3)^+ &= (2,3) \\
(0,0,2,2)^- &= (0,2) & (0,0,1,3)^- &= (0,1) \\
(0,0,2,2)^+ &= (0,2) & (0,0,1,3)^+ &= (0,3) \\
(0,0,3,3)^- &= (0,3) & (0,1,1,3)^- &= (0,1) \\
(0,0,3,3)^+ &= (0,3) & (0,1,1,3)^+ &= (1,3) \\
(1,1,2,2)^- &= (1,2) & (0,2,2,3)^- &= (0,2) \\
(1,1,2,2)^+ &= (1,2) & (0,2,2,3)^+ &= (2,3) \\
(1,1,3,3)^- &= (1,3) & (0,2,3,3)^- &= (0,3) \\
(1,1,3,3)^+ &= (1,3) & (0,2,3,3)^+ &= (2,3)
\end{align*}
Since the class $\alpha$ is supported at the first, third and fifth vertices, the span $\partial^3_{U^-}\wedge\partial^3_{U^+}$ is irrelevant unless both $U^-$ and $U^+$ belong to $\{(0,1),(0,3),(1,3)\}$. Therefore the only relevant $U$'s are
\begin{align*}
(0,0,1,1)&&
(0,0,3,3)&&
(1,1,3,3)&&
(0,0,1,3)&&
(0,1,1,3)&
\end{align*}
and $\partial_{U^-}$ and $\partial_{U^+}$ are one of the following three spans, where we write $\partial^k_{ij}$ for $\partial^k_{\{i,j\}}$:
\begin{align*}\xymatrixrowsep{0cm}\xymatrix{
&& \gen{z}_{1,1} \ar@{|->}[dl]\ar@{|->}[r] & \gen{z}_{1,1} \\
\partial^3_{01}&\gen{x}& \gen{z}_{1,2} \ar@{|->}[l]\ar@{|->}[r] & \gen{z}_{1,2} \\
&& \gen{z}_{1,3} \ar@{|->}[ul]\ar@{|->}[r] & \gen{z}_{1,3}
}
\\
\xymatrixrowsep{-.2cm}\xymatrixcolsep{1.06cm}\xymatrix{
&& a \ar@{|->}[dl] \ar@{|->}[dr] &  \\
\partial^3_{03}&\gen{x}& & \gen{z}_{3} \\
&& b \ar@{|->}[ul] \ar@{|->}[ur] & 
}
\\
\xymatrixrowsep{-.2cm}\xymatrixcolsep{1.06cm}\xymatrix{
&& c \ar@{|->}[dl] \ar@{|->}[dr] &  \\
\partial^3_{13}&\gen{x}& & \gen{z}_{5} \\
&& d \ar@{|->}[ul] \ar@{|->}[ur] & 
}
\end{align*}
We order the elements of these spans up-to-bottom, i.e.: $\z_{1,1}<\z_{1,2}<\z_{1,3}$ and $a<b$ and $c<d$. In order to understand the $2$-morphisms involved, we first make explicit the following spans. We also pick an ordering of each span, which is indicated on the right (recall that, at the end, the Steenrod squares will not depend on the chosen order). The orders indicated may be extended arbitrarily to the whole spans $\partial^2_0,\partial^2_1$ and $\partial^2_2$. Since the computations do not depend on this extended order, we do not indicate it, in order to keep the example short.
\begin{align*}\xymatrixrowsep{-.2cm}\xymatrixcolsep{1.1cm}\xymatrix{
&& \gen{y}_{1,1} \ar@{|->}[dl]\ar@{|->}[r] & \gen{y}_{1,1} & 2 \\
\partial^3_0 &\gen{x}&&  \\
&& \gen{y}_{1,2} \ar@{|->}[ul]\ar@{|->}[r] & \gen{y}_{1,2} & 1
}&
\\
\xymatrixrowsep{-.2cm}\xymatrixcolsep{1.1cm}\xymatrix{
&& \gen{y}_{2,1} \ar@{|->}[dl]\ar@{|->}[r] & \gen{y}_{2,1} & 1 \\
\partial^3_1  &\gen{x}& & \\
&& \gen{y}_{2,2} \ar@{|->}[ul]\ar@{|->}[r] & \gen{y}_{2,2} & 2
}&
\\
\xymatrixrowsep{-.2cm}\xymatrixcolsep{1.1cm}\xymatrix{
&& \gen{y}_{4,1} \ar@{|->}[dl]\ar@{|->}[r] & \gen{y}_{4,1}&2 \\
\partial^3_3 &\gen{x}&  &\\
&& \gen{y}_{4,2} \ar@{|->}[ul]\ar@{|->}[r] & \gen{y}_{4,2}&1
}&
\end{align*}

\begin{align*}
\xymatrixrowsep{-.2cm}\xymatrix{
&\gen{y}_{1,1}& \gen{z}_{1,1} \ar@{|->}[l]\ar@{|->}[r] & \gen{z}_{1,1}&1 \\
 \\
\partial^2_0|_{0111} &&\gen{z}_{1,2} \ar@{|->}[dl]\ar@{|->}[r]& \gen{z}_{1,2}&2  \\
&\gen{y}_{1,2}&&& \\
&& \gen{z}_{1,3} \ar@{|->}[ul]\ar@{|->}[r] & \gen{z}_{1,3}&3
}&
\\
\xymatrixrowsep{-.2cm}\xymatrix{
&& \gen{z}_{1,1} \ar@{|->}[dl]\ar@{|->}[r] & \gen{z}_{1,1} & 1 \\
 &\gen{y}_{2,1}&&& \\
\partial^2_0|_{1011} &&\gen{z}_{1,2} \ar@{|->}[ul]\ar@{|->}[r]& \gen{z}_{1,2}& 2 \\
&\gen{y}_{2,2}& \gen{z}_{1,3} \ar@{|->}[l]\ar@{|->}[r] & \gen{z}_{1,3}& 3
}&
\\
\xymatrixrowsep{-.2cm}\xymatrix{
&\gen{y}_{4,1}& \gen{y}_{4,1} \ar@{|->}[l]\ar@{|->}[dr] &&1  \\
\partial^2_0|_{1110} &&& \gen{z}_3  &\\
&\gen{y}_{4,2}& \gen{y}_{4,2} \ar@{|->}[l]\ar@{|->}[ur] && 2
}&
\\
\xymatrixrowsep{-.2cm}\xymatrix{
&\gen{y}_{1,1}& \gen{y}_{1,1} \ar@{|->}[l]\ar@{|->}[dr] &&1  \\
\partial^2_2|_{0111} &&& \gen{z}_3 & \\
&\gen{y}_{1,2}& \gen{y}_{1,2} \ar@{|->}[l]\ar@{|->}[ur] && 2
}&
\\
\xymatrixrowsep{-.2cm}\xymatrix{
&\gen{y}_{4,1}& \gen{y}_{4,1} \ar@{|->}[l]\ar@{|->}[dr] && 1 \\
\partial^2_1|_{1110} &&& \gen{z}_5 & \\
&\gen{y}_{4,2}& \gen{y}_{4,2} \ar@{|->}[l]\ar@{|->}[ur] && 2
}&
\\ 
\xymatrixrowsep{-.2cm}\xymatrix{
&\gen{y}_{2,1}& \gen{y}_{2,1} \ar@{|->}[l]\ar@{|->}[dr] &&1  \\
\partial^2_2|_{1011} &&& \gen{z}_5 & \\
&\gen{y}_{2,2}& \gen{y}_{2,2} \ar@{|->}[l]\ar@{|->}[ur] && 2
}&
\end{align*}
and the bijections relevant for the spans $\partial^3_{01},\partial^3_{03}$ and $\partial^3_{13}$ are (we omit the superscript $3$ on $\mu^3_{i,j}$):
\begin{align*}
\mu_{01}\colon \partial^3_{01}&\lra \partial^2_{0}\circ \partial^3_{0} & \mu_{10}\colon \partial^3_{01}&\lra \partial^2_{0}\circ \partial^3_{1}\\
\gen{z}_{1,1} &\longmapsto (\gen{y}_{1,1},\gen{z}_{1,1}) & \gen{z}_{1,1} &\longmapsto (\gen{y}_{2,1},\gen{z}_{1,1})\\
\gen{z}_{1,2} &\longmapsto (\gen{y}_{1,2},\gen{z}_{1,2}) & \gen{z}_{1,2}&\longmapsto (\gen{y}_{2,1},\gen{z}_{1,2})\\
\gen{z}_{1,3} &\longmapsto (\gen{y}_{1,2},\gen{z}_{1,3}) & \gen{z}_{1,3} &\longmapsto (\gen{y}_{2,2},\gen{z}_{1,3}) 
\end{align*}
\begin{align*}
\mu_{03}\colon \partial^3_{03}&\lra \partial^2_{2}\circ \partial^3_{0} & \mu_{30}\colon \partial^3_{03}&\lra \partial^2_{0}\circ \partial^3_{3} \\
a &\longmapsto (\gen{y}_{1,1},\gen{y}_{1,1}) & a &\longmapsto (\gen{y}_{4,2},\gen{y}_{4,2}) \\
b &\longmapsto (\gen{y}_{1,2},\gen{y}_{1,2}) & b &\longmapsto (\gen{y}_{4,1},\gen{y}_{4,1})
\end{align*}
\begin{align*}
\mu_{13}\colon \partial^3_{13}&\lra \partial^2_{2}\circ \partial^3_{1} & \mu_{31}\colon \partial^3_{13}&\lra \partial^2_{1}\circ \partial^3_{3} \\
c &\longmapsto (\gen{y}_{2,1},\gen{y}_{2,1}) & c &\longmapsto (\gen{y}_{4,2},\gen{y}_{4,2})\\
d &\longmapsto (\gen{y}_{2,2},\gen{y}_{2,2})& d &\longmapsto (\gen{y}_{4,1},\gen{y}_{4,1}).
\end{align*}
To determine $\mu_{03},\mu_{30},\mu_{13}$ and $\mu_{31}$ we have used the ladybug matching. Now,to compute $\partial_{U^-}\wedge \partial_{U^+}$, we need to understand $\cO_{s,t}(\twoU)^+$ for each $(s,t)\in \partial_{U^-}\times \partial_{U^+}$. As in this case $n+|U^-|+|U^+| = 3+2+2 = 7$, we have that an $(s,t)$-good pair $(W^\shortparallel,W^\smallwedge)$ is positive if
\begin{align*}\lambda_{W^{\shortparallel}}(s) &= \lambda_{W^\shortparallel}(t) & \lambda_{W^\shortparallel,W^{\smallwedge}}(s) &> \lambda_{W^\shortparallel,W^\smallwedge}(t).\end{align*}

If $\twoU = (u_1,u_2)$ has two elements, then each $(s,t)$ has exactly two maximal $(s,t)$-good chains of length $2$, which we call ``left'' and ``right'': the left maximal chain is either $(\emptyset,\{u_1\})\prec (\emptyset,\{u_1,u_2\})$ or $(\{u_1\},\{u_2\})\prec (\emptyset,\{u_1,u_2\})$, whereas the right maximal chain is either $(\emptyset,\{u_2\})\prec (\emptyset,\{u_1,u_2\})$ or $(\{u_2\},\{u_1\})\prec (\emptyset,\{u_1,u_2\})$. We will refer to any of these maximal chains by its first pair, as the second pair is always $(\emptyset,\{u_1,u_2\})$. To simplify the notation, we will often write $(W^\shortparallel,W^\smallwedge)$ instead of $(W^\shortparallel_1,W^\smallwedge_1)$.

If $\twoU=(u)$ has a single element, then each $(s,t)$ has a single maximal chain of length $1$, namely $\{(\emptyset,\{u\})\}$. 

Here are the spans $\partial^3_{U^-}\wedge\partial^3_{U^+}$ for the five relevant cases. 
\begin{itemize}
\item $U = (0,0,1,1)$. Then $\partial^3_{U^-} = \partial^3_{U^+} = \partial^3_{01}=\{\gen{z}_{1,1},\gen{z}_{1,2},\gen{z}_{1,3}\}$, so $\twoU$ has two elements. Now, if $(s,t)$ is any of the pairs $(\z_{1,1},\z_{1,2}), (\z_{1,1},\z_{1,3}), (\z_{1,2},\z_{1,3})$, then $\cO_{s,t}(\twoU)^+ = \emptyset$, because as $\z_{1,1}<\z_{1,2}$, $\z_{1,1}<\z_{1,3}$ and $\z_{1,2}<\z_{1,3}$ we have that the pair $(\emptyset,\{u_1,u_2\})$, present in all maximal chains, is not positive. On the other hand, if $(s,t)$ is any of the pairs $(\z_{1,2},\z_{1,1})$, $(\z_{1,3},\z_{1,1})$, $(\z_{1,3},\z_{1,2})$, the pair $(\emptyset,\{u_1,u_2\})$ is positive, and we have the following table indicating the positiveness of each maximal chain. The column ``span'' specifies the target of $\lambda_{W^{\shortparallel},W^\smallwedge}$ (its source is always $\partial^3_{01}$). To compute the values $\lambda_{W^{\shortparallel},W^\smallwedge}(s)$ and $\lambda_{W^{\shortparallel},W^\smallwedge}(t)$, one uses the bijections $\mu_{01}$ and $\mu_{10}$ above.

\[\begin{array}{|cccccc|}\hline
(s,t) & \text{maximal chain} & \lambda_{W^{\shortparallel},W^\smallwedge}(s) & \lambda_{W^{\shortparallel},W^\smallwedge}(t) & \text{span}&\text{positiveness}  \\ \hline
(\z_{1,2},\z_{1,1}) & (\emptyset, \{0\}) & \y_{1,2} & \y_{1,1} &\partial^3_0& - \\
(\z_{1,2},\z_{1,1}) & (\{1\}, \{0\}) & \z_{1,2} & \z_{1,1} &\partial^2_0|_{1011}& + \\
(\z_{1,3},\z_{1,1}) & (\emptyset, \{0\}) & \y_{1,2} & \y_{1,1} &\partial^3_0& - \\
(\z_{1,3},\z_{1,1}) & (\emptyset, \{1\}) & \y_{2,2} & \y_{2,1} &\partial^3_1& + \\
(\z_{1,3},\z_{1,2}) & (\{0\}, \{1\}) & \z_{1,3} & \z_{1,2} &\partial^2_1|_{0111}& + \\
(\z_{1,3},\z_{1,2}) & (\emptyset, \{1\}) & \y_{2,2} & \y_{2,1} &\partial^3_1& + \\ \hline
\end{array}\]
For example, the first row in the table computes the left maximal chain of the pair $s=\z_{1,2},t=\z_{1,1}$. The first step is to find $W_1^\shortparallel$, which will be empty if $\lambda_{\{0\}}(s)\neq \lambda_{\{0\}}(t)$ and equal to $\{0\}$ if $\lambda_{\{0\}}(s)= \lambda_{\{0\}}(t)$. The projection $\lambda_{\{0\}}$ is defined as the composition 
\[\lambda_{\{0\}}\colon \partial^3_{01}\overset{\mu_{01}}{\lra} \partial^2_0\circ\partial^3_0\lra \partial^3_0\] 
that sends $s$ and $t$ first to $(\y_{1,2},\z_{1,2})$ and $(\y_{1,1},\z_{1,1})$, and then to $\y_{1,2}$ and $\y_{1,1}$. Since $\y_{1,2}\neq \y_{1,1}$, we deduce that $W_1^\shortparallel=\emptyset$, hence $(W^\shortparallel_1,W^\smallwedge_1)=(\emptyset,\{0\})$. The second step is to compute the image of $s$ and $t$ under the projection 
\[\lambda_{\emptyset,\{0\}}\colon \partial^3_{01}\lra \partial^2_0\circ\partial^3_0\circ \partial^2_\emptyset\lra \partial^3_0.\]
As $\partial^n_\emptyset$ is the identity morphism, the first bijection $\mu_{\emptyset,\{0\},\{1\}}$ may be replaced by $\mu_{01}$, and therefore the value of $\lambda_{\emptyset,\{0\}}$ on $s$ and $t$ is 
\begin{align*}
s\longmapsto (\y_{1,2},\z_{1,2}) &\longmapsto \y_{1,2} & t&\longmapsto (\y_{1,1},\z_{1,1})\longmapsto \y_{1,1}.\end{align*}
For the fourth step one checks that $\y_{1,2}<\y_{1,1}$ in the span $\partial^3_0$, and therefore the pair is negative.

As another example, the second row in the table computes the right maximal chain of the pair $s=\z_{1,2},t=\z_{1,1}$. The first step is to find $W_1^\shortparallel$, which will be empty if $\lambda_{\{1\}}(s)\neq \lambda_{\{1\}}(t)$ and equal to $\{1\}$ if $\lambda_{\{1\}}(s)= \lambda_{\{1\}}(t)$. The projection $\lambda_{\{1\}}$ is defined as the composition 
\[\lambda_{\{1\}}\colon \partial^3_{01}\overset{\mu_{10}}{\lra} \partial^2_0\circ\partial^3_1\lra \partial^3_1\] 
that sends $s$ and $t$ first to $(\y_{2,1},\z_{1,2})$ and $(\y_{2,1},\z_{1,1})$, and then to $\y_{2,1}$ and $\y_{2,1}$. Since $\y_{2,1}= \y_{2,1}$, we deduce that $W_1^\shortparallel=\{1\}$, hence $(W^\shortparallel_1,W^\smallwedge_1)=(\{1\},\{0\})$. The second step is to compute the image of $s$ and $t$ under the projection 
\[\lambda_{\{1\},\{0\}}\colon \partial^3_{01}\lra \partial^2_\emptyset\circ\partial^2_0\circ\partial^3_1\lra \partial^2_0.\]
As $\partial^n_\emptyset$ is the identity morphism, the first bijection $\mu_{\{1\},\{0\},\emptyset}$ may be replaced by $\mu_{10}$. Additionally, the image of the bijection is contained in the subspan $\partial^2_{0}|_{1011}$ of $\partial^2_0$, and the value of $\lambda_{\{1\},\{0\}}$ on $s$ and $t$ is 
\begin{align*}
s&\longmapsto (\y_{2,1},\z_{1,2}) \longmapsto \z_{1,2}\in \partial^2_{0}|_{1011} \\ 
t&\longmapsto (\y_{2,1},\z_{1,1})\longmapsto \z_{1,1}\in \partial^2_{0}|_{1011}.\end{align*}
For the fourth step one checks that $\z_{1,2}>\z_{1,1}$ in the span $\partial^2_{0}|_{1011}$, and therefore the pair is positive.

Going back to the table, we deduce that for the pairs $(\z_{1,2},\z_{1,1})$ and $(\z_{1,3},\z_{1,1})$, the set $\cO_{s,t}(\twoU)^+$ has a single element, whereas for the pair $(\z_{1,3},\z_{1,2})$ the set $\cO_{s,t}(\twoU)^+$ has two elements. Therefore $\partial^3_{01}\wedge\partial^3_{01}$ is isomorphic to the following span:
\[\xymatrixrowsep{-.1cm}\xymatrix{
&\star\ar@{|->}[dl]\ar@{|->}[r]&(\z_{1,2},\z_{1,1}) \\
(\gen{x},\gen{x}) &\star \ar@{|->}[l]\ar@{|->}[r] & (\z_{1,3},\z_{1,1})\\
&\star,\star\ar@{|->}[ul]\ar@{|->}[r]&(\z_{1,3},\z_{1,2})
}\]

\item $U = (0,0,3,3)$. Then $\partial^3_{U^-} = \partial^3_{U^+} =\partial^3_{03}= \{a,b\}$. Now, since $b>a$, then $\cO_{s,t}(\twoU)^+ = \emptyset$ for $(s,t) = (a,b)$. On the other hand, if $(s,t) = (b,a)$ then
\[\begin{array}{|cccccc|}\hline
(s,t) & \text{maximal chain} & \lambda_{W^{\shortparallel},W^\smallwedge}(s) & \lambda_{W^{\shortparallel},W^\smallwedge}(t) &\text{span}& \text{positiveness}  \\ \hline
(b,a) & (\emptyset, \{0\}) & \y_{1,2} & \y_{1,1} &\partial^3_{0}& - \\
(b,a) & (\emptyset, \{3\}) & \y_{4,1} & \y_{4,2} &\partial^3_{3}& + \\ \hline
\end{array}\]
Therefore, $\cO_{b,a}(\twoU)^+$  
 has a single element, hence $\partial^3_{03}\wedge\partial^3_{03}$ is
 isomorphic to the following span:
\[\xymatrixrowsep{-.2cm}\xymatrix{
(\gen{x},\gen{x}) &\star \ar@{|->}[l]\ar@{|->}[r]&(\z_{3},\z_{3}) 
}\]
\item $U = (1,1,3,3)$. Then $\partial^3_{U^-} = \partial^3_{U^+} =\partial^3_{13}= \{c,d\}$. Now, since $d>c$, then $\cO_{s,t}(\twoU)^+ = \emptyset$ for $(s,t) = (c,d)$. On the other hand, if $(s,t) = (d,c)$ then
\[\begin{array}{|cccccc|}\hline
(s,t) & \text{maximal chain} & \lambda_{W^{\shortparallel},W^\smallwedge}(s) & \lambda_{W^{\shortparallel},W^\smallwedge}(t) &\text{span}& \text{positiveness}  \\ \hline
(d,c) & (\emptyset, \{1\}) & \y_{2,2} & \y_{2,1} &\partial^3_1& + \\
(d,c) & (\emptyset, \{3\}) & \y_{4,1} & \y_{4,2} &\partial^3_3& + \\ \hline
\end{array}\]
Therefore, $\cO_{d,c}(\twoU)^+$ has two elements, so $\partial^3_{13}\wedge\partial^3_{13}$ is isomorphic to the following span:
\[\xymatrixrowsep{-.2cm}\xymatrix{
(\gen{x},\gen{x}) &\star,\star\ar@{|->}[l]\ar@{|->}[r]& (\z_5,\z_5)\\
}\]
\item $U=(0,0,1,3)$. Then $\partial^3_{U^-} =\partial^3_{01}= \{\z_{1,1},\z_{1,2},\z_{1,3}\}$ and $\partial^3_{U^+} =\partial^3_{03}= \{a,b\}$ and we have that for each $(s,t)\in \partial^3_{U^-}\times \partial^3_{U^+}$, the positiveness of the only maximal chain is as follows:
\[\begin{array}{|cccccc|}\hline
(s,t) & \text{maximal chain} & \lambda_{W^{\shortparallel},W^\smallwedge}(s) & \lambda_{W^{\shortparallel},W^\smallwedge}(t) &\text{span}& \text{positiveness}  \\ \hline
(\z_{1,1},a) & (\emptyset, \{0\}) & \y_{1,1} & \y_{1,1} &\partial^3_0& \text{not $(s,t)$-good} \\
(\z_{1,2},a) & (\emptyset, \{0\}) & \y_{1,2} & \y_{1,1} &\partial^3_0& - \\
(\z_{1,3},a) & (\emptyset, \{0\}) & \y_{1,2} & \y_{1,1} &\partial^3_0& - \\ 
(\z_{1,1},b) & (\emptyset, \{0\}) & \y_{1,1} & \y_{1,2} &\partial^3_0& + \\
(\z_{1,2},b) & (\emptyset, \{0\}) & \y_{1,2} & \y_{1,2} &\partial^3_0& \text{not $(s,t)$-good} \\
(\z_{1,3},b) & (\emptyset, \{0\}) & \y_{1,2} & \y_{1,2} &\partial^3_0& \text{not $(s,t)$-good} \\ \hline
\end{array}\]
Therefore, $\cO_{\z_{1,1},b}(\twoU)^+$ has a single element, and every other $\cO_{s,t}(\twoU)^+$ is empty, so $\partial^3_{01}\wedge\partial^3_{03}$ is isomorphic to the following span:
\[\xymatrixrowsep{-.2cm}\xymatrix{
(\gen{x},\gen{x})&\star\ar@{|->}[l]\ar@{|->}[r]& (\z_{1,1},\z_3)}\]
\item $U=(0,1,1,3)$. Then $\partial^3_{U^-} =\partial^3_{01}= \{\z_{1,1},\z_{1,2},\z_{1,3}\}$ and $\partial^3_{U^+} = \partial^3_{13}=\{c,d\}$ and we have that for each $(s,t)\in \partial^3_{U^-}\times \partial^3_{U^+}$, the positiveness of the only maximal chain is as follows:
\[\begin{array}{|cccccc|}\hline
(s,t) & \text{maximal chain} & \lambda_{W^{\shortparallel},W^\smallwedge}(s) & \lambda_{W^{\shortparallel},W^\smallwedge}(t) &\text{span}& \text{positiveness}  \\ \hline
(\z_{1,1},c) & (\emptyset, \{1\}) & \y_{2,1} & \y_{2,1} &\partial^3_1& \text{not $(s,t)$-good} \\
(\z_{1,2},c) & (\emptyset, \{1\}) & \y_{2,1} & \y_{2,1} &\partial^3_1& \text{not $(s,t)$-good} \\
(\z_{1,3},c) & (\emptyset, \{1\}) & \y_{2,2} & \y_{2,1} &\partial^3_1& +\\
(\z_{1,1},d) & (\emptyset, \{1\}) & \y_{2,1} & \y_{2,2} &\partial^3_1& - \\
(\z_{1,2},d) & (\emptyset, \{1\}) & \y_{2,1} & \y_{2,2} &\partial^3_1& - \\
(\z_{1,3},d) & (\emptyset, \{1\}) & \y_{2,2} & \y_{2,2} &\partial^3_1& \text{not $(s,t)$-good} \\ \hline
\end{array}\]
Therefore, $\cO_{\z_{1,3},c}(\twoU)^+$ has a single element, and every other $\cO_{s,t}(\twoU)^+$ is empty, so $\partial^3_{01}\wedge\partial^3_{03}$ is isomorphic to the following span:
\[\xymatrixrowsep{-.2cm}\xymatrix{
(\gen{x},\gen{x}) &\star\ar@{|->}[l]\ar@{|->}[r]&(\z_{1,3},\z_{5}) 
}\]
\end{itemize}
As a consequence, we have that
\[
\begin{array}{|cccl|} \hline
U &&& \btoab{\bF_2}{\partial^3_{U^-}\wedge\partial^3_{U^+}\circ \Delta_3}(\gen{x}) \\ \hline
(0,0,1,1) &&& \z_{1,2}\otimes \z_{1,1} + \z_{1,3}\otimes \z_{1,1} + 2\cdot\z_{1,3}\otimes \z_{1,2} \\
(0,0,3,3) &&& 1\cdot \z_3\otimes \z_3 \\
(1,1,3,3) &&& 2\cdot\z_5\otimes \z_5 \\
(0,0,1,3) &&& \z_{1,1}\otimes \z_{3}  \\
(0,1,1,3) &&& \z_{1,3}\otimes \z_5\\ \hline
\end{array}
\]
Dualising:
\begin{align*}
\alpha\smile_{-1}\alpha =&\  \nabla_{-1}^*\left((\z^*_{1,1}+\z^*_{1,3} + \z^*_{3} + \z^*_{5})\otimes(\z^*_{1,1}+\z^*_{1,3} + \z^*_{3} + \z^*_{5})\right) \\
=&\ \nabla_{-1}^*(
\z^*_{1,1}\otimes \z^*_{1,1}+ \z^*_{1,1}\otimes \z^*_{1,3} + \z^*_{1,1}\otimes \z^*_{3} + \z^*_{1,1}\otimes \z^*_{5}
\\ &+
\z^*_{1,3}\otimes \z^*_{1,1}+ \z^*_{1,3}\otimes \z^*_{1,3} + \z^*_{1,3}\otimes \z^*_{3} + \z^*_{1,3}\otimes \z^*_{5}
\\ &+
\z^*_{3}\otimes \z^*_{1,1}+ \z^*_{3}\otimes \z^*_{1,3} + \z^*_{3}\otimes \z^*_{3} + \z^*_{3}\otimes \z^*_{5}
\\ &+
\z^*_{5}\otimes \z^*_{1,1}+ \z^*_{5}\otimes \z^*_{1,3} + \z^*_{5}\otimes \z^*_{3} + \z^*_{5}\otimes \z^*_{5}) \\ 
=&\ (0 + 0 + 1 + 0  
\\ &+
1 + 0 + 0 + 1 
\\ &+
0 + 0 + 1 + 0 
\\ &+
0 + 0 + 0 + 0)\cdot \gen{x}^* \\
=&\  0
\end{align*}
So we conclude that $\Sq^2([\alpha]) = [0]$.

\end{example}

\def\cprime{$'$}
\providecommand{\bysame}{\leavevmode\hbox to3em{\hrulefill}\thinspace}
\providecommand{\MR}{\relax\ifhmode\unskip\space\fi MR }
\providecommand{\MRhref}[2]{%
  \href{http://www.ams.org/mathscinet-getitem?mr=#1}{#2}
}
\providecommand{\href}[2]{#2}

\end{document}